\documentclass{article}

\usepackage[french, english]{babel}

\usepackage[letterpaper,top=2cm,bottom=2cm,left=3cm,right=3cm,marginparwidth=1.75cm]{geometry}

\usepackage{amsmath}
\usepackage{graphicx}
\usepackage[colorlinks=true, allcolors=blue]{hyperref}
\usepackage{xr}
\title{Asymptotic Expansion of passage probability for finite range random walk on Free Groups}
\author{Chevalier Guillaume}

\usepackage[T1]{fontenc} 
\usepackage[utf8]{inputenc} \usepackage{amsmath,amsthm,amsfonts,amscd,amssymb,graphics,color,xspace}  \usepackage{xr}
\usepackage{cite} 
\usepackage{mathabx} 
\usepackage{hyperref}
\usepackage{tikz-cd} 
\usepackage{tikz}	 %
\usepackage{enumitem}
\usepackage{imakeidx}
\usepackage{etoolbox}
\usepackage{caption}
\usepackage{subcaption}
\usepackage{multicol}
\usepackage{dsfont} 
\usepackage{esvect}
\usepackage{pdfpages}
\usepackage{graphics}
\usepackage{graphicx}
\usepackage{setspace}

\addto\captionsfrench{}

\usetikzlibrary{tikzmark,shapes.misc}

\usepackage{geometry}

\makeindex
\begin{document}
\newtheorem{Thm}{Theorem}[section]
\newtheorem{Prop}[Thm]{Proposition} 
\newtheorem{Lem}[Thm]{Lemma} 
\newtheorem{Cor}[Thm]{Corollary} 
\newtheorem{Clm}[Thm]{Claim} 
\newtheorem{prop}[Thm]{Property} 
\newtheorem*{Thm*}{Theorem} 
\newtheorem*{Cor*}{Corollary} 
\newtheorem*{Prop*}{Proposition}
\theoremstyle{definition} 
\newtheorem{Def}[Thm]{Definition} 
\theoremstyle{remark}
\newtheorem{Rem}[Thm]{Remark} 
\newtheorem{Notation}[Thm]{Notation}
\newtheorem{Ex}[Thm]{Example}  
\newtheorem*{Rem*}{Remark}
\newtheorem*{Ex*}{Example}
\numberwithin{equation}{section}

\newcounter{nmbrexercise}
\newenvironment{Exercise}[1][ ] 
{
  \par\vspace{\baselineskip}%
 {\refstepcounter{nmbrexercise} \color{PineGreen} \noindent \textbf{Exercise~\thenmbrexercise :} \textit{#1}}%
  \par\vspace{\baselineskip}%
}%
{\vspace{\baselineskip}}

\maketitle

\begin{abstract}
\begin{center}
     \og
 Given a finite-range random walk on a finitely generated free group , what is the asymptotic behaviour, as the number of steps goes to infinity, of the sequence of probabilities that the random walk is at a given element of the group?
 \fg{}
\end{center}
In this article, we answer this question by providing an asymptotic expansion to any order for the sequence of probabilities that a finite-range random walk on an infinite tree with bounded valence is at a given vertex (Theorem \ref{Asymptotics_of_probability_Green_General}) We will work under the assumption that each vertex of the tree has valence greater than or equal to three, and we will assume that the tree is endowed with a cofinite action of an automorphism group preserving the step distribution. The answer to the above question will appear as an immediate Corollary of the previously mentioned Theorem.

This article is part of a triptych with \cite{RTaub} and \cite{LCurve}. It relies on Tauberian results from \cite{RTaub}, used throughout the text to derive asymptotic expansions. The careful study of the objects introduced in this text is done in \cite{LCurve}.

\end{abstract}
\begin{center}
	\textbf{\Large Introduction}
\end{center}

\begin{center}
     \og
Given a finite-range random walk on a finitely generated free group , what is the asymptotic behaviour, as the number of steps goes to infinity, of the sequence of probabilities that the random walk is at a given element of the group?
 \fg{}
\end{center}

\vspace{\baselineskip}

 This question was addressed in 1971 by Peter Gerl \cite{Gerl_1971}, where he found in the case of nearest-neighbour random walk in the free group $\mathbb{F}_2$ on two generators (Introduced in example \ref{Example_intro_2}), a simple equivalent for the sequence of probabilities $\left(\mathbb{P}^x(Z_n = y)\right)_n$, that the random walk $(Z_n)_n$ based at some element $x$ the group $\mathbb{F}_2$, is at a given element $y$ of the group, at time $n\in\mathbb{N}$:
 \begin{equation*}
     \mathbb{P}^x(Z_{2n+r} = y) \sim_\infty CR^{-2n}\frac{1}{n^{3/2}},
 \end{equation*}
 with $C > 0$, $R > 1$, and $r \in \mathbb{Z}/2\mathbb{Z}$ such that $r = d(x, y) \,[2]$. Moreover, $\mathbb{P}^x(Z_{2n+1 - r} = y) \equiv 0$. In 1986, this result was extended by Peter Gerl and Wolfgang Woess \cite{Gerl-Woess_1986} to any nearest-neighbour random walk on a free group $\mathbb{F}_{q+1}$ on $(q+1)$ generators, for $q \geq 2$. Then, in 1993, Steven P. Lalley \cite{Lalley_1993} further extended this result to any finite range irreducible random walk in a free group, under the additional assumption that the random walk is aperiodic ($\gcd\{n : \mathbb{P}^x(Z_n=x) > 0\} = 1$), thus obtaining the asymptotic equivalent:
 \begin{equation*}
     \mathbb{P}^x(Z_{n} = y) \sim_\infty CR^{-n}\frac{1}{n^{3/2}},
 \end{equation*}
 with $C > 0$ and $R > 1$. See also \cite{Lalley_2001} and \cite{Nagnibeda-Woess_2002} for a slightly different presentation.

 In these three papers, the authors used \textbf{Darboux's Method} to obtain an asymptotic expansion of the above form. This method focuses on finding asymptotics of coefficients of a power series expansions of a functions $g$ with a finite number of singularities at its radius of convergence. It consists in finding a function $f$ whose power series expansion is known, such that the difference function $g-f$ is smoother than $g$ in a neighbourhood of the singularities (for example, of class $\mathcal{C}^k$ for some non-negative integer $k$). Then, by computing an estimate of the coefficients in the power series expansion of this latter difference, we obtain that the coefficients of Green's functions asymptotically behave as the coefficients of the power series expansion of $f$ (See \cite{Flajolet-Sedgewick_2009} Theorem VI.14). Thanks to the Newton-Puiseux theorem, this method is quite effective when dealing with algebraic functions $g$, which are functions satisfying an equation of the form $P(g(z),z)=0$, where $P$ is a fixed complex polynomial and $z$ is any complex number of small enough modulus.
 
 The functions $g$ considered by the authors were the generating functions of the sequence of probabilities $\left(\mathbb{P}^x(Z_n = y)\right)_n$, called \textbf{Green's functions}, that we usually denote by $G_z(x,y)=\sum_n \mathbb{P}^x(Z_n = y) z^n $. These function for irreducible finite range random walk in free groups, as you may have guessed, are algebraic functions. Actually Steven P.Lalley even proved:
 
\begin{Prop*}[Proposition 3.1 and Corollary 3.2 in \cite{Lalley_1993}, See also \cite{Kazuhiko_1984}]
Let $\mathbb{F}$ be a finitely generated free group and let $\mu\in\operatorname{Prob}(\mathbb{F})$ be a finitely supported probability measure on $\mathbb{F}$ whose support generates the group as a semi-group. Let $R>1$ be the radius of convergence of the Green's functions $G_z(x,y)$ (The radius $R$ does not depend on the pair $x,y\in\mathbb{F}$), and denote by $\mathbb{D}(0,R)=\{z:|z-0|<R\}$ the disk of radius $R$, centered at $0$ in $\mathbb{C}$. 

Then there exists an affine algebraic curve $\mathcal{C}=\{(z,J)\in\mathbb{C}\times \mathbb{C}^N: J=z\psi(J)\}$, with $N$ a positive integer and $\psi:\mathbb{C}^N\to\mathbb{C}^N$ a polynomial map with non-negative coefficients, and a parametrization $u:\mathbb{D}(0,R)\to\mathcal{C}$ such that $u(0)=0$ and for any two elements $x,y$ in $\mathbb{F}$, there exists a rational function $g_{x,y}\in\mathbb{C}(\mathcal{C})$ such that for any complex number in the disk $\mathbb{D}(0,R)$, we have
        $$G_z(x,y)=g_{x,y}(u(z)).$$
   Moreover the function field generated by the Green's function $(G_z(x,y))_{x,y\in\mathbb{F}}$ over $\mathbb{C}$ is a finite extension of $\mathbb{C}(z)$.
\end{Prop*}
We will call the above curve $\mathcal{C}$\footnote{To be more rigorous we should consider the Lalley's curve to be the irreducible variety of $\mathcal{C}$, which indeed is a curve (See \cite{Lalley_1993} Proposition 3.1). Besides this irreducible component coincide with $\mathcal{C}$ on the part that is of interest to us.}, the \textit{Lalley's curve}. A study of the singularities of this curve as well as a study of the rational functions $g_{x,y}$, enables to improve the previously stated asymptotics of probability:
\begin{Thm}\label{Thm_RW_on_FG}
Let $\mathbb{F}$ be a finitely generated free group, and let $\mu\in \operatorname{Prob}(\mathbb{F})$ be a finitely supported probability measure on $\mathbb{F}$ such that its support $\operatorname{Supp}(\mu)$ generates $\mathbb{F}$ as a semi-group. Then for any element $x$ of $\mathbb{F}$, there exist real constants $C>0, (c_l)_{l\geq 1}$ and $r\in\mathbb{Z}/d\mathbb{Z}$ depending on $x$ such that for any non-negative integer $K\in\mathbb{N}$, we have the asymptotic expansion, as $n$ goes to $\infty$:
$$\mathbb{\mu}^{\ast (dn+r)}(x)=C R^{-dn}\frac{1}{n^{3/2}}\left(1+\sum_{l=1}^{K-1}\frac{c_l}{n^l} +O \left(\frac{1}{n^K}\right)\right),$$
where $d=\gcd\{n : \mu^{\ast{n}}(e)>0\}$ is the return period of the random walk, with $e\in\mathbb{F}$ the neutral element of the group $\mathbb{F}$, and where $R>1$ is the radius of convergence of the power series $\sum_{n\geq 0} \mu^{\ast n}(e)z^n$.
\end{Thm}
This Theorem will appear as a Corollary of a Theorem concerning finite range random walks in infinite trees of bounded valence (Theorem \ref{Asymptotics_of_probability_Green_General} stated in page \pageref{Pour page ref}).

In the case where the support of $\mu$ contains a free system of generators and the neutral element (implying $d=1$) this result is a consequence of the work of Steven P. Lalley \cite{Lalley_1993} and of Theorem $VII.6$ in Subsection $VII.6.3$ of \cite{Flajolet-Sedgewick_2009}. Using different methods, S. Gouëzel, in \cite{Gouezel_2014}, proved the first-order asymptotic above in the case where the measure is symmetric, finitely supported, aperiodic ($d = 1$), and generates a Gromov-hyperbolic group as a semigroup (see also \cite{Gouezel-Lalley_2013}). Moreover, once strong uniform Ancona inequalities have been established for finite-range random walks on free groups —something that follows almost immediately from Lemma 2.7 in \cite{Gouezel_2014}— The method of S. Gouëzel allows to prove the same first-order asymptotic under the only assumption that the finitely supported measure is aperiodic ($d = 1$) and generates the free-group as a semigroup.

\vspace{\baselineskip}

The method we use to prove this theorem is found to be essentially equivalent to the one described in Subsection $VII.6.3$ of \cite{Flajolet-Sedgewick_2009}, regarding the family of generating functions satisfying a polynomial system of equations with non-negative coefficients. Especially Theorem $VII.6$ which compiles and mixes methods due to Michael Drmota, S. P. Lalley, and Alan R. Woods, see \cite{Drmota_1997}, \cite{Lalley_1993}, \cite{Lalley_2001} and \cite{Woods_1997}. 


The methods developed in this article, while essentially equivalent to those in the aforementioned texts, differ in the point of view adopted and with the definition given to the Lalley's curve. This different perspective and the preciseness of the definitions allows us to improve the previously known results and adapt it to the case of \textbf{periodic random walks} ($d>1$).

The key ingredient is the study the \textbf{dependency digraph}, associated with the system of polynomial equations defining the Lalley's curve. Which is, using the notation of the above theorem, the simple directed graph with set of vertices $\{1,...N\}$ and with an edge from a vertex $i$ to another vertex $j$, if the polynomial function $J\mapsto \psi_j(J)$ of the $j$-th coordinate of $\psi$, depends on the coordinate $J_i$. We prove that this graph has one \textbf{strong connected component}, which is a connected component of the digraph with the property that, for any vertex in the digraph there exists a path from this vertex to a vertex in the strong connected component. Alan R. Woods in \cite{Woods_1997} encountered the same structure while computing the limit distribution of the ways of colouring the root of a finite simple descending rooted tree, following some types of coloring rules, as the number of vertices goes to infinity.

This article is the central piece of a triptych with \cite{RTaub} and \cite{LCurve}. It contains the demonstration of the Theorem \ref{Asymptotics_of_probability_Green_General}, which generalizes Theorem \ref{Thm_RW_on_FG}, above. Its proof relies on Tauberian results exposed in \cite{RTaub}, and on a careful study of the Lalley's curve $\mathcal{C}$ and analytic properties of the functions $g_{x,y}$ introduced in Corollary \ref{Cor_algbraicity_of_Greens_function}.

\vspace{\baselineskip}

We will work under the following assumptions - that we will carefully define in the next section and will be regularly recalled along the text:

		Let $X=(X_0,X_1)$ be a tree of bounded valence with set of vertices $X_0$, and set of edges $X_1$. Let $\Gamma<\operatorname{Aut}(X)$ be a group of automorphisms of $X$, that acts cofinitely on $X_0$ (i.e the set of orbits $\Gamma\backslash X_0$ of $X_0$ under the action of $\Gamma$ is finite). Let $p:X_0\times X_0 \to [0,1]$ be a transition kernel that makes the Markov chain $(X_0,p)$ irreducible. 
	We will assume that there exists a non-negative integer $k$, such that for any vertices $x,y\in X_0$, if $d(x,y)>k$ then $p(x,y)=0$, such $p$ is said to have \textit{finite range}. Besides we assume that the transition kernel $p$ is $\Gamma$-invariant, that is to say, for any vertices $x,y$ of $X$, and any automorphism $g\in\Gamma$, we have: 
		$$ p(gx,gy)=p(x,y).$$
Lastly we suppose that each vertex of the tree $X$ has at least $3$ neighbours\footnote{This assumption is only used to prove Proposition \ref{LCurve-Prop_V_infini_is_absorbent} in \cite{LCurve} which is a key result in this triptych. Actually, it can be relaxed to weaker assumptions (for we only need a "visibility property", see Lemma \ref{LCurve-Lem_of_visible_edges} in \cite{LCurve}) but this may increase the technicality of the proof}.
		
		Under the above assumptions, we prove (Theorem \ref{Asymptotics_of_probability_Green_General}) that for any pair of vertices $(x,y)$ in the tree, there exist constants $C>0,\, (c_l)_{l\geq 1}$ such that we have the asymptotic expansion\footnote{This is the Poincaré notation for asymptotic expansion, see Notation \ref{Poincare_notation}.} 
		 \begin{equation}\label{Intro_eq_asymptotic}
		        p^{(dn+r(x,y))}(x,y) \sim_\infty C R^{-dn}\frac{1}{n^{3/2}}\left( 1 + \sum_{l=1}^\infty\frac{c_l}{n^{l}}\right),
        \end{equation}
		and $p^{(dn+t)}(x,y)=0$, if $t\neq r(x,y)\,[d]$, where $r:X_0\times X_0\to\mathbb{Z}/d\mathbb{Z}$ is the periodicity cocycle (Definition \ref{DEF:PERIODICITY_COCYCLE}).

	\begin{Ex}[	Irreducible Random Walk with finite support on a free group ]\label{Ex_ifrrw}\index{Random walk on free group}
	Let $\mathbb{F}$ be a finitely generated free group, and let $S$ be a finite free set of generators of $\mathbb{F}$, so that the Cayley graph of $\mathbb{F}$ associated with $S$ is a tree. We denote by $X$ the Cayley graph $Cay(\mathbb{F},S)$.
	Let $\mu$ be a probability measure on $\mathbb{F}$, whose support is finite and generates $F$ as a semi-group (we do not assume that the support of $\mu$ is a free set of generators).
	
	We consider the transition kernel $p:\mathbb{F}\times \mathbb{F}\to [0,1]$, associated with $\mu$, defined by $p(x,y)=\mu(x^{-1}y)$, for any $(x,y)\in \mathbb{F}\times \mathbb{F}$. Since $\operatorname{Supp}(\mu)$ generates $\mathbb{F}$ as semi-group, it can be shown in an elementary way that the Markov chain $(\mathbb{F},p)$ is irreducible. 
	
	 Lastly, if we pose $\Gamma=\mathbb{F}$ that acts by left multiplication on $\mathbb{F}$, then $\Gamma$ identifies as a group of automorphisms of $X$, that preserves the transition kernel $p$.
	Remark that $\mathbb{F}$ acts transitively on the vertices of $X$. 
	
	We can define the associated random walk $(Z_n)_n$, based at $x\in\mathbb{F}$ as the product:
	$$Z_n=x \xi_1\cdots \xi_n,$$
	where $(\xi_i)_{i\geq 1}$ is a family of independent and identically distributed random variables on $\mathbb{F}$, of law $\mu$.
	\end{Ex}
	
The Typical example that should be kept in mind is the following one
\begin{Ex}[Random walk on the free product $\left(\mathbb{Z}/2\mathbb{Z}\right)^{\ast 3}$]\label{Typical_Example}	
	Let $F:=\left(\mathbb{Z}/2\mathbb{Z}\right)^{\ast 3}$ be the free product of $3$ copies of $\mathbb{Z}/2\mathbb{Z}$, and $S=\{a,b,c\}$ be the natural generating set of $F$ i.e. consisting of the $3$ generators $a$, $b$ and $c$ respectively of the $3$ copies of $\mathbb{Z}/2\mathbb{Z}$. Let $\mu$ be a probability measure on $F$ with support equal to $S$. The Cayley graph associated with the pair $(F,S)$ is a regular tree of valence $3$. It is denoted by $\mathcal{T}=(X_0,X_1)$, and we denote by $p$ the transition kernel on $X_0$ associated with the random walk on $F$, whose step distribution is $\mu$. The group $\Gamma=F$ acts transitively on $X_0$ and the transition kernel $p:X_0\times X_0\to [0,1]$ is $\Gamma$-invariant. In the case where $\mu$ is the uniform distribution over $S$, Harry Kesten \cite{Kesten_1959b} gave a simple formula for the Green's functions $G_z(x,y)$.
\end{Ex}
\begin{figure}[h]
\centering\includegraphics[scale=0.2]{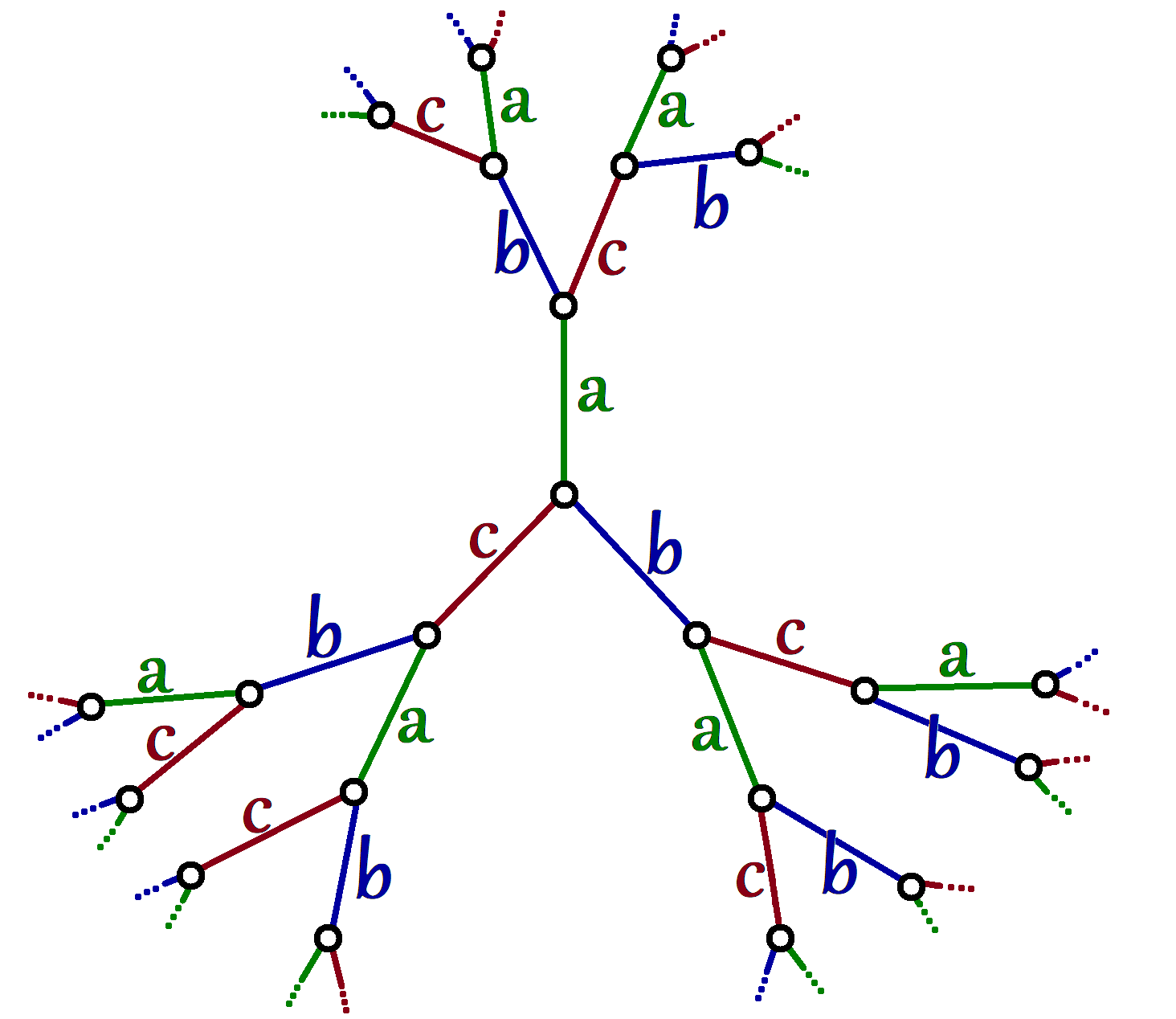}
		\caption{ The Cayley graph of the free product $\left(\mathbb{Z}/2\mathbb{Z}\right)^{\star 3}$ with three generators $a,b,c$.}\label{Figure_Free_labelled_tree_1}
\end{figure}

\section{Preliminaries: Definitions and examples}

	\begin{Def}\textit{Tree}\label{Def_tree}\index{Tree}
	
	A \textit{tree} is the data given by a pair $X=(X_0,X_1)$ where $X_0$ is a set called the \textit{set of vertices} of $X$ and $X_1$ is a subset of $X_0\times X_0\setminus \{(x,x):x\in X_0\}$ called the \textit{set of edges} of $X$ such that
	\begin{enumerate}[label=\roman*)]
	\item For all $(x,y)\in X_1$, $(y,x)$ is also in $X_1$ ($X_1$ is \textit{symmetric}).
	\item For any pair of vertices $(x,y)\in X_0\times X_0$, there exists a finite sequence of vertices $(x_0,...,x_n)$ such that $x_0=x$, $x_n=y$ and for any integer $i\in\{1,...,n\}$, $(x_{i-1},x_i)$ is an edge of $X$ (\textit{Connectedness}).
	\item There does not exist a finite sequence of distinct vertices $(x_0,...,x_n)$ in $X_0$ such that $n\geq 2$, for all integer $i\in\{1,...,n\}, (x_{i-1},x_i)\in X_1$ and $(x_n,x_0)$ is also in $X_1$ (\textit{Acyclicity}).
	\end{enumerate}
	
	The tree will be said to be an \textit{infinite tree} if its set of vertices $X_0$ is infinite.
	
	For $x$ a vertex of the tree $X$, we will say that $y$ is a \textit{neighbour} of $x$ if $(x,y)\in X_1$ is an edge of $X$. We will denote by $\mathcal{N}_x$ the set of neighbours of $x$ in $X$, so that $\{x\}\times \mathcal{N}_x$ is a subset of $X_1$.
	
For $x$ a vertex of $X$, we will name \textit{valence}\index{Valence} of the vertex $x$ the number of its neighbours $\operatorname{Card}(\mathcal{N}_x)$. And we will say that the tree has \textit{bounded valence}\index{Bounded valence} if the valence of its vertices is majored, that is $\max_{x\in X_0} \operatorname{Card}(\mathcal{N}_x)<\infty$.

\end{Def}
\begin{figure}	
\centering
\includegraphics[scale=0.3]{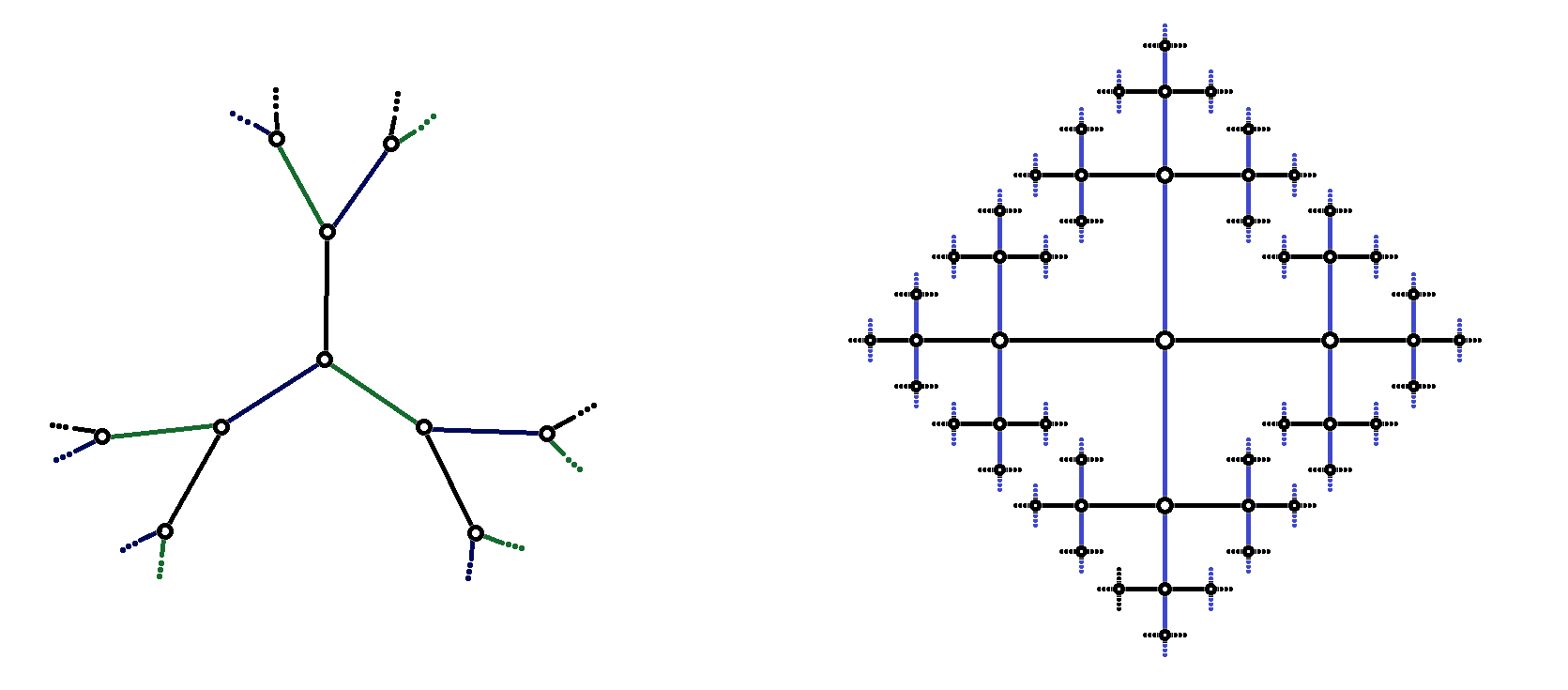}
		\caption{ The regular trees of valence $3$ and $4$.}\label{The_regular_trees}
\end{figure}
We give three examples of infinite trees with bounded valence.
\begin{Ex}[The regular tree $\mathcal{T}_{q+1}$ of valence $q+1$, with $q\in\mathbb{N}\setminus\{0\}$ (see figure \ref{The_regular_trees}).]\index{Regular tree}
Let $q\geq 1$ be an integer. The \textit{regular tree of valence }$q+1$ is a tree such that all its vertices have valence $q+1$. A model of the regular tree is given by the pair $(X_0,X_1)$, where $X_0$ is the set of finite words without double letters over the alphabet $S:=\{0,...,q\}$, denoted $\Sigma_S^{\tikzmarknode[strike out,draw]{1}{2}}:=\{(s_1,...,s_n): \forall i, s_i\in S, s_{i}\neq s_{i+1}\}$; and such that two words $(s_1,...,s_n),(s_1',...,s_m')\in\Sigma_S^{\tikzmarknode[strike out,draw]{1}{2}}$ are neighbours if, and only if $|n-m|=1$ and for any $i\leq \min(n,m)$, we have $s_i= s_i'$.

That is $X_0=\Sigma_S^{\tikzmarknode[strike out,draw]{1}{2}}$ and $$X_1=\{\left((s_1,...,s_n),(s_1',...,s_m')\right): |n-m|=1, \forall i\leq\min(n,m), s_i=s_i' \}.$$- 
\end{Ex}
\begin{Ex}[The Cayley graph of a group with respect to a free set of generators]
If $G$ is a finitely generated group and $S$ is a finite symmetric  $(S=S^{-1})$ set of generators, then one can define  \textit{the Cayley Graph} of $G$ with respect to $S$, as the graph whose set of vertices is $X_0=G$, and whose set of edges is $X_1:=\{(g,h)\in G\times G: \exists s\in S, gs=h\}$ (See figure \ref{Figure_Free_labelled_tree_1}).

When $G$ is a free group that is freely generated by $S$, then the Cayley graph $Cay(G,S)=(X_0,X_1)$ is a tree (\cite{Serre_1977} Proposition I.15).
\end{Ex}

\vspace{0.5\baselineskip}
\begin{Ex}[The $(n,m)$-regular tree]
Let $n,m\geq 2$ be two integers. The \textit{$(n,m)$-regular tree} is a tree such that its set of vertices can be divided into two disjoint sets $U$ and $V$ such that any vertex in $U$ possesses exactly $n$ neighbours, and all its neighbours are in $V$; and any vertex in $V$ possesses exactly $m$ neighbours, and all its neighbours are in $U$. 

Denote by $\Sigma_{\{1,...,m\}}$ the set of finite words over the alphabet $\{1,...,m\}$. The $(n,m)$-regular tree can be modelised by the pair $(X_0,X_1)$ where $X_0$ is the subset of $\Sigma_{\{0,...,n\}}^{\tikzmarknode[strike out,draw]{1}{2}}\times\Sigma_{\{1,...,m\}}$ given by 
$$ X_0=\left\lbrace 
(a_1,...,a_k),(b_1,...,b_l) :l=k \text{ or } l=k+1 
\right\rbrace;$$
and two vertices $((a_1,...,a_k),(b_1,...,b_l))$ and $((a_1',...,a_k'),(b_1',...,b_l'))$ are neighbours if, and only if
$|k-k'|+|l-l'|=1$, for all $i\leq\min(k,k')$, $a_i=a_i'$ and for all $j\leq\min(l,l')$, $b_j=b_j'$. See figure \ref{Fig:(n,m)regular_tree}
\end{Ex}

\begin{figure}		
        \centering
        \includegraphics[scale=0.2]{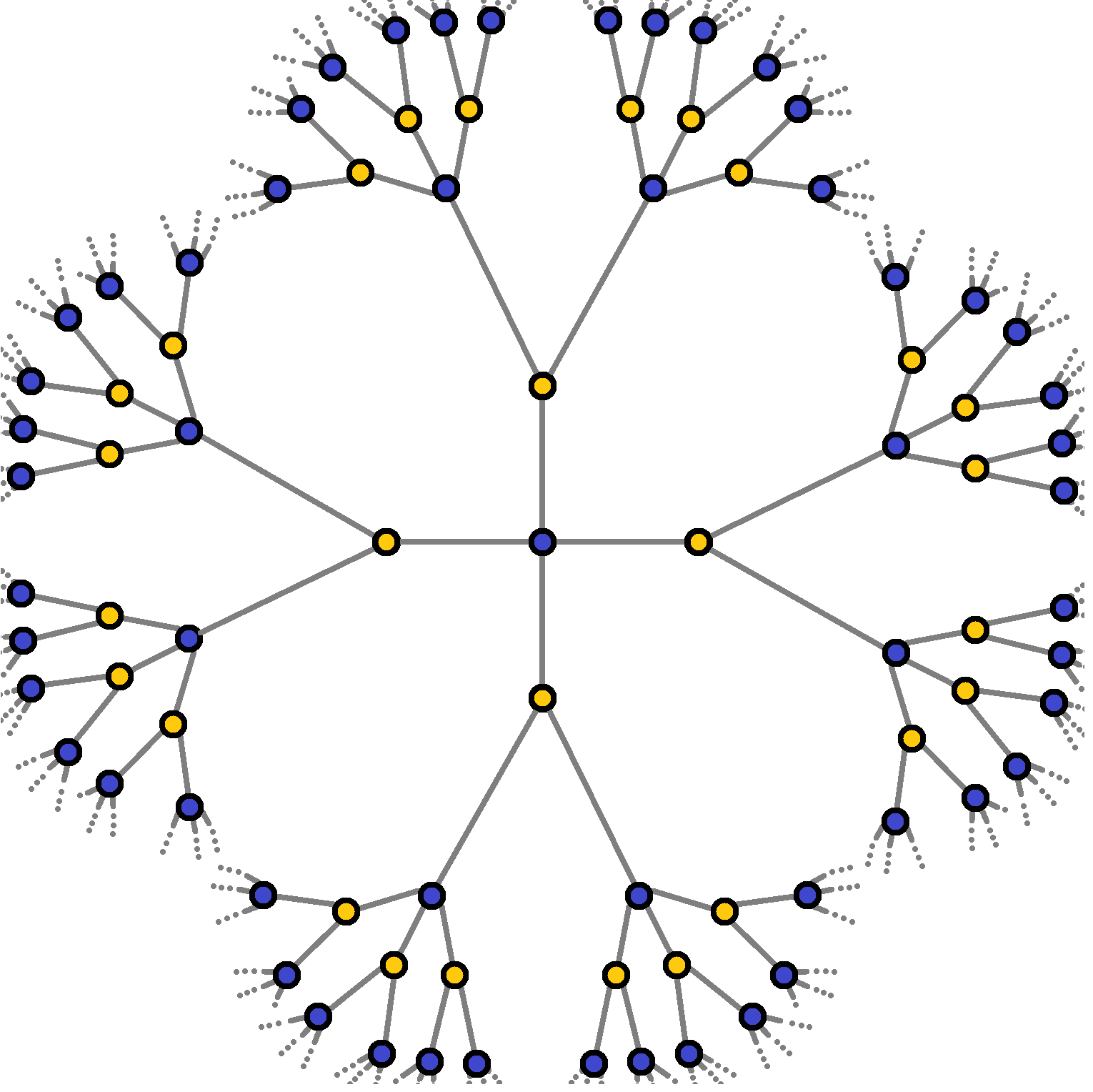}
		\caption{ The $(2,3)$-regular tree.}
\label{Fig:(n,m)regular_tree}
\end{figure}
\begin{Def}\textit{Geodesic segment}

Let $X=(X_0,X_1)$ be a tree, with set of vertices $X_0$, and set of edges $X_1$. A \textit{geodesic segment} in a tree is a finite sequence of \underline{distinct} vertices $(x_0,...,x_n)$ of $X$, such that for any integer $i\in\{1,...,n\}$, the vertex $x_{i-1}$ is a neighbour of $x_i$. By definition of a tree (Definition \ref{Def_tree}), given any two vertices $x,y\in X_0$ of the tree there is only one geodesic segment $(x_0,...,x_n)$ such that $x_0=x$ and $x_n=y$. We denote this geodesic segment by $[x,y]$. The integer $n$ will be denoted $d(x,y)=n$ and is called the \textit{combinatorial distance of $X$}\index{Combinatorial distance} between $x$ and $y$.
\end{Def}

In this text we will consider random walks on infinite tree with bounded valence. That is we will consider the movement of a particle initially set at a given vertex called \textit{based point} of the random walk, which jumps randomly from vertex to vertex following a prescribed law.

To properly define such a random walk, we will consider transition kernels and borrow vocabulary from discrete Markov chain theory\footnote{See \cite{Woess_2000} Chapter I, for more details on this topic}.

\begin{Def}\textit{Transition Kernel and Discrete Markov Chain on tree}\label{DEF:TRANSITION_KERNEL_DSCT_MKV_CHN}

Let $X_0$ be a set.	A function $p:X_0\times X_0\to [0,1]$ is called \textit{transition kernel}\index{Transition kernel} over $X_0$ if it satisfies:
	\begin{equation}
		\forall x \in X_0,\,\sum_{y\in X} p(x,y)=1.
	\end{equation}
	
	Given a transition kernel $p$ over $X_0$, for any non-negative integer $n\geq 0$, we can define another transition kernel $p^{(n)}:X_0\times X_0\to[0,1]$, that associates to any pair of vertices $(x,y)\in X_0\times X_0$, the value:
	\begin{equation}
	    p^{(n)}(x,y):=\sum_{\substack{\omega_0,...,\omega_n\in X_0 \\ \omega_0=x, \, \omega_n=y }} p(\omega_0,\omega_1)\cdots p(\omega_{n-1},\omega_n).
	\end{equation}

The pair $(X_0,p)$ is called a \textit{discrete Markov chain}\index{Markov Chain (discrete)} with \textit{state space} $X_0$. It will be said to be \textit{irreducible}\index{Irreducible (Markov chain)}, if for any pair $(x,y)\in X_0\times X_0$, there exists a non-negative integer $n$ such that $p^{(n)}(x,y)>0$.

	For an irreducible discrete Markov chain $(X_0,p)$ we define its \textit{period} $d$ as the greatest common divisor 
	\begin{equation}\index{Period}
	d:=\gcd\left\lbrace n : p^{(n)}(x,x)>0 \right\rbrace.
	\end{equation}
	If $d=1$, we say that the Markov chain $(X_0,p)$ is \textit{aperiodic}. This definition does not depend on any choice of $x$ as proven in Proposition \ref{period_is_well_defined} in the Appendix.
\end{Def}

For any pair of vertices $(x,y)\in X_0 \times X_0$, the value $p(x,y)$ is called \textit{probability of transition} from vertex $x$ to vertex $y$. Informally, if a particle is at vertex $x$ at a given time, then $p(x,y)$ gives the probability that the same particle is at vertex $y$ at the next time. The value $p^{(n)}(x,y)$ corresponds to the probability that a particle initially set in $x$, is in $y$ after exactly $n$ transitions.

Given a Markov chain $(X_0,p)$, we can associate a random walk $(Z_n)_{n\in\mathbb{N}}$. That is, a sequence of random variables with values in the set $X_0$, whose law is induced by the transition probabilities $(p(x,y))_{x,y\in X_0}$, via the set of relations - known as Markov relations,
\begin{equation}\label{Markov_relationship}\index{Markov relationship}\index{Markov relation}
\mathbb{P}(Z_{n+1}=y\mid Z_n=x)=p(x,y), \text{ when }\mathbb{P}(Z_n=x)\neq 0.
\end{equation}

To be more precise, for any Markov chain $(X_0,p)$, we provide $X_0$ with the discrete $\sigma$-algebra $\Lambda$ and with a probability measure $\theta$ on $(X_0,\Lambda)$. There exists a probability measure $\mathbb{P}^\theta$ on $X_0^\mathbb{N}$ provided with the product $\sigma$-algebra $\Lambda^{\otimes\mathbb{N}}$ such that for any integer $n\in\mathbb{N}$, if $Z_n:X_0^{\mathbb{N}}\to X_0$ denotes the map $(x_k)_{k\in\mathbb{N}}\mapsto x_n$ giving the $n$-th element of the sequence, then
\begin{gather}\label{equationpourlesmarchesaleatoires}
	\forall (x,y)\in X_0\times X_0,\, \mathbb{P}^\theta(Z_{n+1}=y \mid Z_n=x)=p(x,y),\\
	\text{when }\mathbb{P}^{\theta}(Z_n=x)>0, \text{ and }\nonumber\\
	\forall x \in X_0, \, \mathbb{P}^\theta(Z_{0}=x )=\theta(x).\nonumber
\end{gather}

In general we will take $\theta$ to be equal to $\delta_x$ the Dirac\index{Dirac measure} measure\footnote{$\delta_x:X_0\to\{0,1\}$ defined by $\delta_x(y)=0$ if $x\neq y$ and $\delta_x(x)=1$.} with support $\{x\}$. We will denote by $\mathbb{P}^x$ the associated probability measure on $X_0^\mathbb{N}$ instead of $\mathbb{P}^{\delta_x}$ and we will also say that the random walk $(Z_n)_n$ satisfying (\ref{equationpourlesmarchesaleatoires}) for $\theta=\delta_x$ is \textit{based at $x$} . We verify without difficulties that the probability $\mathbb{P}^x(Z_n=y)$ equals $p^{(n)}(x,y)$.

We give two examples of such random walk.
%
\begin{Ex}[Nearest-neighbour/Simple Random Walk on regular Tree]\label{Example_intro_1} 
Let $\mathcal{T}=(X_0,X_1)$ be a regular tree of valence $q+1\geq 2$ with set of vertices $X_0$, and set of edges $X_1$. Let $p:X_0\times X_0\to [0,1]$  be a transition kernel over $X_0$ with support $Supp(p)=\{(x,y): p(x,y)>0\}$ equal to the set of edges $X_1$.

A particle following the associated random walk will jump from a vertex $x$ to one of its neighbours, following the probability distribution $y\mapsto p(x,y)$ on the set $\mathcal{N}_x$ of neighbours of $x$. Such random walk is called a \textit{Nearest-neighbour random walk}\index{Random walk! Nearest-neighbour}. It is irreducible of period $2$.

It will be more specifically called \textit{Simple Random Walk}\index{Random walk! Simple}, if for any vertex $x\in X_0$ the probability distribution $y\mapsto p(x,y)$ is the uniform probability over $\mathcal{N}_x$, that is:
$$\forall (x,y)\in X_0\times X_0,\, p(x,y)=\frac{1}{q+1}\mathds{1}_{X_1}(x,y).$$
\end{Ex}
\begin{Ex}[Finite Range Random Walk on Free Group]\label{Example_intro_2} 
Let $G$ be a group that is freely generated by a finite symmetric set $S$. Let $\mu$ be a finitely supported probability measure over $G$ with support $Supp(\mu)=\{ x\in G :  \mu(x)>0 \}$ generating $G$ (not necessarily freely) as a semi-group.
Setting $p_\mu:G\times G\to[0,1]$ to be the transition kernel over $G$, $(x,y)\mapsto \mu(x^{-1} y)$, we get a discrete Markov chain $(G,p_\mu)$ that is irreducible.

This gives what we call a \textit{finite range random walk} on the tree $Cay(G,S)=(X_0,X_1)$. If $\mu$ has support $Supp(\mu)=S$, then we recognize a nearest-neighbour random walk on $G$ (See example \ref{Example_intro_1}). If furthermore $\mu$ is the uniform distribution over $S$ then we get the simple random walk on the $\operatorname{Card}(S)$-regular tree $Cay(G,S)$.
\end{Ex}
In the two previous examples the transition kernels $p$ satisfy two properties. The first one is that they have finite range:

We say that a transition kernel $p:X_0\times X_0\to[0,1]$ has \textit{finite range}\index{Finite range} if there exists a non-negative integer $k\geq 0$ such that 
\begin{equation}\label{Def_Finite_range_equation}
    \forall x,y\in X_0,\, ( d(x,y)> k \implies p(x,y)=0).
\end{equation}
Informally, the particle following a random walk $(Z_n)_n$ associated with $(X_0,p)$ cannot jump to vertices that are at combinatorial distance more than $k$ from its current position.

The second property satisfied is that they are invariant under a transitive action of a subgroup $\Gamma$ of the group of automorphisms the tree, defined below. I.e for any automorphism $g\in \Gamma $ and any pair of vertices $(x,y)\in X_0\times X_0$, we have
\begin{equation}
	p(gx,gy)=p(x,y).
\end{equation}

\begin{Def}\textit{Automorphism of a Tree}\index{Automorphism of a tree}

Let $X=(X_0,X_1)$ be a tree. An \textit{automorphism} of $X$ is a bijective map $g:X_0\to X_0$ that preserves edges, i.e
$$\forall (x,y)\in X_1,  (gx,gy)\in X_1.$$

We will denote by $\operatorname{Aut}(X)$ the set of all automorphism of $X$. This set endowed with the map composition is a group.

Let $\Gamma<\operatorname{Aut}(X)$ be a subgroup of $\operatorname{Aut}(X)$. We will say that $\Gamma$ \textit{acts cofinitely} on the set of vertices $X_0$, if it possesses only a finite number of orbits:
$$\operatorname{Card}(\Gamma\backslash X_0)<\infty.$$
\end{Def}
Note: In this text we will always consider \underline{closed} groups $\Gamma$ of automorphism, under the topology of the pointwise convergence. This is a technical assumption used in Lemma \ref{LCurve-Lem_of_visible_edges}) in \cite{LCurve}.
\begin{Rem}
Jacques Tits proved in \cite{Haefliger_1970} (p.188-211, Proposition 3.2 ) that there exists three types of automorphism in a tree:
\begin{Thm}[Tits 1970]
Let $X=(X_0,X_1)$ be a tree, and $g$ be an automorphism of the tree. Denote by $l:=\min_{x\in X_0}d(x,gx)$ the minimal distance $d(x,gx)$, where $x$ runs over $X_0$. Then one of the three following conditions is satisfied
\begin{enumerate}[label=\roman*)]
\item $g$ fixes at least one vertex of the tree;
\item $g$ permutes two neighbours;
\item There exists a sequence $(x_n)_{n\in\mathbb{Z}}$ of vertices of $X$ such that for every integer $n\in\mathbb{Z}$, $x_n\sim x_{n+1}$, on which $g$ acts by translation, that is for any integer $n$, $g x_n= x_{n+l}$, for some fixed integer $l\in\mathbb{Z}$ depending on $g$.
\end{enumerate}
\end{Thm}
\end{Rem}

\begin{Notation}\index{Asymptotic Expansion (Poincaré)}\label{Poincare_notation}
	We adopt the Poincaré notation for asymptotic expansions:
	
	Given a sequence $(u_n)_n$ that we are studying, $C>0$, $(c_l)_{l\geq 1}$ real constants, $R>0$, and $d$, $M$ two integers, instead of writing:
	
\textit{	 For any non-negative integer $K\geq 0$ and any $n\geq 1$, the sequence $(u_n)_n$ admits the asymptotic expansion}
\begin{equation*}
			u_n = C R^{-dn}\frac{1}{n^{M/2}}\left( 1 + \frac{c_1}{n^{1}}+...+\frac{c_K}{n^{K}}+ O\left(\frac{1}{n^{K+1}}\right)\right),
		\end{equation*}
	we will write for short:
	
	 \textit{$(u_n)_n$ admits the asymptotic expansion}
$$ u_n \sim_\infty C R^{-dn}\frac{1}{n^{M/2}}\left( 1 + \sum_{l=1}^\infty \frac{c_l}{n^{l}}\right).$$

\end{Notation}
\begin{Rem}
	As for Taylor Expansion theorem, we do not ask the series $\sum_{l\geq1} \frac{c_l}{n^{l}}$ to converge.
\end{Rem}

The theorem we prove in this text is the following one, concerning passage probability of irreducible finite range random walks on trees that admits a subgroup $\Gamma < \operatorname{Aut}(X)$, which acts cofinitely on $X$ and preserves the transition kernel:

\begin{Thm*}[Theorem \ref{Asymptotics_of_probability_Green_General}]\label{Pour page ref}

		Let $X=(X_0,X_1)$ be an infinite tree of bounded valence such that any vertex possesses at least three neighbours, let $\Gamma<\operatorname{Aut}(X)$ be a subgroup of automorphisms of the tree such that $\Gamma\backslash X_0$ is finite, let $p:X_0\times X_0\to [0,1]$ be a finite range, $\Gamma$-invariant transition kernel, making $(X_0,p)$ an irreducible Markov chain and let $k\in\mathbb{N}$ be an integer such that 
	$$\forall x,y\in X_0,\, (d(x,y)>k\Rightarrow p(x,y)=0).$$
		Denote by $d$ the period of $(X_0,p)$. And let $R>1$ be the radius of convergence of the Green's function defined below (See Definition \ref{DEF:GREEN_FUNCTION}).
		
		For all $x,y\in X_0$, there exist a non-negative integer $r\in\{0,...,d-1\}$ and constants $C>0,\, c_1, c_2,...$ such that we have the asymptotic expansion,
		$$ p^{(dn+r)}(x,y) \sim_\infty C R^{-dn}\frac{1}{n^{3/2}}\left( 1 + \sum_{l=1}^\infty\frac{c_l}{n^{l}}\right).$$
		And $p^{(dn+t)}(x,y)=0$, if $t\neq r\,[d]$.
\end{Thm*}
\begin{Rem}
The integer $r$ in the above asymptotic expansion is equal to $r(x,y)$ with $r:X_0\times X_0\to \mathbb{Z}/d\mathbb{Z}$ being the periodicity cocycle defined in definition \ref{DEF:PERIODICITY_COCYCLE}.
\end{Rem}

The strategy used in this text follows the same pattern as in most articles on the subject, in big lines: We study the Green's functions $G_z(x,y)$ which are the generating function of the sequence of probabilities $\left(\mathbb{P}^x(Z_n=y)\right)_n$ (See definition \ref{DEF:GREEN_FUNCTION} below and see Appendix \ref{Appendix:Greens_functions} for general properties on Green's functions) then, we analyse the behaviour of the functions near their singularities at their radius of convergence. 

\begin{Def}\textit{Green's functions}\label{DEF:GREEN_FUNCTION}

	Let $(X_0,p)$ be a discrete Markov chain. For a pair of vertices $(x,y)\in X_0 \times X_0$, the Green's function at $(x,y)$, associated with the Markov chain $(X_0,p)$ is the germ of holomorphic function in the neighbourhood of $z=0$ in $\mathbb{C}$, denoted by $z\mapsto G_z(x,y)$, given by the power series:
	\begin{equation}
		G_z(x,y)=\sum_{n\in\mathbb{N}} p^{(n)}(x,y)z^n.
	\end{equation}
	In other words, it is the generating function associated with the sequence of probabilities $\left(p^{(n)}(x,y)\right)_n$. Using the Cauchy-Hadamard formula\footnote{See\cite{Queffelec_2017} formula (19) p.54.} we get that the radius of convergence of the Green's function $z\mapsto G_z(x,y)$ is
\end{Def}
\begin{equation}\label{Cauchy_Hadamard_CV_Green}
    \frac{1}{R}=\limsup_n (p^{(n)}(x,y))^{1/n}. 
\end{equation}
$R$ does not depend on the pair $(x,y)\in X_0\times X_0$, thanks to the irreducibility assumption (See \ref{Prop_Markov_Ired_et_rayon_de_cv} in the Appendix).

\begin{Notation}\label{no_name_label}\index{$\mathbb{D}(0,R)=\{z\in\mathbb{C}: |z|\leq R\}$}
	For $z_0\in\mathbb{C}$ and $r\in ]0,\infty]$, we will denote $\mathbb{D}(z_0,r)$ (respectively $\overline{D}(z_0,r)$) the open (respectively closed) disk of center $z_0$ and radius $r$ in $\mathbb{C}$ that is to say:
	$$\mathbb{D}(z_0,r):\{z\in\mathbb{C}: |z-z_0|<r\},$$
	$$\overline{\mathbb{D}}(z_0,r):\{z\in\mathbb{C}: |z-z_0|\leq r\}.$$
\end{Notation}

\underline{Sketch of the method used to prove Theorem \ref{Asymptotics_of_probability_Green_General}}:

Following the work of Steven P. Lalley for computing Green's function of irreducible finite-range random walks on free groups, we construct an algebraic curve $\mathcal{C}=\{(z,J)\in\mathbb{C}\times\mathbb{C}^N : J=z\psi(J)\}$ for some non-negative integer $N$, and some polynomial map $\psi:\mathbb{C}^N\to\mathbb{C}^N$, that contains the origin $(0,0_N)\in\mathbb{C}\times\mathbb{C}^N$ and is a smooth complex curve in a neighbourhood of the origin. Such algebraic curve always comes with a projection $\lambda:(z,J)\mapsto z$ on the first coordinates, that is actually a ramified covering $\mathcal{C}\to\mathbb{C}$ (See \cite{Farkas-Kra_1992} Chapter I). 

 Besides the above curve $\mathcal{C}$ will satisfy that for any pair $(x,y)$ of vertices of the tree, the Green's function $G_z(x,y)$ factors through a section $u:\mathbb{D}(0,R)\to \mathcal{C}$ of $\lambda$ (i.e $\lambda\circ u = Id_{\mathbb{D}(0,R)}$) in a neighbourhood of the origin, with a rational function $g_{x,y}\in\mathbb{C}(\mathcal{C})$ over $\mathcal{C}$, that is regular in a neighbourhood of $\overline{\{u(z):z\in\mathbb{D}(0,R)\}}$.
 
$$G_z(x,y)=g_{x,y}\circ u(z).$$

A key argument (in particular to justify the compactness of the set $\overline{\{u(z):z\in\mathbb{D}(0,R)\}}$) used in this text is the finiteness of the Green's function at the radius of convergence:

\begin{Thm}\label{Finitude_of_Green_functions}(See \cite{Woess_2000} Chapter 2.)~
			
			For any finitely supported irreducible Markov chain $(\mathcal{X}_0,p)$ on an infinite tree such that all of its vertices has at least three neighbours, and such that there exists a subgroup of automorphisms $\Gamma$ of the tree that preserves $p$ and acts cofinitely on $\mathcal{X}_0$, if $\mathcal{R}=\left(\displaystyle\limsup_{n\to\infty} p^{(n)}(x,y)^{\frac{1}{n}}\right)^{-1}$ is the radius of convergence of Green's functions $z\mapsto G_z(x,y)$, then
			$$\forall x,y\in X_0, \; G_{\mathcal{R}}(x,y)<\infty$$
\end{Thm}
\begin{Rem}
This theorem is actually true for any finite range random walk on infinite non-amenable graph. And the same applies to first-passage generating functions (Definition \ref{Def_Green_function_first_pass}) $F_{R}(x,y)$ and to restricted Green's functions (Definition \ref{Def_Green_Restricted}) $G_R(x,y;\Omega)$, for $\Omega$ a subset of $X_0$.

This result is a slight generalisation of Theorem II.12.5 in \cite{Woess_2000} combined with the method described in Proposition II.7.4 in \cite{Woess_2000}. See also Proposition 2.1 in \cite{Lalley_1993}).
\end{Rem}

Studying the ramification properties of $\mathcal{C}$ above $\mathbb{C}$, with respect to the map $\lambda:(z,J)\mapsto z$, we can give a precise description of the singularities of its sections.
In our case we prove that $\mathcal{C}$ ramifies at the $d$ points $u(R)$, $u(R e^{2i\pi/d})$,..., $u(R e^{(d-1)2i\pi/d})$ in $\mathcal{C}$ with ramifications of degree $2$. Proving this is a central part to generalise the work of Steven Lalley and it is carefully done in section \ref{LCurve-Section Structure of C} in \cite{LCurve}.
Then applying Tauberian theorems from \cite{RTaub}, we obtain the desired asymptotics, the details are given in section \ref{Section_Asymptotics}.

\section{Admissible Path and Green's functions}
	
In what follow we consider a family of restricted Green's functions. At this point we invite the reader to read subsection \ref{Appendix:Greens_functions} in the Appendix, giving general results about Green's functions. The presentation that follows is borrowed from \cite{Gouezel-Lalley_2013}.

\begin{Def}\textit{Admissible path} \index{Admissible path}\index{$\mathcal{P}h(x,y)$}

	Let $X_0$ be a set, and let $(X_0,p)$ be a discrete Markov chain. A $p$-\textit{admissible path} in $X_0$ is a tuple $\gamma=(\omega_0,...\omega_n)$ of elements in $X_0$, such that for any $i\in\{1,...,n\}$, $p(\omega_{i-1},\omega_i)>0$. The integer $n$ is called the \textit{length} of the $p$-admissible path $\gamma$, and is denoted $l(\gamma)$. The element $\omega_0$ will be called \textit{the source} of the path, and $\omega_n$ \textit{the aim} of the path.

	Denote by $\mathcal{P}h^{(p)}(X_0)$ (or simply $\mathcal{P}h(X_0)$ if there is no ambiguity) the set of $p$-admissible paths in $X_0$ with respect to the transition kernel $p:X_0\times X_0\to[0,1]$, and by $\mathcal{P}h^{(p)}(x,y)$ (or $\mathcal{P}h(x,y)$), the set of $p$-admissible paths in $X_0$, with source $x$ and aim $y$: 
	$$\mathcal{P}h(x,y):=\Big\{(\omega_0,...\omega_n)\text{ : }n\in\mathbb{N}, \, \forall 1\leq i \leq n,\,  p(\omega_{i-1},\omega_{i})>0,\, \omega_0=x \text{ and } \omega_n=y \Big\}.$$
	
	And for $n$ a non-negative integer, we denote by $\mathcal{P}h^{(p)}_n(x,y)$ (or simply $\mathcal{P}h_n(x,y)$), the set of $p$-admissible paths in $X_0$ with source $x$, aim $y$ and with length equals to $n$.
\end{Def}

\begin{Rem}
	A $p$-admissible path of length zero identifies as an element $\omega_0$ in $X_0$.
\end{Rem}

\begin{Def}\textit{Concatenation of admissible paths}\index{Concatenation of admissible paths}

	Let $(X_0,p)$ be a Markov chain. Let $\gamma_1 =(\omega_0,... , \omega_{n_1})$ and $\gamma_2=(\omega_0',...,\omega_{n_2}')$ be two $p$-admissible paths in $X_0$. If $\omega_{n_1}=\omega_0'$ then we can define the \textit{concatenation of the $p$-admissible paths} $\gamma_1$ and $\gamma_2$, denoted $\gamma_1\ast\gamma_2$, where $\gamma_1\ast\gamma_2$ is the $p$-admissible path $(\omega_0,...,\omega_{n_1},\omega_1',...,\omega_{n_2}')$ of length $l(\gamma_1\ast\gamma_2)=l(\gamma_1)+l(\gamma_2)$.
\end{Def}
\begin{Rem} 
	Let $(X_0,p)$ be a discrete Markov chain, $(Z_n)_n$ be an associated random walk and let $\gamma=(\omega_0,...,\omega_n)\in\mathcal{P}h(x,y)$ be a $p$-admissible path. The probability that the random walk based at $\omega_0$ follows the $p$-admissible path $\gamma$ is given by: $$\mathbb{P}^{\omega_0}(Z_0=\omega_0,Z_1=\omega_1,...,Z_n=\omega_n)=p(\omega_0,\omega_1)\cdot\cdot\cdot p(\omega_{n-1},\omega_n).$$
	This follows from Markov relations (\ref{Markov_relationship}):
	\begin{align*}
		\mathbb{P}^{\omega_0}(Z_0=\omega_0,Z_1=\omega_1,..., Z_m=\omega_m)
		&=\prod_{j=0}^{n-1}\mathbb{P}^{\omega_0}(Z_{j+1}=\omega_{j+1}|Z_j=\omega_j)\\
		&=\prod_{j=0}^{n-1}p(\omega_j,\omega_{j+1}).
	\end{align*}
\end{Rem}

It justifies the following definition (introduced by S. Gouëzel and S. P. Lalley in \cite{Gouezel-Lalley_2013} ):

\begin{Def}\label{Def_function_path_weight}\textit{The path weight function $w_r$}\index{$w_r$}

	Let $(X_0,p)$ be a discrete Markov chain.
	For any real number $r\geq 0$, we define the function $w_r:\mathcal{P}h(X_0)\rightarrow [0,\infty[$ which, for any $p$-admissible path $\gamma =(\omega_1,...\omega_n)$ associates the value:
	$$w_r(\gamma):= r^n\left(\prod_{i=0}^{n-1}p(\omega_i,\omega_{i+1})\right)$$
 	This function is multiplicative with respect to the concatenation of $p$-admissible paths ($w_r(\gamma_1\ast\gamma_2)=w_r(\gamma_1)w_r(\gamma_2)$) and we call the function $r\mapsto w_r(\gamma)$, \textit{the path weight of} $\gamma$.
\end{Def}

\begin{Prop}[Gouezel-Lalley's path weight formula]\label{Formula_sum_weight}~

	Let $(X_0,p)$ be a discrete Markov chain. Green's functions verify for all real number $r\geq 0$, and for any $x,y$ in $X_0$:
	\begin{equation}\label{SumWeights}
		 G_r(x,y)=\sum_{\gamma\in\mathcal{P}h(x,y)} w_r(\gamma).
	\end{equation}
\end{Prop}
\begin{proof}
    By definition 
        \begin{align*}
		G_r(x,y)&=\sum_{n=0}^\infty p^{(n)}(x,y)r^n\\
		&=\sum_n\left(\sum_{\substack{(\omega_0,...\omega_{n})\in (X_0)^{n+1}\\\omega_0=x,\, \omega_n=y}}p(\omega_0,\omega_1)\cdots p(\omega_{n-1},\omega_n)r^n\right)\\
		&=\sum_n\left(\sum_{\gamma\in\mathcal{P}h_n(x,y)}w_r(\gamma)\right).
	\end{align*}
Where in the last equality we used the definition of the path-weight function and the fact that for any finite sequence $(\omega_0,...,\omega_n)\in (X_0)^{n+1}$, we have $p(\omega_0,\omega_1)\cdots p(\omega_{n-1},\omega_n)$ is non-zero if, and only if, $(\omega_0,...,\omega_n)$ is a $p$-admissible path.
\end{proof}

\begin{Rem}\label{Rem_sum_weight}

	All the equalities in this proof are legitimized by the non-negativity of the coefficients of the series (and by Fubini-Tonelli's theorem, for example). Actually, for $r\geq 0$, if the sum in the equation (\ref{SumWeights}) is finite, then this equation remains valid for any complex number $z$ instead of $r$, of modulus $|z|\leq r$.
\end{Rem}

As a Corollary of Proposition \ref{Existence_of_a_period_cocycle} in the Appendix, we have:
	\begin{Cor}\label{Cor_Admissible_Path_Lenght}
		Let $(X_0,p)$ be an irreducible discrete Markov chain with period $d\geq 1$, and denote by $r:X_0\times X_0\to\mathbb{Z}/d\mathbb{Z}$, the associated periodicity cocycle (Definition \ref{DEF:PERIODICITY_COCYCLE}).
		
		For any pair $(x,y)\in X_0\times X_0$, and for any $p$-admissible path $\gamma$ in $\mathcal{P}h(x,y)$, we have
		$$l(\gamma)\equiv r(x,y) \, [d].$$	
	\end{Cor}
		\begin{Cor}\label{Cor_operateur_diagonal}
 		Let $(X_0,p)$ be an irreducible discrete Markov chain and let $r$ be the periodicity cocycle from definition \ref{DEF:PERIODICITY_COCYCLE}. If we denote by $\zeta_d$ the primitive square root of unity, $\zeta_d=\exp(\frac{2i\pi}{d})$, then for any vertices $a,b\in X_0$ and any $p$-admissible path $\gamma\in\mathcal{P}h(a,b)$, we have
 		\begin{equation}\label{eq-t21}
     		\forall z\in\mathbb{C},\, w_{\zeta_d z}(\gamma)=\zeta_d^{r(a,b)} w_z(\gamma),
 		\end{equation}
 		and
 		\begin{equation}\label{eq-t22}
 		\forall z \in \mathbb{D}(0,R),\, G_{\zeta_d z}(a,b)=\zeta_d^{r(a,b)} G_z(a,b).
 		\end{equation}
 	\end{Cor}
 	\begin{proof}
 		Equality (\ref{eq-t21}) is an immediate consequence of the definition of the path weight function $w_z$ and of the Corollary \ref{Cor_Admissible_Path_Lenght}.
 		
 		The equality (\ref{eq-t22}) is a consequence of (\ref{eq-t21}), and of formula (\ref{SumWeights}) giving the Green's function $G_z(a,b)$ as a sum of path weight.
 	\end{proof}

\vspace{\baselineskip}

\begin{Def}\label{Restricted_admissible_paths}\textit{Restricted admissible paths}

	Let $(X_0,p)$ be a discrete Markov chain and $\Omega$ be a subset of $X_0$. We denote by $\mathcal{P}h(x,y;\Omega)$, the set of $p$-admissible paths with source $x$ and aim $y$ for which all intermediates vertices belong to $\Omega$ ( i.e $\gamma=(\omega_0,\omega_1,...\omega_n)\in\mathcal{P}h(x,y)$ such that $n\in\mathbb{N}$ and $\forall 1\leq i\leq n-1,\,\omega_i\in\Omega).$ 
\end{Def}

\begin{Prop}
	Let $(X_0,p)$ be a discrete Markov chain and $\Omega$ be a subset of $X_0$. The Green's functions restricted to $\Omega$ verify for all $r\geq 0$:
	\begin{equation}\label{SumRestrictedWeight}
		\forall x,y\in X_0,\; G_r(x,y;\Omega)=\sum_{\gamma\in\mathcal{P}h(x,y;\Omega)} w_r(\gamma)
	\end{equation}
	with $w_r$ the weight function (Definition \ref{Def_function_path_weight}).
\end{Prop}
\begin{proof}
 Recall that for any $x,y\in X_0$, we have
	    $$ p^{(n)}(x,y;\Omega):=\sum_{\omega_0,...,\omega_n} p(\omega_0,\omega_1)\cdots p(\omega_{n-1},\omega_n),$$
	where the sum is taken over all $(n+1)$-tuple $(\omega_0,...,\omega_n)$ of elements in $X_0$ such that $\omega_0=x$, $\omega_n=y$ and for any $i\in\{1,...,n-1\}$, $\omega_i$ belongs to $\Omega$.
	
		The proof is the same mutatis-mutandis as the proof of the Proposition \ref{Formula_sum_weight}.
\end{proof}
\begin{Rem}
	The same remark as in \ref{Rem_sum_weight} applies here, namely: for $r\geq 0$, if the sum in equality (\ref{SumRestrictedWeight}) is finite then for any complex number $z$ of modulus $|z|\leq r$, we have
	$$G_z(x,y;\Omega)=\sum_{\gamma\in\mathcal{P}h(x,y;\Omega)} w_z(\gamma).$$
\end{Rem}
\vspace{\baselineskip}
\section{Decomposition of \texorpdfstring{$p$}{p}-admissible path in a Tree}\label{Section_Decomposition_of_admissible_paths}
In the case of a nearest-neighbour random walk, a random walker starting from a vertex $x$ and reaching a vertex $y$ at time $n$ must have visited all vertices of the geodesic segment $[x,y]$, namely $x=x_0, x_1, ..., x_{m-1}, x_m=y$ in this exact order, before arriving at $y$. However, in the case of a finite range random walk, this may not always happen (see Figure \ref{Figure_Ex_admi_path_avoid_geos_seg}). But instead it passes successively and in this order through the balls centered at each vertices of the geodesic segment $[x,y]$ (See section \ref{Section_Decomposition_of_admissible_paths}).
This idea was exploited by Steven P.Lalley in \cite{Lalley_1993} to establish the first-order asymptotics in the case of aperiodic random walks, under the assumption that there exists a group of automorphisms $\Gamma$ preserving the transition kernel and acting transitively on the set of vertices of the tree.

Building on this approach, we define an algebraic curve $\mathcal{C}$, which we call Lalley's Curve in honour of its introduction by Lalley (See also \cite{Kazuhiko_1984} for a different presentation). We then apply the same method used for nearest-neighbour random walks, to obtain the asymptotics (\ref{Intro_eq_asymptotic}).

\begin{figure}[ht]
\begin{center}
		\includegraphics[scale=0.35]{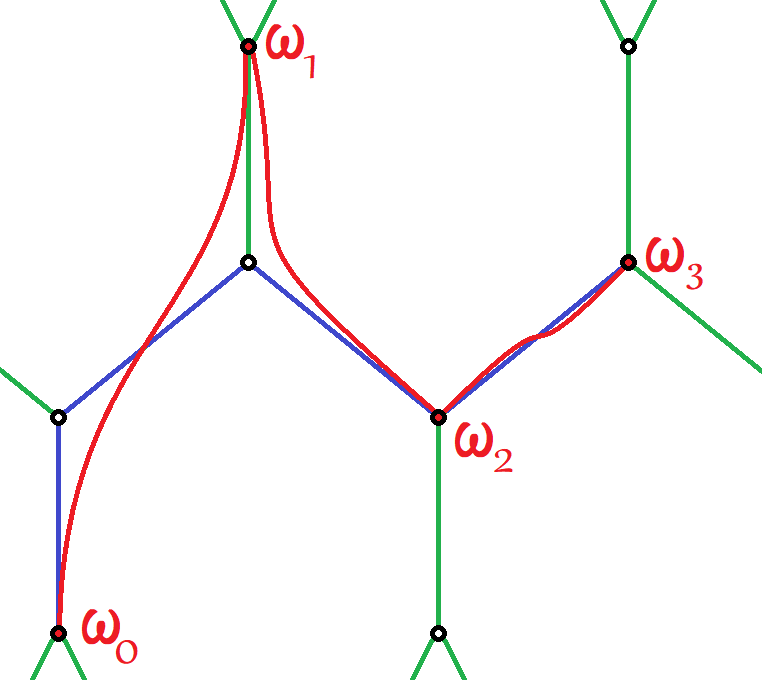}
		\caption{Example of a finite range random walk in a tree, with range at least $3$, avoiding some vertices of the geodesic paths linking the starting point of the random walk, to its position at time $3$.}\label{Figure_Ex_admi_path_avoid_geos_seg}
	\end{center}
	\end{figure}
	
Let $X=(X_0,X_1)$ be a tree of bounded valence, with set of vertices $X_0$, and set of edges $X_1$. Let $p:X_0\times X_0\to [0,1]$ be a transition kernel with finite range over $X_0$, and let $k$ be a non-negative integer such that for any vertices $x,y$ of $X$, we have
$$ d(x,y)>k \Rightarrow p(x,y)=0.$$

In this section we explain how to decompose $p$-admissible paths $\gamma=(\omega_0,...,\omega_n)\in\mathcal{P}h(X_0)$ along a geodesic segment $[x,y]$, with respect to $\mathcal{B}_k$ (See Definition \ref{DEF:BALLS} below), with $x,y$ being two distinct vertices of the tree such that, $\omega_0\in\mathcal{B}_k(x)$ and $\omega_n\in\mathcal{B}_k(y)$. 
We start with two definitions:

	\begin{Def}\textit{The ball $\mathcal{B}_k$}\label{DEF:BALLS}\index{$\mathcal{B}_k$}
	
		Let $X=(X_0,X_1)$ be a tree with set of vertices $X_0$, and set of edges $X_1$. And let $d$ be the combinatorial distance on $X$. For any vertex $y$ of $X$, and any non-negative integer $k$, we denote by $\mathcal{B}_k(y)$ the ball of radius $k$ centered at $y$ in $X$, i.e
		\begin{equation}\label{Def_of_the_ball}
		     \mathcal{B}_k(y):=\Big\{w\in X_0 : d(y,w)\leq k \Big\}.
		\end{equation}
		
		We also denote by $\partial \mathcal{B}_k(y)$ its outer boundary, that is: 
		\begin{equation}\label{Def_of_the_outer_boundary}
			\partial \mathcal{B}_k(y)=\Big\{w'\in \mathcal{B}_k(y): \exists w\in \mathcal{B}_k(y) \text{ such that } (w,w')\in X_1 \Big\}.
		\end{equation}
		
		We easily check that $\partial \mathcal{B}_k(y)$ is the sphere $\Big\{w'\in X_0 : d(y,w')=k+1\Big\}$.
	\end{Def}
	
	\begin{Def}\textit{Shadow of a vertex under an other vertex}\index{Shadow}
	
		Let $X=(X_0,X_1)$ be a tree with set of vertices $X_0$, and set of edges $X_1$. Let $x,y\in X_0$ be two distinct vertices of $X$.
		We denote by $\Omega_x^y$ the subset of vertices
\begin{equation}\label{Def_ombre}
	\Omega_x^y:=\Big\{w\in X_0: x\in [y,w]\Big\}
\end{equation}
called \textit{shadow of $x$ under $y$}\footnote{If we place an omnidirectional light source at $y$, and a blackout panel at $x$, blocking the passage of light rays from $y$, then $\Omega^y_x$ corresponds to the entire part of $X_0$, that is in the shadow of the panel at $x$ under the light from $y$, the vertex $x$ included. This definition differs from the initial definition of Sullivan's shadows used in \cite{Sullivan_1979}.}. 
In the case where $x$ and $y$ are \textit{neighbours}, that is $(x,y)\in X_1$ is an edge, the shadow $\Omega_x^y$ is the set of vertices $w$ that are closer to $x$ than to $y$: $$\Omega_x^y=\Big\{ w : d(x,w)<d(y,w)\Big\}.$$

\begin{figure}[h]
        \centering
		\includegraphics[scale=0.25]{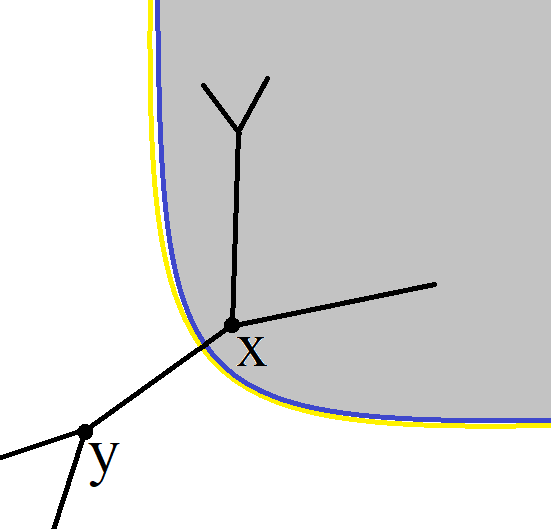}
		\caption{The shadow of $x$ under $y$, $\Omega_x^y$.}\label{Figure_of_shadows_def}
	\end{figure}
	\end{Def}
	
		\subsection{\texorpdfstring{$k$}{k}-moderated path}\index{Moderated path}
	\begin{Def}\textit{paths and $k$-moderated paths}\label{Def_k_moderated_path}
	
	Let $X=(X_0,X_1)$ be a tree with set of vertices $X_0$, and set of edges $X_1$ and let $k$ be a positive integer.
		We will say that a finite sequence of vertices of $X$, $\gamma=(\omega_0,...,\omega_n)$ is a $k$-moderated path if the distance between two consecutive vertices of $\gamma$ is less or equal to $k$, i.e for any integer $i\in\{1,...n\}$, we have $d(\omega_{i-1},\omega_i)\leq k$. 
	\end{Def}
	\begin{Ex} ~
	
	\begin{itemize}
	\item Any geodesic segment in $X$ is $1$-moderated.
	\item Let $p$ be a finite range transition kernel on $X_0$ and $k\geq 1$ be a positive integer such that for any pair of vertices $(x,y)\in X_0\times X_0$, we have 
		\begin{equation}\label{Property_over_Kernel}
			d(x,y)>k\Rightarrow p(x,y)=0,
		\end{equation}
	 Then any $p$-admissible path is $k$-moderated.
\end{itemize}			
	\end{Ex}
	\begin{Lem}\label{Lem_pre_Shadow_Lemma}
	Let $X=(X_0,X_1)$ be a tree with set of vertices $X_0$, and set of edges $X_1$. Let $z,u,v\in X_0$ be vertices of $X$, $u$ and $v$ different from $z$, and let $k$ be a non-negative integer.
	
	If $d(u,v)\leq k$ and $u$ is not in $\mathcal{B}_k(z)$, and if we denote by $z^*$ the only neighbour of $z$, in the geodesic segment $[z,u]$, then $u$ and $v$ both belong to the shadow $\Omega_{z^*}^z$ of $z^*$ under $z$ (i.e $z^\star\in [z,v]\cap [z,u]$).
	\end{Lem}
	\begin{proof}
		By contraposition let $u$ and $v$ be two vertices, distinct from $z$, with $u\notin\mathcal{B}_k(z)$, and suppose that they belong to the shadow of two different neighbours of $z$, under $z$. That is to say, suppose there exists $z^-$ and $z^+$ two distinct neighbours of $z$ such that, $u\in\Omega_{z^-}^z$ and $v\in\Omega_{z^+}^z$. Then the concatenation of the geodesic segments $[u,z^-], [z^-,z^+], [z^+,v]$ is the geodesic segment $[u,v]$ in $X$ between $u$ and $v$. Thus necessarily $d(u,v)\geq d(u,z^-)+d(z^-,z^+)=k+1$ and the result follows (see figure \ref{Figure_of_shadows} ).
	\end{proof}
	
	\begin{figure}[h]
	\centering
		\includegraphics[scale=0.40]{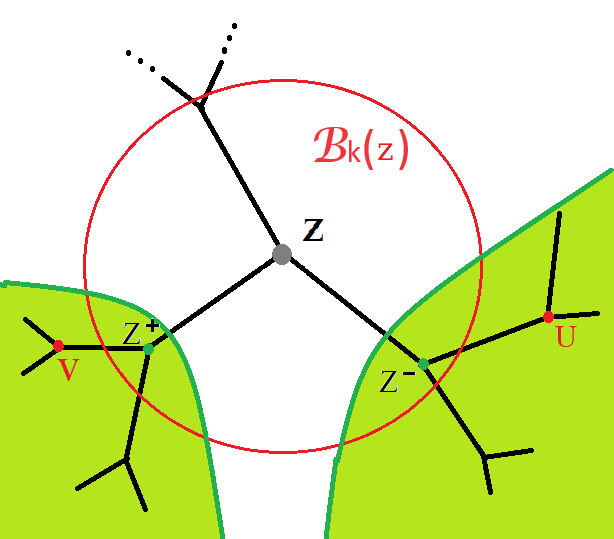}
		\caption{\og $u\in\Omega_{z^-}^z$ and $v\in\Omega_{z^+}^z$. \fg{}}\label{Figure_of_shadows}
	\end{figure}
	
	From this Lemma and the definition of a $k$-moderated path (Definition \ref{Def_k_moderated_path} we get the two following lemmas:
		
	\begin{Lem}\label{Shadow_Lemma}	~
	
	Let $X=(X_0,X_1)$ be a tree with set of vertices $X_0$, and set of edges $X_1$, $\gamma=(\omega_0,...,\omega_n)$ be a $k$-moderated path in $X$ and $z\neq \omega_0$ be a vertex of $X$. Denote by $i_z:=\min\{ i : \omega_i\in \mathcal{B}_k(z)\}$ ( taking $\min{\emptyset}=n$ ) the first index $i$ such that $\omega_i$ is in the ball $\mathcal{B}_k(z)$, and let $z^*$ be the only neighbour of $z$ such that $\omega_0$ belongs to the shadow of $z^*$ under $z$.
	
	 Then all the vertices $\omega_0,...,\omega_{i_z}$ belong to the shadow $\Omega_{z^*}^z$ of $z^*$ under $z$.
	\end{Lem}
	\begin{proof}
		If $i_z=0$ then the result is trivial, else an iterative application of the Lemma \ref{Lem_pre_Shadow_Lemma} with $(u,v)$ successively equal to $(\omega_{i-1},\omega_{i})$, for $i=1,...,i_z$, gives that the vertices $\omega_1,...,\omega_{i_z}$ all belong to the same shadow, that is $\Omega_{z^*}^z$. Note that the Lemma \ref{Lem_pre_Shadow_Lemma} can be applied because the vertices $\omega_0,...,\omega_{i_z-1}$ are not in $\mathcal{B}_k(z)$, by definition of $i_z$ and because $\omega_{i_z}$ is not equal to $z$, since $d(\omega_{i_z-1},\omega_{i_z})\leq k < k+1= d(z,\mathcal{B}_k(y)^\complement)$.
	\end{proof}
	\begin{Rem}
		We could have defined $z^*$ as the only neighbour of $z$ such that $w_i$ belongs to the shadow of $z^*$ under $z$, where $i\leq i_z$ is a non-negative integer.
	\end{Rem}
	\begin{Lem}[Intermediate vertex lemma]\label{Peage_Lemma}
	Let $X=(X_0,X_1)$ be a tree with set of vertices $X_0$, and set of edges $X_1$, $\gamma=(\omega_0,...,\omega_n)$ be a $k$-moderated path, and let $x,y$ be two distinct vertices of $X$. 
	
	Suppose that the source $\omega_0$ of $\gamma$ is in $\mathcal{B}_k(x)$ and its aim $\omega_n$ is in $\mathcal{B}_k(y)$, then for any vertex $z$ in the geodesic segment $[x,y]$, there exists a vertex $\omega_i$ of $\gamma$, for some $i\in\{0,...,n\}$ such that $\omega_i$ belongs to the ball $\mathcal{B}_k(z)$.
	\end{Lem}
	\begin{proof}
		If $z=x$, then the result follows from the assumption. Else, denote by $z^-$ the neighbour of $z$ such that $z^-$ is in the geodesic segment $[x,z]$ (equivalently such that $x$ is in the shadow $\Omega_{z^-}^z$ of $z^-$ under $z$).	
		
		By contradiction, suppose that no vertex of $\gamma$ belongs to $\mathcal{B}_k(z)$. Then by Lemma \ref{Lem_pre_Shadow_Lemma} applied to $(x,w_0)$ and Lemma \ref{Shadow_Lemma} applied to $\gamma$ and $z$, all vertices of $\gamma$ belong to the shadow $\Omega_{z^-}^z$. However, since $\omega_n$ is in $\mathcal{B}_k(y)$, if $\omega_n$ also belongs to the shadow $\Omega_{z^-}^z$ then the concatenation of the geodesic segments $[\omega_n,z^-],[z^-,z]$ and $[z,y]$, is the geodesic segment $[\omega_n,y]$ and in particular $d(\omega_n,z)=d(\omega_n,y)- d(z,y)\leq k- d(z,y)$, so $\omega_n$ belongs to $\mathcal{B}_k(z)$. This is the contradiction we were looking for. The lemma follows.
	\end{proof}
	\vspace{\baselineskip}

	\subsection{Paths admitting a decomposition along a geodesic segments}

	\begin{Def}\textit{Path admitting a decomposition along a geodesic segment}\index{Path admitting a decomposition along a geodesic segment}
\label{Def_Chemin_admitting_a_decomposition_along_a_geodesic}

	 Let $X=(X_0,X_1)$ be a tree with set of vertices $X_0$, and set of edges $X_1$, $\gamma=(\omega_0,...,\omega_n)$ be a finite sequence of vertices of $X$, $k$ be a non-negative integer and $x,y\in X_0$ be two distinct vertices of $X$.
	 
	 We will say that $\gamma$ \textit{admits a decomposition along the geodesic segment $[x,y]=(x_0,...,x_m)$ with respect to $\mathcal{B}_k$}, if $\gamma$ passes successively, and in that order, through the disks $\mathcal{B}_k(x_0),...,\mathcal{B}_k(x_m)$.

		Formally, if we note, for all $s\in\{0,...,m\}$, 
		$$t_s:=\min\{ t : \omega_t\in\mathcal{B}_k(x_s)\},$$
		the index of the first vertex of $\gamma$ that is in $\mathcal{B}_k(x_s)$, with the convention that $\min(\emptyset)=\infty$; we will say that $\gamma$ admits a decomposition along the geodesic segment $[x,y]=(x_0,...,x_m)$ with respect to $\mathcal{B}_k$, if:
	\begin{enumerate}[label=\roman*)]
		\item for all $s\in\{0,...,m\}$, $t_s$ is finite,
		\item $(t_s)_s$ is a finite non-decreasing sequence: $0\leq t_0\leq t_1\leq ...\leq t_m.$
	\end{enumerate}
	\end{Def}
\begin{Ex}
	Using the notations of the definition, 
\begin{itemize}
	\item if $k=0$, then for any vertex $y\in X_0$, $\mathcal{B}_0(y)=\{y\}$. $\gamma$ will admit a decomposition along the geodesic segment $[x,y]=(x_0,...,x_m)$ with respect to $\mathcal{B}_k$, if and only if $\gamma$ passes successively, and in that order, through the vertices of the geodesic segment $[x,y]=(x_0, x_1, ..., x_m)$. This is the case if $\gamma$ is a $p$-admissible path for a nearest-neighbour random walk, i.e a random walk such that $p(x,y)>0$ if, and only if $d(x,y)=1$.
	\item For any $k\in\mathbb{N}$, a finite sequence of vertices $\gamma$ \textbf{does not admit} a decomposition along a geodesic segment $[x,y]=(x_0,...,x_m)$ with respect to $\mathcal{B}_k$, if, and only if, it avoids one of the balls $\mathcal{B}_k(x_s)$, for $s\in\{0,...,m\}$, in which case $t_s=\infty$, or if there exist $s_1<s_2$ such that $\gamma$ visits $\mathcal{B}_k(x_{s_2})$ before visiting $\mathcal{B}_k(x_{s_1})$, i.e. using the notations from the definition, $t_{s_1}>t_{s_2}$ (See figure \ref{Fig_decomposition_chemins_1}).
\end{itemize}
\end{Ex}

	\begin{figure}
		\includegraphics[scale=0.70]{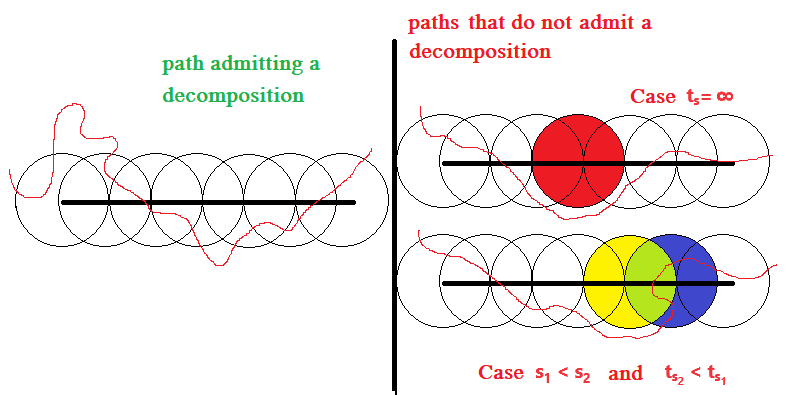}
		\caption{Examples of paths admitting or not a decomposition along a geodesic segment.}\label{Fig_decomposition_chemins_1}
	\end{figure}
	
	\begin{Prop}\label{k-moderated_implies_admits_decomposition_along_geod}
	Let $X=(X_0,X_1)$ be a tree with set of vertices $X_0$, and set of edges $X_1$, $k$ be a non-negative integer and $x,y\in X_0$ be two distinct vertices of $X$.
	
	Any $k$-moderated path $\gamma=(\omega_0,...,\omega_n)$ such that, $\omega_0$ and $\omega_n$ belong respectively to the balls $\mathcal{B}_k(x)$ and $\mathcal{B}_k(y)$, admits a decomposition along the geodesic segment $[x,y]$ with respect to $\mathcal{B}_k$.
	\end{Prop}
	\begin{proof}
	We use the notation of the statement. For $s\in\{0,...,m\}$, denote by $t_s:=\min\{ t: \omega_t\in\mathcal{B}(x_s)\}$ the first time $\gamma$ gets in $\mathcal{B}(x_s)$. From Lemma \ref{Peage_Lemma}, for any $s\in\{0,...,m\}$, the time $t_s$ is finite. We also have that for any $s\in\{1,...,m\}$, again by Lemma \ref{Peage_Lemma} applied to the $k$-moderated path $(\omega_0,...,\omega_{t_s})$ with $\omega_{t_s}\in\mathcal{B}_k(x_s)$, that $t_{s-1}\leq t_s$, hence the sequence $(t_s)_s$ is non-decreasing. It follows by definition that $\gamma$ admits a decomposition along the geodesic segment $[x,y]$.
	\end{proof}

		\subsection{Decomposition of admissible paths along a geodesic segment}
		
Any finite sequence of vertices $\gamma=(\omega_0,...,\omega_n)$ admitting a decomposition along a geodesic segment $[x,y]=(x_0,...,x_n)$, with respect to $\mathcal{B}_k$, can be decomposed in a unique way as a product (for concatenation of path) of sequence $\gamma_1\ast\gamma_2\ast...\ast\gamma_n\ast\gamma'$ such that for any $i\in\{1,...,n\}$, all vertices of $\gamma_i$ except its aim, belong to the complement $\mathcal{B}(x_i)^\complement\subset X_0$. In order to avoid overwhelming this paper with notation, we only develop this decomposition property for $p$-admissible paths, with respect to some discrete Markov chain $(X_0,p)$.

All definitions and properties in this subsection are stated under the following assumptions:
		
		Let $X=(X_0,X_1)$ be a tree of bounded valence, with set of vertices $X_0$ and set of edges $X_1$, let  $p:X_0\times X_0\to [0,1]$ be a finite range transition kernel over $X_0$ such that the Markov chain $(X_0,p)$ is irreducible. Let $k\geq 0$ be a non-negative integer such that for any vertices $x,y\in X_0$ we have
	\begin{equation}\label{above_eqts1}
	d(x,y)>k \Rightarrow p(x,y)=0.
\end{equation}	

The above property (\ref{above_eqts1}) guarantees that any $p$-admissible path is a $k$-moderated path, and as such, admits a decomposition along some geodesic segments with respect to $\mathcal{B}_k$.

\begin{Def}\textit{The set $\Xi$.}\label{Def_of_Xi}\index{$\Xi$}
	
	For $y\in X_0$ a vertex, we define the set
	\begin{equation}
		\Xi_{y}:=\Big\{(a,b)\in\partial\mathcal{B}_k(y)\times\mathcal{B}_k(y):\,\exists n\geq 0, p^{(n)}(a,b;\mathcal{B}_k(y)^\complement)>0\Big\}.
	\end{equation}
	We set:
	\begin{equation}\label{union_Xi_B}
		\Xi=\bigsqcup_{y\in X_0} \Xi_{y},
	\end{equation}
	to be the disjoint union of the sets $\Xi_y$ for $y$ in $X_0$.
	We will denote by $(a,b)_{y}$ an element of $\Xi$, meaning that it is the pair $(a,b)$ that belongs to $\Xi_{y}$.  
\end{Def}

\begin{Rem} With the notations of the definition, any action of a subgroup $\Gamma<\operatorname{Aut}(X)$ of automorphisms of $X$, preserving the transition kernel $p$ (i.e $\forall g\in\Gamma,\forall x,y\in X_0,\, p(gx,gy)=p(x,y)$), naturally induces an action of $\Gamma$ on $\Xi$:
$$ g\in\Gamma, (a,b)_{y}\in\Xi, \, g\cdot(a,b)_{y}= (ga,gb)_{gy}.$$
\end{Rem}
\begin{Rem}\label{Rem_verteces_admissible_in_shadow}
	From Lemma \ref{Shadow_Lemma}, for any $(a,b)_{y}\in\Xi$, if $x$ denotes the unique neighbour of $y$ in the geodesic segment $[a,y]$, then all the vertices of a $p$-admissible paths $\gamma\in\mathcal{P}h(a,b;\mathcal{B}(y)^\complement)$ lie in the shadow $\Omega_x^y$ of $x$ under $y$.
\end{Rem}

\begin{Notation} \label{Notation_mathcal_B}
	To simplify the notation, we will write $\mathcal{B}$ instead of $\mathcal{B}_k$. Since the non-negative integer $k$ is fixed, there will be no ambiguity.
\end{Notation}

\begin{Def}\textit{The sets $\Xi_{[x,y]}$}\label{Def_Xi_Segment}\index{$\Xi_{[x,y]}$}

	For $(x,y)\in X_0\times X_0$ a pair of vertices of $X$, $x\neq y$, denote by $[x,y]=(x_0,x_1,...,x_m)$ the geodesic segment between $x$ and $y$ in $X$. We define the set $\Xi_{[x,y]}$ as the set of all tuples $(c_0,c_1,...,c_l)\in (X_0)^{l+1}$, where $0<l\leq m$ is a positive integer, for which there exist a strictly increasing finite sequence of integers $0< i_1<...<i_l \leq m$, such that:

 		\begin{gather}
		c_0\in \mathcal{B}(x)\cap\partial\mathcal{B}(x_{i_1}) \nonumber ;\\
		c_1\in \mathcal{B}(x_{i_1})\cap \partial \mathcal{B}(x_{i_2}),\quad ...\, ,\,c_{l-1}\in \mathcal{B}(x_{i_{l-1}})\cap\partial\mathcal{B}(x_{i_l}) \label{crossing_vertindex}\\
		c_l\in\mathcal{B}(x_{i_l})\cap\mathcal{B}(y)\nonumber.
		\end{gather}

	And such that for any $j\in\{1,...,l\}$,
	$$\exists n_j\geq 0,\, p^{(n_j)}(c_{j-1},c_j; \mathcal{B}(x_{i_j})^\complement)>0.$$
	In other words $\Xi_{[x,y]}$ is the set of all tuples $(c_0,...,c_l)$ such that $l\leq m$ is a non-negative integer, $c_0\in\mathcal{B}(x)\setminus\mathcal{B}(y)$, $c_l\in\mathcal{B}(y)$ and such that there exists a strictly increasing finite sequence of $l$ integers $0<i_1<...<i_l\leq m$ verifying for all $j\in\{1,...,l\}$, $(c_{j-1},c_j)\in\Xi_{x_{i_j}}$.
\end{Def}

\begin{Lem}\label{Lem_2.6}
Using the notations from the previous definition, the $l$-tuple of integers $(i_1,...,i_l)$ is uniquely determined by the $(l+1)$-tuple $(c_0,...,c_l)$ and the geodesic segment $[x,y]$: 

For a fixed $s\in\{1,...,l\}$, we have
\begin{equation}\label{eqa1}
	i_s=1+\max\{ i : c_{s-1}\in\mathcal{B}(x_i)\}.
\end{equation}
	We will note $\iota_{[x,y]}(c_0,..,c_l)$ this $l$-tuple and name it the \textit{crossing indices} associated with $(c_0,...,c_l)$ along $[x,y]$.
\end{Lem}
\begin{proof}
	Let $0<i_1<...<i_l\leq m $ be a strictly non-decreasing sequence of integer satisfying (\ref{crossing_vertindex}). 	
	Take $s\in\{1,...,l\}$ arbitrary. By definition $c_{s-1}\notin \mathcal{B}(x_{i_s})$.  Considering the distances we have the equality between sets $\mathcal{B}(x_{i_{(s-1)}})\setminus\mathcal{B}(x_{i_s})=\mathcal{B}(x_{i_{(s-1)}})\setminus\left(\bigcup_{i\geq i_s}\mathcal{B}(x_{i})\right)$.
	
	To prove (\ref{eqa1}), we thus just need to show that $c_{s-1}$ is in $\mathcal{B}(x_{(i_s) - 1})$. Consider $w\in [x_{i_{(s-1)}},x_{i_s}]$ such that the distance $d(c_{s-1},w)$ is minimal (i.e $w$ is the nearest vertex from $c_{s-1}$, in the geodesic segment $[x_{i_{(s-1)}},x_{i_s}]$). Note that $w\neq x_{i_s}$ since $d(c_{s-1},x_{i_s})=k+1>k=d(c_{s-1},x_{i_{(s-1)}})$. Then 
	$$k+1=d(c_{s-1},x_{i_s})= d(c_{s-1},w)+d(w,x_{(i_s) -1})+\underbrace{d(x_{(i_s) -1},x_{i_s})}_{=1},$$ thus $d(c_{s-1},x_{(i_s)-1})=d(c_{s-1},w)+d(w,x_{(i_s) -1})=k+1-1=k$. Hence $c_{s-1}\in\mathcal{B}(x_{(i_s) - 1})$ as expected.
\end{proof}

\begin{Rem}\label{Rem_Xi_y_as_a_union_of_Xi_x_y}
	In the previous definition \ref{Def_Xi_Segment}, if $x$ and $y$ are neighbours (i.e. $d(x,y)=1$), then $\Xi_{[x,y]}$ identifies with a subset of $\Xi_y$. More precisely, with the notation of the definition, we have $m=l=1$ and $\Xi_{[x,y]}$ consists of all pairs $(a,b)\in\Xi_y$ such that $a$ belongs to $\mathcal{B}(x)$. This is equivalent to saying that $x$ belongs to the geodesic segment $[a,y]$, or equivalently that $a\in\Omega_x^y$. In particular $x$ only depends on $a$ and $y$ (See remark \ref{Rem_verteces_admissible_in_shadow}).
	We immediately verify using the Lemma \ref{Shadow_Lemma} that $\Xi_y$ is the disjoint union:
	\begin{equation}
		\Xi_y=\bigsqcup_{x\,:\,d(x,y)=1}\Xi_{[x,y]}.
	\end{equation}
	
\end{Rem}
\begin{Rem}\label{Finiteness_of_the_Xi_segment}
	With the notations of the definition, for any $x,y\in X_0$, $x\neq y$, the set $\Xi_{[x,y]}$ is finite. Indeed, having fixed a strictly increasing sequence of integers $0<i_1<...<i_l\leq m$, by finiteness of the set $\Xi_{x_{i_1}}\times\cdot\cdot\cdot \times \Xi_{x_{i_l}}$, there exists only a finite number of $(l+1)$-tuples $(c_0,...,c_l)$ in $\Xi_{[x,y]}$, such that $\iota_{[x,y]}(c_0,...,c_l)=(i_1,...,i_l)$. Furthermore there is only a finite number of such increasing sequences $0<i_1<...<i_l\leq m$. Hence $\Xi_{[x,y]}$ must be finite.
\end{Rem}

\begin{Prop}\label{Prop_decomposition_of_paths_on_a_tree}
Let $\gamma\in\mathcal{P}h(X)$ be a $p$-admissible path. 
Denote its source by $a$, and its aim by $b$. Let $x,y\in X_0$ be two distinct vertices such that, the vertex $a$ belongs to $\mathcal{B}(x)\setminus\mathcal{B}(y)$ and the vertex $b$ belongs to $\mathcal{B}(y)$. Also denote by $(x_0,...,x_m)$ the geodesic segment $[x,y]$.

There exist an integer $0<l\leq m$, an $(l+1)$-tuple $(c_0,...,c_l)\in\Xi_{[x,y]}$ and $l+1$ $p$-admissible paths $\gamma_1,...\gamma_{l},\gamma_f$ such that:
	\begin{equation}\label{Decomposition_Formula}
		\gamma=\gamma_1\ast\cdot\cdot\cdot\ast\gamma_l\ast\gamma_f;
	\end{equation}
	and for all $s\in\{1,...,l\}$, we have $\gamma_s\in\mathcal{P}h(c_{s-1},c_s;\mathcal{B}(x_{i_s})^\complement)$, and $\gamma_f\in \mathcal{P}h(c_l,b)$, where $(i_1,...,i_l)=\iota_{[x;y]}(c_0,...,c_l)$.

	The decomposition (\ref{Decomposition_Formula}) and the $(l+1)$-tuple of vertices $(c_0, ...,c_l)$ are unique for these properties, and are called respectively \textit{the decomposition of} $\gamma$ \textit{along} $[x,y]$ with respect to $\mathcal{B}$, and the \textit{crossing vertices of} $\gamma$ \textit{along} $[x,y]$ with respect to $\mathcal{B}$.
\end{Prop}
\begin{proof}
	We use the notation of the statement, denote by $(\omega_0,...,\omega_n)$ the $p$-admissible path $\gamma$. For $j\in\{0,...,n\}$ denote by $t_j:=\min\{ t: \omega_t\in\mathcal{B}(x_j)\}$, the first time $\gamma$ enters the ball $\mathcal{B}(x_j)$. From Proposition \ref{k-moderated_implies_admits_decomposition_along_geod} we have $t_0=0\leq t_1\leq ...\leq t_m\leq n$. By definition, for $j\in\{0,...,m\}$ the vertex $\omega_{t_{j}}$ is in the ball $\mathcal{B}(x_j)$, and for all $t<t_j$, the vertex $\omega_t$ is not in the ball $\mathcal{B}(x_j)$. We then set the $p$ admissible paths:
	\begin{align*}
		\gamma_1'&:=(\omega_0,\omega_1,...,\omega_{t_1})\in\mathcal{P}h(\omega_0,\omega_{t_1};\mathcal{B}(x_1)^\complement);\\
		\gamma_2'&:=(\omega_{t_1},\omega_{t_1+1}...,\omega_{t_2})\in\mathcal{P}h(\omega_{t_1},\omega_{t_2};\mathcal{B}(x_2)^\complement);\\
		...\\
		\gamma_m'&:=(\omega_{t_{(m-1)}},\omega_{t_{(m-1)}+1},...,\omega_{t_m})\in\mathcal{P}h(\omega_{t_{(m-1)}},\omega_{t_{m}};\mathcal{B}(x_m)^\complement);\\
		\gamma_f&:=(\omega_{t_m},...,\omega_n)\in\mathcal{P}h(\omega_{t_m},\omega_n).
	\end{align*}
	This gives the decomposition:
	$$\gamma=\gamma_1'\ast\cdot\cdot\cdot\ast\gamma_m'\ast\gamma_f,$$
	where $\gamma_i'$, for $i\in\{1,...,m\}$, and $\gamma_f$ are $p$-admissible paths that could possibly have length zero. We now remove from this decomposition the $p$-admissible paths of length zero. Consider the set of integers $i\in\{1,...,m\}$ such that the $p$-admissible path $\gamma_i$ is non trivial, $\{i : 1\leq i \leq m, \, t_{i-1}<t_{i}\}$. And write its elements as $1\leq i_1<i_2<...<i_l\leq m$. We then define the $(l+1)$-tuple $(c_0,...,c_l)$ as $c_0:=\omega_0$, $c_1:=\omega_{t_{i_1}}$, ..., $c_l:=\omega_{t_{i_l}}$ and verify, by definition of $i_1,...,i_l$ and $t_1,...,t_m$ that
	\begin{align*}
		c_0&:=\omega_{0}=\omega_{t_{1}}=...=\omega_{t_{i_1-1}}\in\mathcal{B}(x_0)\cap\partial\mathcal{B}(x_{i_1})\\
		c_1&:=\omega_{t_{i_1}}=\omega_{t_{i_1+1}}=...=\omega_{t_{i_2-1}}\in\mathcal{B}(x_{i_1})\cap\partial\mathcal{B}(x_{i_2}),\\
 		&...\\
  		c_l&:=\omega_{t_{i_l}}=\omega_{t_{i_l+1}}=...=\omega_{t_{m}}\in\mathcal{B}(x_m).
	\end{align*}
	For $s\in\{1,...,l\}$, we finally set $\gamma_s:=\gamma_{i_s}'\in\mathcal{P}h(c_{s-1},c_s;\mathcal{B}(x_{i_s})^\complement)$. We thus have $(c_0,...,c_l)\in\Xi_{[x,y]}$ and as desired:
	$$\gamma=\gamma_1\ast\cdot\cdot\cdot\ast\gamma_l\ast\gamma_f.$$
	
	We now need to show that this decomposition is unique.
	Since the crossing vertices $c_0,...,c_l$ correspond to the place where the $p$-admissible path $\gamma$ enters, for the first time, in each ball $\mathcal{B}(x_s)$, given $[x,y]$, they are entirely determined by the crossing times $t_1,...,t_n$ and $\gamma$. Thus the crossing vertices are entirely determined by $\gamma$. Hence $(c_0,...,c_l)$ is unique.
	
Let $(i_1,...,i_l)$ be the non-decreasing sequence of indices $\iota_{[x,y]}(c_0,...,c_l)$ from Lemma \ref{Lem_2.6}. By definition of $\mathcal{P}h(c_0,c_1;\mathcal{B}(x_1)^\complement)$, we necessarily have, by definition $\gamma_1=(w_0,...,w_{t_{i_1}})$. Then necessarily $\gamma_2=(w_{t_{i_1}},...,w_{t_{i_2}})$, ... Iterating this process until getting $\gamma_l=(w_{t_{i_{l-1}}},...,w_{t_{i_l}})$, we finally also obtain the uniqueness of the decomposition (\ref{Decomposition_Formula}).
\end{proof}

\begin{Cor}\label{corollary_of_decomposition_of_paths_on_a_tree}
	Let $a,b,y,x$ be vertices of $X$ such that $a$ is in $\mathcal{B}(x)\setminus \mathcal{B}(y)$, $b$ is in $\mathcal{B}(y)$ and such that $\mathcal{P}h(a,b;\mathcal{B}(y)^\complement)$ is non-empty. Denote by $(x_0,...,x_m)$ the geodesic segment $[x,y]$.
	
	 Let $\gamma \in \mathcal{P}h(a,b;\mathcal{B}(y)^\complement)$ be a $p$-admissible path, then the previous Proposition \ref{Prop_decomposition_of_paths_on_a_tree} applies and, taking the same notations we necessarily have $c_l=b$ and $\gamma_f=(b)$.

	Furthermore, for all $(c_0,c_1,...,c_l)\in\Xi_{[x,y]}$ the map
	\begin{equation}\label{The_inverse_map}
		(\gamma_1,...,\gamma_l)\mapsto\gamma_1\ast\cdot\cdot\cdot\ast\gamma_l
	\end{equation}
	establishes a bijection between the set of $l$-tuples of $p$-admissible paths $(\gamma_1,...,\gamma_l)$ in the cartesian product $\prod_{s=1}^l \mathcal{P}h\left(c_{s-1},c_s;\mathcal{B}(x_{i_s})^\complement\right)$, where $(i_1,...,i_l)=\iota_{x,y}(c_0,...,c_l)$, and the set of $p$-admissible paths $\gamma$ in $\mathcal{P}h(c_0,c_l;\mathcal{B}(x_m)^\complement)$ admitting a decomposition along the geodesic segment $[x,y]$, whose crossing vertices are $(c_0,...,c_l)\in\Xi_{[x,y]}$.
\end{Cor}
\begin{proof}
By definition, $c_l\in\mathcal{B}(y)$, but the only vertex of $\gamma$ that is in $\mathcal{B}(y)$ is $b$. Necessarily $c_l=b$ and it follows that $\gamma_f=(b)$, and therefore 
$$\gamma_1\ast\cdot\cdot\cdot\ast\gamma_l=\gamma.$$

The second part of the Corollary is a consequence of the uniqueness of the previous decomposition: The map that associates to any $p$-admissible path $\gamma$, the tuple of $p$-admissible paths $(\gamma_1,...,\gamma_l)$, from its path decomposition (\ref{Decomposition_Formula}) along $[x,y]$, clearly establishes a bijection between $l$-tuples of $p$-admissible paths $(\gamma_1,...,\gamma_l)\in\prod_{s=1}^l \mathcal{P}h\left(c_{s-1},c_s;\left(\bigcup_{i=i_s}^m\mathcal{B}(x_{i})\right)^\complement\right)$ 
	where $(i_1,...,i_l)=\iota_{[x,y]}(c_0,...,c_l)$, and the set of $p$-admissible paths $\gamma\in\mathcal{P}h(c_0,c_l;\mathcal{B}(x_m)^\complement)$ admitting a decomposition along the geodesic segment $[x,y]$, whose crossing vertices are $(c_0,...,c_l)\in\Xi_{[x,y]}$. The inverse of this map, is the map (\ref{The_inverse_map}).
	
To conclude the proof, we only need to prove that for any $s\in\{1,...,l\}$ we have	
	\begin{equation}\label{equality_on_paths}
		\mathcal{P}h\left(c_{s-1},c_s;\left(\bigcup_{i=i_s}^m\mathcal{B}(x_{i})\right)^\complement\right) = \mathcal{P}h\left(c_{s-1},c_s;\mathcal{B}(x_{i_s})^\complement\right).
	\end{equation}
	 We show this equality by double inclusion.
	The inclusion \og $\subseteq$ \fg{} is immediate, so we just need to prove the reciprocal inclusion \og$\supseteq$\fg{}. This is a consequence of the Lemma \ref{Peage_Lemma}: Let $\gamma_s=(\omega_0,...,\omega_{n_s})$ be a $p$-admissible path in $\mathcal{P}h\left(c_{s-1},c_s;\left(\bigcup_{i=i_s}^m\mathcal{B}(x_{i})\right)^\complement\right)$. By Lemma \ref{Peage_Lemma} (applied to $\gamma_s$ with \og $x=x_{i_{(s-1)}}$ \fg{} and \og $y=x_i$ \fg for some $i\geq i_s$ such that $c_s\in\mathcal{B}(x_i)$, $\gamma_s$ must possess a first vertex that belongs to $\mathcal{B}(x_{i_s})$. This vertex must be $\omega_{n_s}=c_s$, by definition of $\gamma_s$. In particular $\gamma_s$ is in $\mathcal{P}h\left(c_{s-1},c_s;\mathcal{B}(x_{i_s})^\complement\right)$ as expected.
\end{proof}
\subsection{Decomposition of restricted Green's functions}

		Let $X=(X_0,X_1)$ be a tree of bounded valence, with set of vertices $X_0$ and set of edges $X_1$, let $p:X_0\times X_0\to [0,1]$ be a finite range transition kernel over $X_0$ such that the Markov chain $(X_0,p)$ is irreducible. Let $k\geq 0$ be a non-negative integer such that for any vertices $x,y\in X_0$, we have
	\begin{equation}
	d(x,y)>k \Rightarrow p(x,y)=0.
\end{equation}	

Recall the writing of restricted Green's functions as a sum of path weights (\ref{SumRestrictedWeight}): For any pair of distinct vertices $(x,y)$ of $X$, and any pair $(a,b)\in\mathcal{B}(x)\times\mathcal{B}(y)$, with $a\notin\mathcal{B}(y)$ and for any complex number $z$ of modulus sufficiently small, we have
$$G_z(a,b;\mathcal{B}(y)^\complement)=\sum_{\gamma\in\mathcal{P}h(a,b;\mathcal{B}(y)^\complement)} w_z(\gamma).$$

Using the previous Corollary \ref{corollary_of_decomposition_of_paths_on_a_tree} to decompose $p$-admissible paths $\gamma$ in $\mathcal{P}h(a,b;\mathcal{B}(y)^\complement)$ in the above sum, we get:

\begin{Prop}\label{Green_decomposition_along_a_geodesic}~
Under the above assumptions, for any two distinct vertices $x,y\in X_0$, and any pair $(a,b)\in \mathcal{B}(x)\times\mathcal{B}(y)$, with $a\notin\mathcal{B}(y)$, for any complex number $z$ of modulus small enough\footnote{to guarantee the convergence of the power series expansion in the neighbourhood of $0\in\mathbb{C}$, of Green's functions.}, we have the equality:
	\begin{equation}\label{Sum_product_formula}
		G_z(a,b;\mathcal{B}(y)^\complement)=\sum_{\substack{(c_0,...,c_l)\in\Xi_{[x,y]}\\ c_0=a, \, c_l=b }}\left(\prod_{j=1}^l G_z(c_{j-1},c_j;\mathcal{B}(x_{i_j})^\complement)\right),
	\end{equation}
	with $(i_1,...,i_l):= \iota_{[x,y]}(c_0,...,c_l)$. Note that the above sum is finite (See remark \ref{Finiteness_of_the_Xi_segment}).
\end{Prop}
\begin{proof}
	For $\gamma\in\mathcal{P}h(a,b;\mathcal{B}(y)^\complement)$ arbitrary, in accordance with Corollary \ref{corollary_of_decomposition_of_paths_on_a_tree}, we will denote by $\gamma_1\ast\cdot\cdot\cdot\ast\gamma_l=\gamma$ the unique decomposition of $\gamma$ along the geodesic segment $[x,y]$ with respect to $\mathcal{B}$, and by $(c_0,...,c_l)\in\Xi_{[x,y]}$ its associated crossing vertices.
	We have, by multiplicativity of the weight function $w_z$ with respect to path concatenation, for any complex number $z\in\mathbb{C}$,
	$$w_z(\gamma)=w_z(\gamma_1)\cdot\cdot\cdot w_z(\gamma_l),$$
	then for any complex number $z$ of modulus strictly less than $R(a,b;\mathcal{B}(y)^\complement)$, the radius of convergence of the restricted Green's function $z\mapsto G_z(a,b;\mathcal{B}(y)^\complement)$, we have 
	\begin{align*}
		G_z(a,b;\mathcal{B}(y)^{\complement})&=\sum_{\gamma\in\mathcal{P}h(a,b;\mathcal{B}(y)^\complement)} w_z(\gamma)\\
		&=\sum_{\gamma\in\mathcal{P}h(a,b;\mathcal{B}(y)^\complement)} w_z(\gamma_1)\cdot \cdot \cdot w_z(\gamma_l)\\
		&=\sum_{\substack{(c_0,...,c_l)\in\Xi_{[x,y]}\\c_0=a,\,c_l=b}}\sum_\gamma w_z(\gamma_1)\cdot \cdot \cdot w_z(\gamma_l),
	\end{align*}
	where in the last equality, the second sum is taken over all $p$-admissible paths $\gamma$ whose crossing vertices along $[x,y]$ are $(c_0,...,c_l)$, and where the $p$-admissible paths $\gamma_1,...,\gamma_l$ are the ones appearing in the unique decomposition (from Corollary \ref{corollary_of_decomposition_of_paths_on_a_tree}) of $\gamma$, along the geodesic segment $[x,y]$ with respect to $\mathcal{B}$.
	
	Finally, from the above and using the bijection described in  Corollary \ref{corollary_of_decomposition_of_paths_on_a_tree}, for any complex number $z\in\mathbb{C}$ of modulus sufficiently small, we get
	\begin{align*}
		G_z(a,b;\mathcal{B}(y)^\complement)&=\sum_{\substack{(c_0,...,c_l)\in\Xi_{[x,y]}\\
		c_0=a, c_l=b}}\prod_{j=1}^l \left(\sum_{\gamma_j\in\mathcal{P}h(c_{j-1},c_j;\mathcal{B}(x_{i_j})^\complement)} w_z(\gamma_j)\right)\\
		&=\sum_{\substack{(c_0,...,c_l)\in\Xi_{[x,y]}\\
		c_0=a, c_l=b}}\prod_{j=1}^l G_z\left(c_{j-1},c_j;\mathcal{B}(x_{i_j})^\complement\right).
	\end{align*}
	The result is proved.
\end{proof}

\section{The Lalley's Curve, \texorpdfstring{$\mathcal{C}$}{C}}\label{Section_Lalley_s_curve}
	~
	
	All the statements of this section are given under the following assumptions:
Let $X=(X_0,X_1)$ be a tree of bounded valence, with set of vertices $X_0$ and set of edges $X_1$. Let $\Gamma<\operatorname{Aut}(X)$ be a group of automorphisms of $X$ that acts cofinitely on $X_0$ (i.e the sets of orbits $\Gamma\backslash X_0$ of $X_0$ under the action of $\Gamma$ is finite). Let $p:X_0\times X_0 \to [0,1]$ be a transition kernel that makes the Markov chain $(X_0,p)$ irreducible. Let $k$ be a non-negative integer, such that for any two vertices $x,y\in X_0$, if $d(x,y)>k$ then $p(x,y)=0$.
Besides we assume that the transition kernel $p$ is $\Gamma$-invariant, that is to say, for any vertices $x,y$ of $X$, and any automorphism $g\in\Gamma$, we have: 
		$$ p(gx,gy)=p(x,y).$$

In this section, we construct an algebraic curve $\mathcal{C}$ that contains the graph in the neighbourhood of $0$, of the family of restricted Green functions $z\mapsto G_z(a,b;\mathcal{B}(y)^\complement)$, for $(a,b)_y\in\Xi$. At this point, we will use the co-finiteness of the action of $\Gamma$ on $X_0$, to obtain a curve $\mathcal{C}$ that "lives" in a finite dimension vector space (See Proposition \ref{Prop_Action_cofinite} below).

		\subsection{\texorpdfstring{$\Gamma$-}{Gamma-}invariant complex-valued functions over \texorpdfstring{$\Xi$}{Xi}}	 

\begin{Prop}\label{Prop_Action_cofinite}~
    Under the above assumption, since the action of $\Gamma$ on $X_0$ is co-finite (i.e. $\operatorname{Card}(\Gamma\backslash X_0 )<\infty$), the induced action of $\Gamma$ on $\Xi$ is also co-finite.
\end{Prop}
\begin{proof}
	The point is to show that $\operatorname{Card}(\Gamma\backslash\Xi)$ is finite. 
	Let $S$ be a finite subset of $X_0$ such that $\Gamma S = X_0$. By definition, the sets \og $\Xi_y$ \fg{} are finite, the set $T$ defined below is therefore also finite.
	$$T:=\bigcup_{y\in S}\Xi_y.$$
	For any $(a,b)_y\in\Xi$, there exists $g\in\Gamma$ such that $gy\in S$, then $(ga,gb)_{gy}$ belongs to $\Xi_{gy}\subset T$. We deduce that $\Gamma T =\Xi$ and that the action of $\Gamma$ on $\Xi$ is co-finite.
\end{proof}

The action of $\Gamma$ on $\Xi$ naturally induces an action on the $\mathbb{C}$-vector space of complex-valued functions defined on $\Xi$, by right-hand composition:
$$g\in\Gamma,\, J:\Xi\to\mathbb{C},\quad g\cdot J= J\circ g^{-1}:\Xi\to\mathbb{C}.$$

We denote by $\mathcal{F}(\Xi,\mathbb{C})$, the space of functions $J:\Xi\to\mathbb{C}$. 
And we denote by $\mathcal{F}(\Xi,\mathbb{C})^\Gamma$ \index{$\mathcal{F}(\Xi,\mathbb{C})$}the subspace of $\Gamma$-invariant functions.
In the light of the previous proposition and under our assumptions, this vector space is finite-dimensional. Indeed, any $\Gamma$-invariant function from $\Xi$ to $\mathbb{C}$ is uniquely factorized by a complex-valued function, defined on the set of $\Gamma$-classes of $\Xi$. In other words, $\mathcal{F}(\Xi,\mathbb{C})^\Gamma$ identifies as $\mathcal{F}(\Gamma\backslash\Xi, \mathbb{C})$.

\begin{Ex}\label{example_gamma_inv_green}
	An essential example of function in $\mathcal{F}(\Xi,\mathbb{C})^\Gamma$ is the one of Green's functions. For a complex number $z$ of sufficiently small modulus\footnote{To guarantee the convergence of the power series associated with restricted Green's functions}, the function
	\begin{equation}\label{DEF:V_GREEN}\index{$v_z$}
		v_z:(a,b)_{y}\mapsto G_z(a,b;\mathcal{B}(y)^\complement)
	\end{equation}
	is in $\mathcal{F}(\Xi,\mathbb{C})^\Gamma.$
	Indeed, if we fix a complex number $z$ of sufficiently small modulus, let $(a,b)_{y}\in\Xi$, and let $g\in\Gamma$. Since $\Gamma$ preserves $p$, $\Gamma$ acts naturally on the set of $p$-admissible paths $\mathcal{P}h(X)$ and in particular preserves the weight function $w_z$ (Definition \ref{Def_function_path_weight}).
	The restriction at the source of the action $g:\mathcal{P}h(X)\to\mathcal{P}h(X)$ to the set $\mathcal{P}h(a,b;\mathcal{B}(y)^\complement)$ of $p$-admissible paths between $a$ and $b$ with intermediate vertices in $\mathcal{B}(y)^\complement$, induces a bijection:
	$$g:\mathcal{P}h(a,b;\mathcal{B}(y)^\complement)\to \mathcal{P}h(ga,gb;\mathcal{B}(gy)^\complement),$$
	since $g\mathcal{B}(y)=\mathcal{B}(gy)$. It follows by the formula (\ref{SumRestrictedWeight}), expressing Green's functions as a sum of path weights \og $w_z(\gamma)$\fg{}, that:
	\begin{align*}
		g^{-1}\cdot v_z((a,b)_y)&=G_z(ga,gb;\mathcal{B}(gy)^\complement)=\sum_{\gamma\in\,\mathcal{P}h(ga,gb;\mathcal{B}(gy)^\complement)}w_z(\gamma)\\
		&=\sum_{\gamma\in\,\mathcal{P}h(a,b;\mathcal{B}(y)^\complement)}w_z(g\gamma)=G_z(a,b;\mathcal{B}(y)^\complement)\\
		&=v_z((a,b)_y).
	\end{align*}

	This is the only example we will be using. However, we mention the following other examples of $\Gamma$-invariant functions in $\mathcal{F}(\Xi,\mathbb{C})$: for $z$ a complex number of modulus small enough, the following maps are in $\mathcal{F}(\Xi,\mathbb{C})^\Gamma$:
	\begin{multicols}{2}
	\begin{itemize}
		\item $(a,b)_{y}\mapsto G_z(a,b)$; 
		\item $(a,b)_{y}\mapsto F_z(a,b)$;
		\item $(a,b)_{y}\mapsto F_z(x,y)$;
		\item $(a,b)_{y}\mapsto G_z(x,y),$
	\end{itemize}
	\end{multicols}
	where $x$ is the only neighbour of $y$ in the geodesic segment $[a,y]$, and where $G_z$ and $F_z$ are the Green's function and first-passage generating function, respectively (Definitions \ref{DEF:GREEN_FUNCTION} and \ref{Def_Green_function_first_pass}).
\end{Ex}

From now on, we will denote by $E$\index{$E$} the space $\mathcal{F}(\Xi,\mathbb{C})^\Gamma$. A basis of $E$ is given by the functions $\{1_{[\theta]}\}_{[\theta]}$ where $[\theta]=\Gamma\theta$\index{$[\theta]=\Gamma\theta$} runs over the set $\Gamma\backslash\Xi$ of $\Gamma$-orbits of $\Xi$. We adopt the notation $[\theta]$ to denote the orbit $\Gamma\theta\in\Gamma\backslash\Xi$ of an element $\theta\in\Xi$.

Under our assumptions, recall that the set $\Gamma\backslash\Xi$ is finite and therefore the vector space $E$ has finite dimension (Proposition \ref{Prop_Action_cofinite}).

We then define a map $\psi:\mathcal{F}(\Xi,\mathbb{C})\to\mathcal{F}(\Xi,\mathbb{C})$ which stabilizes $E$, and whose coordinate are polynomial maps depending on a finite number of coordinates. The restriction of $\psi$ to $E$ will be used later to define an algebraic curve $\mathcal{C}$. 
		\subsection{Definition of the map \texorpdfstring{$\psi$}{psi} and the curve \texorpdfstring{$\mathcal{C}$}{C}}
		Motivated by formula (\ref{Sum_product_formula}) of the Proposition \ref{Green_decomposition_along_a_geodesic}, we introduce the following notation:
Let $x$ and $y$ be two distinct vertices of $X$, $[x,y]=(x_0,...,x_m)$ denotes the geodesic segment between $x$ and $y$, and let $J\in\mathcal{F}(\Xi,\mathbb{C})$ be a function. We note for any pair $(c,b)\in \mathcal{B}(x)\times\mathcal{B}(y)$, with $c\notin\mathcal{B}(y)$:
\begin{equation}\label{Formula_J_[z,y]}
	J_{[x,y]}(c,b):=\sum_{\substack{(c_0,...,c_l)\in\Xi_{[x,y]}\\ c_0=c, \, c_l=b}} J((c,c_1)_{x_{i_1}})\cdot J((c_1,c_2)_{x_{i_2}})\cdot \cdot \cdot J((c_{l-1},b)_{x_{i_l}}),
\end{equation}
where the sum is taken over all elements $(c_0,...,c_{l})\in\Xi_{[x,y]}$ such that, $l\in\mathbb{N}\setminus\{0\}$ and $c_0=c,\, c_l=b$. And where the $l$-tuple of indices $(i_1,...,i_l)$ corresponds to the $l$-tuple $\iota_{[x,y]}(c_0,...,c_l)$ introduced in Lemma \ref{Lem_2.6}. Note that this sum is finite (See remark \ref{Finiteness_of_the_Xi_segment}).

\begin{Notation}\label{notation_J_[z,y]}
	In order to lighten the notations, and since the indices $i_1,...,i_l$ are entirely determined by the $(l+1)$-tuple $(c_0,...,c_l)$ and the pair $(x,y)$ by the equalities
	$$ \forall s \in\{1,...,l\},\, i_s = 1 + \max\{i:c_{s-1}\in\mathcal{B}(x_i)\},$$
	we will note:
	$$J_{x,y}(c_0,...,c_l):=J((c_0,c_1)_{x_{i_1}})\cdot J((c_1,c_2)_{x_{i_2}})\cdot \cdot \cdot J((c_{l-1},c_l)_{x_{i_l}}).$$
	The formula (\ref{Formula_J_[z,y]}) is rewritten:
	\begin{equation}\label{Formula_J_[z,y]_bis}
		J_{[x,y]}(c,b):=\sum_{\substack{(c_0,...,c_l)\in\Xi_{[x,y]}\\c_0=c,\, c_l=b}} J_{x,y}(c_0,...,c_l).
	\end{equation}
\end{Notation}
Such product \og $J_{x,y}(c_0,...,c_l)$ \fg{} in the previous sum is what we call a \textit{monomial} of the polynomial function with non-negative coefficients $J\mapsto J_{[x,y]}(c,b)$.
\begin{Def}\textit{Monomial of a polynomial function with non-negative coefficients}\label{Def_monomial}

	Let $P:\mathbb{C}^n\to\mathbb{C}, (J_1,...,J_n)\to P(J_1,...,J_n)$ be a polynomial function with non-negative coefficients.
	Let $\alpha_1,...,\alpha_n$ be non-negative integers. We say that a product $J_1^{\alpha_1}\cdots J_n^{\alpha_n}$ is a \textit{monomial} of the polynomial function $J\mapsto P(J)$, if there exist a positive constant $c>0$ and a polynomial function $Q:\mathbb{C}^n\to\mathbb{C}$ with non-negative coefficients such that for any $J=(J_1,...,J_n)\in\mathbb{C}^n$, we have
	$$ P(J)=c J_1^{\alpha_1}\cdots J_n^{\alpha_n} + Q(J).$$
\end{Def}

\vspace{\baselineskip}
With this notation \ref{notation_J_[z,y]}, the formula (\ref{Sum_product_formula}) of the Proposition \ref{Green_decomposition_along_a_geodesic} is rewritten, with all variables defined:
\begin{equation}\label{Green_function_decomposition_simplified}
	G_z(a,b;\mathcal{B}(y)^\complement) = (v_z)_{[x,y]}(a,b).
\end{equation}

Note: In the rest of this paper, we will \textbf{always} take $x$ to be equal to $c$ in formulas (\ref{Formula_J_[z,y]}) / (\ref{Formula_J_[z,y]_bis}).
\vspace{\baselineskip}

\underline{Definition of $\psi$}:
We define $\psi:\mathcal{F}(\Xi,\mathbb{C})\to\mathcal{F}(\Xi,\mathbb{C})$ as the map that associates to any function  $J\in\mathcal{F}(\Xi,\mathbb{C})$, the complex valued function $\psi(J)\in\mathcal{F}(\Xi,\mathbb{C})$ defined by:
\begin{equation}\label{Def_de_psi}
	\forall (a,b)_{y}\in\Xi, \, \psi(J)((a,b)_{y})=p(a,b)+\sum_{c\in\mathcal{B}(y)^\complement}p(a,c) J_{[c,y]}(c,b).
\end{equation}

Since the transition kernel $p$ has finite range (that is to say, for any vertex $a\in X_0$, the set $\{c : p(a,c)>0\}$ is included in the ball $\mathcal{B}(a)$), the above sum is finite. Thus for any element $\theta\in\Xi$, the function $J\mapsto \psi(J)(\theta)$ depends on a finite number of coordinates, and is a polynomial function on the associated finite dimensional vector subspace.

A product $J_{x,y}(c_0,...,c_l)$, where for some vertex $c$ such that $p(a,c)>0$, and $(c_0,...,c_l)\in\Xi_{[c_0,y]}$ with $c_0=c$ and $c_l=b$, from the previous sum (\ref{Def_de_psi}) is by definition (See (\ref{Formula_J_[z,y]_bis}) and Definition \ref{Def_monomial} above) a \textit{monomial} of $J\mapsto \psi(J)((a,b)_{y})$.

\begin{Lem}
	The map $\psi$ leaves the subspace $E\subset\mathcal{F}(\Xi,\mathbb{C})$, of $\Gamma$-invariant complex valued functions over $\Xi$, stable: $$\psi(E)\subset E.$$
	
    In particular, $\psi$ is a polynomial map on $E$.
\end{Lem}
\begin{proof}
	Using the $\Gamma$-invariance of the transition kernel $p$ and of the functions in $E$, and by explicitly writing everything down, this is an immediate consequence of the definitions.
\end{proof}
%

We will note $\psi_E:=\psi\vert_E$, and sometimes just $\psi$, if there is no ambiguity, the restriction at the source of $\psi$ to the subspace $E\subset\mathcal{F}(\Xi,\mathbb{C})$.

Actually the map $\psi$ has been defined so that the next Proposition is satisfied:

\begin{Prop}[Proposition 2.5 \cite{Lalley_1993}]\label{Green_restricted_as_parametrisation_of_C}~

	In a neighbourhood of $z=0$, the map $z\mapsto v_z$ defined in (\ref{DEF:V_GREEN}) with values in $E$ satisfies the equation
	\begin{equation}
		z\psi(v_z)=v_z,
	\end{equation}
	for any complex number $z$ of modulus small enough.
\end{Prop}
\begin{proof}
	The point is to show that for any complex number $z$ in a neighbourhood of $z=0$, for any $(a,b)_y\in\Xi$ we have:
	$$ z\psi(v_z)((a,b)_y) = v_z((a,b)_y).$$
	Let $(a,b)_y\in\Xi$. We compute the left handside of the above using the definition of $\psi$ (\ref{Def_de_psi}), the formula (\ref{Green_function_decomposition_simplified}) from Proposition \ref{Green_decomposition_along_a_geodesic} and the writing of restricted Green's functions as a sum of $p$-admissible path weight (\ref{SumRestrictedWeight}):

	\begin{align*}
		z\psi(v_z)((a,b)_y)&=zp(a,b) + \sum_{c\in\mathcal{B}(y)^\complement} z p(a,c) (v_z)_{[c,y]}(c,b)\\
		&=zp(a,b) + \sum_{c\in\mathcal{B}(y)^\complement} z p(a,c) G_z(c,b;\mathcal{B}(y)^\complement)\\
		&=zp(a,b) + \sum_{c\in\mathcal{B}(y)^\complement} \sum_{\gamma\in\,\mathcal{P}h(c,b;\mathcal{B}(y)^\complement)}z p(a,c)\, w_z(\gamma)\\
		&=zp(a,b) + \sum_{c\in\mathcal{B}(y)^\complement} \sum_{\gamma\in\,\mathcal{P}h(c,b;\mathcal{B}(y)^\complement)} w_z((a,c)\ast\gamma),
	\end{align*}
	where the last equality comes from multiplicativity of the weight function $w_z$ with respect to path concatenation (See Definition  \ref{Def_function_path_weight}).
	But any path from $\mathcal{P}h(a,b;\mathcal{B}(y)^\complement)$ of length at least $2$ is uniquely written as a concatenation of the form $(a, c)\ast\gamma$, where $c\in\mathcal{B}(y)^\complement$ and $\gamma\in\mathcal{P}h(c,b;\mathcal{B}(y)^\complement)$ therefore, using again formula (\ref{SumRestrictedWeight}), we get
	\begin{align*}
		z\psi(v_z)((a,b)_y)&=\underbrace{zp(a,b)}_{\text{case }l(\gamma)=1} + \sum_{\substack{\gamma\in\,\mathcal{P}h(a,b;\mathcal{B}(y)^\complement)\\ l(\gamma)\geq 2} }w_z(\gamma)\\
		&= G_z(a,b;\mathcal{B}(y)^\complement).
	\end{align*}
\end{proof}

Consequently, the restricted Green's functions parametrize, in a neighbourhood of $0$ the algebraic curve defined below (See Lemma \ref{Lem_Green_parametrization_of_C}).
\vspace{\baselineskip}~

\underline{Definition of $\mathcal{C}$}:

Let $W$ be the following map:
\begin{equation}\label{Def_de_W}
	\begin{split}
 		W:\mathbb{C}\times E &\to E\\
 		(z,J)&\mapsto J-z\cdot\psi(J)
	\end{split}
\end{equation}
The map $W$ is a polynomial map on $\mathbb{C}\times E$. Finally, we define $\mathcal{C}$ to be the affine algebraic set
\begin{equation}\label{Def_de_C}
	\mathcal{C}:= W^{-1}(\lbrace 0_E \rbrace )=\{(z,J)\in \mathbb{C}\times E :\, J=z\psi(J)\}.
\end{equation}

At $(z,J)=(0,0_E)$, we have $W(0,0_E)=0_E$ and the partial derivative of $W$ with respect to the variable $J$ satisfies, $\dfrac{\partial W}{\partial J}(0,0_E)=Id_E$, in particular, by the Implicit Functions Theorem (See \cite{Cartan_1961}, p.138), the origin $(0,0_E)$ is a regular point of $W$ and therefore, $\mathcal{C}$ is a smooth analytic curve in the neighbourhood of $(0,0_E)$. More precisely, the complex Implicit Functions Theorem allows us to assert the existence of a germ of holomorphic map in the neighbourhood of $z=0$ in $\mathbb{C}$, with values in $E$, which we will denote $v:z\mapsto v_z$ such that for any complex number $z$ of modulus small enough and any vector $w$ in a neighbourhood of $0_E$ in $E$,
\begin{equation}\label{CIFT_3.11}
	(z,w)\in\mathcal{C}\Leftrightarrow w=v_z.
\end{equation}
The lemma bellow tells us that the notation \og $v_z$ \fg{} is consistent with the previously chosen one (See example \ref{example_gamma_inv_green}). 
\begin{Rem}
   At this point we only know that $\mathcal{C}$ is an affine algebraic set. Lalley proved in \cite{Lalley_1993} (Proposition 3.1) that the irreducible components of $\mathcal{C}$ containing the origin are indeed curves. We thus take the liberty of calling $\mathcal{C}$ an algebraic curve.

In \cite{LCurve}, we prove that $\mathcal{C}$ is a smooth Riemann surface in a neighbourhood of the closure of the range of the map $u:z\mapsto (z,v_z)$.
\end{Rem}

\begin{Lem}\label{Lem_Green_parametrization_of_C}
	The germ of holomorphic map in the neighbourhood of $z=0$, with values in $E$, $z\mapsto v_z$ given by the formula:
	$$\forall (a,b)_y\in\Xi, \, v_z((a,b)_y)=G_z(a,b;\mathcal{B}(y)^\complement)$$
	is the unique map $v$ introduced in (\ref{CIFT_3.11}).
\end{Lem}
\begin{proof}
	This is an immediate consequence of the uniqueness of $v$ given by the complex Implicit Functions Theorem and the previous Proposition \ref{Green_restricted_as_parametrisation_of_C}.
\end{proof}

We denote by $R_G$\index{$R(a,b;\Omega)$}\index{$R_G$} the possibly infinite radius of the largest disk on which $v:z\mapsto v_z$ admits a holomorphic extension. Actually, it is elementary to check that $R_G=\min \{ R(a,b;\mathcal{B}(y)^\complement) : (a,b)_y\in\Xi \}$, where $R(a,b;\mathcal{B}(y)^\complement)$ is the radius of convergence of the power series expansion of the restricted Green's function $G_z(a,b;\mathcal{B}(y)^\complement)$, in the neighbourhood of $z=0$ (Definition \ref{Def_Green_Restricted}). 

Let $R$ be the radius of convergence of the Green's functions \og $G_z(x,y)$ \fg{}, thereby, it verifies the Cauchy Hadamard formula: 
$$\frac{1}{R}=\limsup_{n\to\infty}\left( p^{(n)}(x,y)^{\frac{1}{n}}\right),$$
for any vertices $x$ and $y$ of $X$. Note that $R_G\geq R$, we will see in the next section that we actually have equality $R_G=R$.

In this text we are interested in the regularity of $\mathcal{C}$, in a neighbourhood of the closure noted $\overline{\{(z,v_z):z\in \mathbb{D}(0,R_G)\}}$, as well as the behaviour of $v$ on the boundary $\partial \mathbb{D}(0,R_G)$. These properties are studied in detail in subsection \ref{LCurve-subsection_regularity_of_C} in \cite{LCurve}, see more specifically Proposition \ref{LCurve-Regularity_of_C} in \cite{LCurve}.

\begin{Rem}\label{real_local_inversion_for_u}
	$\psi$ is a polynomial map leaving stable $\mathcal{F}(\Xi,\mathbb{R})^\Gamma$, the space of real-valued $\Gamma$-invariant functions on $\Xi$, so $W$ is also a polynomial map leaving stable $\mathbb{R}\times\mathcal{F}(\Xi,\mathbb{R})^\Gamma$. The real Implicit Functions Theorem therefore also applies. We find again, without using Green's function that, if $r$ is real, then $v_r$ belongs to $\mathcal{F}(\Xi,\mathbb{R})^\Gamma$.
\end{Rem}

	\subsection{Action of \texorpdfstring{$\mathbb{Z}/d\mathbb{Z}$}{Z/dZ} on \texorpdfstring{$\mathcal{C}$}{C}}\label{Subsection_ZdZ_action}
	Motivated by Corollary \ref{Cor_operateur_diagonal} 	we define the diagonal operator $A$ over $E$: 
	\begin{equation}\label{Def_of_A}
			\begin{split}
			A\colon\mathcal{F}(\Xi, \mathbb{C})  & \longrightarrow \mathcal{F}(\Xi, \mathbb{C})\\
				1_{(a,b)_{y}} &\longmapsto \zeta_d^{r(a,b)}1_{(a,b)_{y}},
			\end{split}		
	\end{equation}

		where $\zeta_d=e^{2i\pi/d}$, and $r:X_0\times X_0\to \mathbb{Z}/d\mathbb{Z}$ is the periodicity cocycle (see Definition \ref{DEF:PERIODICITY_COCYCLE}).

By elementary computation, using the cocycle property of the periodicity cocycle (See Proposition \ref{Prop_Computation_for_the_Zd_action}, for the details), we can show that 
\begin{equation}
	\zeta_d \cdot  \psi\circ A = A \circ \psi
\end{equation}
In particular for any complex number $z$ and any function $J$ in $E$, 
\begin{equation} \label{Relation_Zd}
   AJ-(\zeta_d z) \psi(AJ)=A(J-z\psi(J)).
\end{equation}
Thus if $(z,J)$ belongs to $\mathcal{C}$, then so does $(\zeta_d z, A J)$, and reciprocally. We deduce that the map $(z,J)\mapsto (\zeta_d z, AJ)$ is an automorphism of $\mathcal{C}$ that we also denote by $A$. 

Note that if $u$ denotes the germ of holomorphic map $u:z\mapsto (z,v_z)$ with values in $\mathcal{C}$, then we have for any $z\in\mathbb{D}(0,R_G)$, 
\begin{equation}\label{Formula_t1}
 u(\zeta_d z)=A u(z).
\end{equation}
Which is simply a consequence of the definitions of $A$ and $r$ (see Corollary \ref{Cor_operateur_diagonal}).
	
	Hence $A\in Aut(\mathcal{C})$ is an automorphism of order $d$ ($A^1\neq Id_\mathcal{C}$, ..., $A^{d-1}\neq Id_\mathcal{C}$ and $A^d=Id_\mathcal{C}$). It induces a $\mathbb{Z}/d\mathbb{Z}$-action on $\mathcal{C}$ such that, the formula (\ref{Formula_t1}) holds. This property will be essential for computing asymptotics for non-aperiodic random walks on trees (See Theorem \ref{Thm_Cor_Black_box_general}, or more exhaustively Subsection \ref{RTaub-Subsection_Extension_to_ZdZ_action} in \cite{RTaub}).
	
\subsection{Proof of \texorpdfstring{$\zeta_d \cdot  \psi\circ A = A \circ \psi$}{equivariance formula}}
		\label{Computation_for_the_Zd_action}
			\begin{Prop}\label{Prop_Computation_for_the_Zd_action}~
		The polynomial map $\psi$ (see definition \ref{Def_de_psi}) satisfies:
		\begin{equation}\label{Relation_psi_Z_d}
			\zeta_d \cdot  \psi\circ A = A \circ \psi,
		\end{equation}
		where $\zeta_d=e^{2i\pi/d}$ and $A$ is the diagonal operator on $E=\mathcal{F}(\Xi,\mathbb{C})^\Gamma$ defined by:
		\begin{equation*}
			\begin{matrix}{}
				A:\mathcal{F}(\Xi, \mathbb{C})&\to \mathcal{F}(\Xi, \mathbb{C})\\
				1_{(a,b)_{y}}&\mapsto \zeta_d^{r(a,b)}1_{(a,b)_{y}}.
			\end{matrix}			
		\end{equation*}
	\end{Prop}

	\begin{proof}
		Let $(a,b)_{y}$ be in $\Xi$ and let $c$ be in $\mathcal{B}(y)^\complement$. Using the cocycle property of $r$, we get:
		\begin{equation}\label{eq-t12}
			\forall J\in \mathcal{F}(\Xi,\mathbb{C}),\, (AJ)_{[c,y]}(c,b)=\zeta_d^{r(c,b)}(J)_{[c,y]}(c,b)
		\end{equation}
Then we have by definition of $\psi$ that for any function $J\in \mathcal{F}(\Xi,\mathbb{C}),$
		\begin{align*}
			\psi(AJ)(a,b)_{y}&= p(a,b) + \sum_{c\in\mathcal{B}(y)^\complement} p(a,c) (AJ)_{[c,y]}(c,b)\\
			&= p(a,b) + \sum_{c\in\mathcal{B}(y)^\complement} p(a,c) \zeta_d^{r(c,b)}(J)_{[c,y]}(c,b)\\
			&=p(a,b) + \sum_{c\in\mathcal{B}(y)^\complement} \zeta_d^{r(a,c)+r(c,b)-1}p(a,c) (J)_{[c,y]}(c,b)\\
			&= p(a,b) + \zeta_d^{r(a,b)-1}\sum_{c\in\mathcal{B}(y)^\complement} p(a,c) (J)_{[c,y]}(c,b)\\
			&=\zeta_d^{-1} \left(\zeta_d\, p(a,b) + \zeta_d^{r(a,b)}  \sum_{c\in\mathcal{B}(y)^\complement}p(a,c) (J)_{[c,y]}(c,b)\right)
		\end{align*}	

		Two cases are possible:
		Either $p(a,b)>0$ in which case, $r(a,b)=1$ and then :
		\begin{align*}
			\forall J\in\mathcal{F}(\Xi,\mathbb{C}), \psi(AJ)((a,b)_{y})
			&=\zeta_d^{-1}(\zeta_d^{r(a,b)}\psi(J)((a,b)_{y})\\
			&=\zeta_d^{-1}(A\circ\psi)(J)((a,b)_{y}).
		\end{align*}
		Or $p(a,b)=0$ in which case:
		\begin{align*}
			\forall J\in\mathcal{F}(\Xi,\mathbb{C}),\, \psi(AJ)(a,b)_{y}&=\zeta_d^{-1} \left( \sum_{c\in\mathcal{B}(y)^\complement}\zeta_d^{r(a,b)} p(a,c) (J)_{[c,y]}(c,b)\right)\\
			&=\zeta_d^{-1}(A\circ\psi)(J)((a,b)_{y}).
		\end{align*}
		Since $(a,b)_{y}$ is arbitrary, we have that for any function $J\in\mathcal{F}(\Xi,\mathbb{C}),$
		$$\psi(AJ)((a,b)_{y})=\zeta_d^{-1} (A\circ\psi(J))((a,b)_{y}),$$
		as desired.
	\end{proof} 
	
%

\vspace{\baselineskip}

\section{Green functions on a tree}\label{Section_Green_function_in_a_tree}

Under our assumptions, that we recall below, the Green's functions possess particular properties: they are algebraic functions\footnote{A function $f$ is said to be an \textit{algebraic function of $z$}, or simply \textit{an algebraic function} if it is algebraic over $\mathbb{C}[z]$, i.e if there exist a non-zero polynomial function $Q\in\mathbb{C}[Z,F]$ such that $Q(z,f(z))\equiv 0$ for any $z$.} of $z$.
Steven P.Lalley proved it in \cite{Lalley_1993} in the case where $X$ is a regular tree and the group $\Gamma$ acts transitively on the set of vertices. His proof totally adapts in our context.

 Following his work, we proved in the previous section that restricted Green's functions of the form \og $G_z(a,b;\mathcal{B}(y)^\complement)$ \fg{} are algebraic functions of $z$ and parametrize the algebraic curve $\mathcal{C}$, in a neighbourhood of the origin (Proposition \ref{Green_restricted_as_parametrisation_of_C}). In this section we adapt the computations of S. P. Lalley showing that Green's functions \og $G_z(x,y)$ \fg{}, and first-passage generating functions \og $F_z(x,y)$ \fg{} are rational functions of the previous restricted Green's function and of the variable $z$ (See Proposition \ref{algebraicity_of_F} and Corollary \ref{Cor_algbraicity_of_Greens_function}), and thus are also algebraic functions of $z$.

	\subsection{Algebraicity of Green functions}

	\begin{Prop}\label{algebraicity_of_F}
	Let $X=(X_0,X_1)$ be an infinite tree of bounded valence such that any vertex possesses at least three neighbours, $\Gamma<\operatorname{Aut}(X)$ be a subgroup of automorphisms of the tree such that $\Gamma\backslash X_0$ is finite, $p:X_0\times X_0\to [0,1]$ be a finite range $\Gamma$-invariant transition kernel, making $(X_0,p)$ an irreducible Markov chain and let $k\in\mathbb{N}$ be an integer such that 
		\begin{equation}\label{Property_over_Kernel_rec}
		\forall x,y\in X_0,\, d(x,y)>k\Rightarrow p(x,y)=0.
		\end{equation}
		Let $z\mapsto v_z$ be the map defined in (\ref{DEF:V_GREEN}) that parametrizes the Lalley's curve $\mathcal{C}$ (Proposition \ref{Green_restricted_as_parametrisation_of_C}).
	
	For any pair $(x,y)$ of vertices of $X$, there exists a rational function $f_{x,y}:\mathbb{C}\times E \to \mathbb{C}$ such that, $f_{x,y}$ is regular near the origin and for any complex number $z$ of modulus sufficiently small,
		\begin{equation}\label{Eq_Green_F}
			F_z(x,y)=f_{x,y}(z,v_z)
		\end{equation}
	\end{Prop}	
	
\begin{proof}
	We suppose $x\neq y$ otherwise $f_{x,y}\equiv 1$ convene.
	We start by constructing the desired rational function for pairs $(c_0,y)$ such that, $y$ is any vertex of $X$ and $c_0$ is a vertex different from $y$ that is in the ball $\mathcal{B}(y)$. 
		
		Using the expression of $F_z(c_0,y)=G_z(c_0,y;\{y\}^\complement)$ as a sum of $p$-admissible path weights (\ref{SumRestrictedWeight}), and decomposing $p$-admissible paths of $\mathcal{P}h(c_0,y;\{y\}^\complement)$ by their first step, we get for any complex number $z$ of modulus sufficiently small,
		\begin{equation}\label{eq-t8}
			F_z(c_0,y)= z p(c_0,y) + \sum_{c_1\neq y} zp(c_0,c_1) F_z(c_1,y).
		\end{equation}
		In the previous sum, we focus on the term \og $F_z(c_1,y)$\fg{}. By decomposing $p$-admissible paths in the set $\mathcal{P}h(c_1,y;\{y\}^\complement)$, by its first passage in $\mathcal{B}(y)$, we get for $c_1\notin\mathcal{B}(y)$:
		\begin{equation}\label{eq-t9}
			F_z(c_1,y)= G_z(c_1,y;\mathcal{B}(y)^\complement) +\sum_{\substack{d_1\in\mathcal{B}(y)\\ d_1\neq y}} G_z(c_1,d_1;\mathcal{B}(y)^\complement)F_z(d_1,y).
		\end{equation}
		In this equality, the term \og $G_z(c_1,y;\mathcal{B}(y)^\complement)$\fg{} is $0$. Indeed, recall that $\mathcal{B}(y)=\mathcal{B}_k(y)$, where $k$ is an integer such that (\ref{Property_over_Kernel_rec}) holds, therefore $d(y,\mathcal{B}(y)^\complement)=k+1$, thus for any vertex $w\in\mathcal{B}(y)^\complement$, we must have $p(w,y)=0$ and so, for any $c_1\in\mathcal{B}(y)^\complement$, we must have
		$$G_z(c_1,y;\mathcal{B}(y)^\complement)=0.$$
		
		Injecting (\ref{eq-t9}) into (\ref{eq-t8}), we obtain,
		\begin{align*}
			F_z(c_0,y)
			&=z p(c_0,y) + \sum_{\substack{c_1\in\mathcal{B}(y)\\c_1\neq y}} z p(c_0,c_1) F_z(c_1,y)\\
			&~\quad+\sum_{c_1\notin\mathcal{B}(y)}\sum_{\substack{d_1\in\mathcal{B}(y)\\d_1\neq y}} z p(c_0,c_1)G_z(c_1,d_1;\mathcal{B}(y)^\complement)F_z(d_1,y).
		\end{align*}
		Performing the index change $d_1=c_1$ in the first sum of the last equality, we obtain for all $c_0\in\mathcal{B}(y)\setminus\{y\}$,
		\begin{equation*}
			F_z(c_0,y)=z p(c_0,y) + \sum_{\substack{d_1\in\mathcal{B}(y)\\d_1\neq y }} \left(z p(c_0,d_1) + \sum_{c_1 \notin \mathcal{B}(y)} z p(c_0,c_1) G_z(c_1,d_1;\mathcal{B}(y)^\complement)\right) F_z(d_1,y).
		\end{equation*}
		
		For $c,d\in\mathcal{B}(y)\setminus\{y\}$, and for $J\in E$, we define:
		\begin{gather}
			M_J(c,d):= p(c,d) + \sum_{c_1\notin\mathcal{B}(y)} p(c,c_1)(J)_{[c_1,y]}(c_1,d)\\
			\text{ and }\\
			\bar{p}\in [0,1]^{\mathcal{B}(y)\setminus\{y\}}, \, \bar{p}(c):=p(c,y),
		\end{gather}

		so that using formula (\ref{Green_function_decomposition_simplified}) we get the matrix formula for first-passage generating functions,
		\begin{equation}
			\forall z\in \mathbb{D}(0,R),\; F_z(\,.\,,y)=z \bar{p} + z M_{v_z} F_z(\,.\, ,y)
		\end{equation}
		where $F_z(\,.\,,y)$ is the vector in $\mathbb{C}^\mathcal{B}(y)\setminus\{y\}$ with coordinates $F_z(c,y)$, for $c\in\mathcal{B}(y)\setminus\{y\}$.
Note that $M_J$ and $\bar{p}$ depend on $y$, and if we identify the space of square matrices indexed by $\mathcal{B}(y)\setminus\{y\}$ and the space of complex-valued vectors indexed by $\mathcal{B}(y)\setminus\{y\}$, to some $\mathcal{M}_n(\mathbb{C})$ and $\mathbb{C}^n$, then $M_J$ and $\bar{p}$ will actually depend on the class $[y]\in\Gamma\backslash\Xi$ of $y$ modulo $\Gamma$.

		Finally, for $z=0$, $v_0=0\in E$ and $M_0=\left( p(c,d) \right)_{c,d}$ We infer that in a neighbourhood of $z=0$, the operator $(I-zM_{v_z})$ is invertible. And in this same neighbourhood we have over $\mathcal{B}(y)\setminus \{y\}$:
		\begin{equation}\label{Egalite_Green_premier_passage}
	 			F_z(\,.\, ,y)=\left(I-z M_{v_z}\right)^{-1}(z p(\,.\, ,y))
	 	\end{equation} 
		 In particular, for any $c\in\mathcal{B}(y)\setminus\{y\}$, there exists a rational function $f_{c,y}: \mathbb{C}\times E \to\mathbb{C}$ regular near the origin such that for any complex number $z$ in a neighbourhood of $0$, we obtain as wished,
		\begin{equation}\label{equation_temp_5}
			F_z(c,y)=f_{c,y} (z,v_z).
		\end{equation}

We now prove the proposition for any pair $(x,y)$ of vertices of $X$:
\begin{Clm}\label{Claim_3}
	For any pair of vertices $x,y$ of $X$, $x\neq y$, there exists a polynomial function $Q_{x,y}: E\times\mathbb{C}^{\mathcal{B}(y)\setminus\{y\}}\to\mathbb{C}$ depending on $(x,y)$ such that for any complex number $z$ in a neighbourhood of $0$ in $\mathbb{C}$,
		\begin{equation}\label{equation_stage_3}
			F_z (x,y)=Q_{x,y}(v_z,\Big(F_z(\,.\,,y))\Big).
		\end{equation} 
\end{Clm} 
The result is an immediate consequence of this claim - that we prove below -  and of (\ref{equation_temp_5}) for pairs $(c,y)$, where $c$ belongs to $\mathcal{B}(y)\setminus \{y\}$, that we just proved.
\end{proof}
\begin{proof}[Proof of Claim \ref{Claim_3}.]
		If $x\in\mathcal{B}(y)\setminus\{y\}$, just take $Q_{x,y}:(J,(F_c)_c)\mapsto F_x$.
		Otherwise, decomposing any $p$-admissible path of $\mathcal{P}h(x,y;\{y\}^\complement)$ by its first passage through $\mathcal{B}(y)$, we obtain the following equation for Green's functions: For any complex number $z$ of modulus small enough,
		\begin{equation*}
			F_z(x,y)=\sum_{c\in\mathcal{B}(y)\setminus\{y\}} G_z(x,c;\mathcal{B}(y)^\complement) F_z(c,y)
		\end{equation*}
		In addition, we have shown (See formula (\ref{Green_function_decomposition_simplified})) that for any complex number $z\in \mathbb{D}(0,R)$, and any vertices $x\notin\mathcal{B}(y)$, $c\in\mathcal{B}(y)$,
		\begin{equation*}
			G_z(x,c;\mathcal{B}(y)^\complement)=(v_z)_{[x,y]}(x,c).
		\end{equation*}
		We then define:
		\begin{equation}\label{Def_de_f_p}
			Q_{x,y}:=(J,(F_c)_c)\mapsto\sum_{c\in\mathcal{B}(y)\setminus\{y\}} (J)_{[x,y]}(x,c) F_c,
		\end{equation}
		and we immediately check that it satisfies equality (\ref{equation_stage_3}).
\end{proof}

In view of formulas $i)$ and $iii)$ of Proposition \ref{SurMult}, we get the Corollary:

	\begin{Cor}\label{Cor_algbraicity_of_Greens_function}
		Under the assumptions of Proposition \ref{algebraicity_of_F}, for any pair $(x,y)$ of vertices of $X$, there exists a rational function $g_{x,y}:\mathbb{C}\times E \to \mathbb{C}$ such that, $g_{x,y}$ is regular near the origin and for any complex number $z$ of modulus sufficiently small, we have
		\begin{equation}\label{Eq_Green_G}
		 G_z(x,y)=g_{x,y}(z,v_z).
		\end{equation}	
	\end{Cor}

\begin{Rem}
	Since the rational functions $f_{x,y}$ and $g_{x,y}$ for $x,y$ vertices of $X$ are regular in a neighbourhood of zero, their restriction to the Lalley's curve $\mathcal{C}$ are rational functions over $\mathcal{C}$. So Green's functions $G_z(x,y)$, $F_z(x,y)$ and $G_z(a,b;\mathcal{B}(y)^\complement)$, $(a,b)_y\in\Xi$, are rational functions over $\mathcal{C}$. 
	
	Besides the rational functions $g_{x,y}$ and $f_{x,y}$ admit a power series expansion in the neighbourhood of zero with non-negative coefficients (This is proven in  Proposition \ref{LCurve-Prop_PSeries_expansion} in \cite{LCurve}).
\end{Rem}	

\subsection{Green function and the radius of convergence}
We work under the assumptions of Proposition \ref{algebraicity_of_F}. Here, we extend the two equalities (\ref{Eq_Green_F}) and (\ref{Eq_Green_G}) to any complex number  $z\in \mathbb{D}(0,R)$ and we discuss briefly the radius of convergence of Green's functions \og $G_z(a,b;\mathcal{B}(y)^\complement), F_z(x,y)$ and $G_z(x,y)$ \fg.

For that we will need the theorem \ref{Finitude_of_Green_functions}, recalled below, about finiteness of Green's functions at the radius of convergence and we will need for the Lalley's curve $\mathcal{C}$ to be a smooth analytic curve in a neighbourhood of $\overline{\{(z,v_z): z\in\mathbb{D}(0,R)\}}\subset\mathcal{C}$. 

It is not immediate that $\mathcal{C}$ is a smooth complex curve in a neighbourhood of $\overline{\{(z,v_z): z\in\mathbb{D}(0,R)\}}$. We prove it in section \ref{LCurve-Section Structure of C} (see Proposition \ref{LCurve-Regularity_of_C}) in \cite{LCurve} and we admit it for the rest of this section.

\begin{Thm*}\ref{Finitude_of_Green_functions}(See \cite{Woess_2000} Chapter 2.)~
			
			For any finitely supported irreducible Markov chain $(\mathcal{X}_0,p)$ on an infinite tree such that all of its vertices has at least three neighbours, if $\mathcal{R}=\left(\limsup_{n\to\infty} p^{(n)}(x,y)\right)^{-1}$ is the radius of convergence of Green's functions $z\mapsto G_z(x,y)$. Then
			$$\forall x,y\in \mathcal{X}_0, \; G_{\mathcal{R}}(x,y)<\infty$$
			
			The same applies to first-passage generating functions $F_{R}(x,y)$, and to restricted Green's functions $G_R(x,y;\Omega)$ for any subset $\Omega\subset \mathcal{X}_0$.
		\end{Thm*}
		
	\begin{Cor}\label{Continuity_extension_of_Greens_functions}
		The Green's functions can be continuously extended to the closed disk $\overline{\mathbb{D}}(0,R)$.
	\end{Cor}

We now extend the equalities (\ref{Eq_Green_F}) and (\ref{Eq_Green_G}) to any complex number $z$ in the closed disk $\overline{\mathbb{D}}(0,R)$.
	
\begin{Lem}\label{Lem_sing_ratio}
		Let $f\in\mathbb{C}(\mathcal{C})$ be a rational function on a smooth algebraic curve $\mathcal{C}$ and let $p$ be an element of $\mathcal{C}$. If $f(p)$ is finite, then there exist a neighbourhood of $p$ on which $f$ is bounded.
	\end{Lem}
	\begin{proof}
	Without lost of generality we can suppose that $\mathcal{C}=\mathbb{C}$. If $p$ is a pole of $f$ then $\lim_{z\to p}|f(z)|=+\infty$, thus if $f(p)$ is finite by continuity of $f$, there must be a neighbourhood of $p$ in $\mathbb{C}$ on which $f$ is bounded.
\end{proof}		
	
	\begin{Prop}\label{Extension_of_algebraicity_say}
	Let $(x,y)$ be a pair of vertices in $X$. The rational functions $f_{x,y}$ and $g_{x,y}$ from Proposition \ref{algebraicity_of_F} and Corollary \ref{Cor_algbraicity_of_Greens_function}, when restricted to $\mathcal{C}$ are rational functions on $\mathcal{C}$ that are regular in a neighbourhood of $\overline{\{(z,v_z): z\in \mathbb{D}(0,R)\}}$.
	\end{Prop}
	\begin{proof}
		By Proposition \ref{LCurve-Regularity_of_C} from \cite{LCurve}, the curve $\mathcal{C}$ is smooth in a neighbourhood of the closed set $\overline{\{(z,v_z): z\in\mathbb{D}(0,R)\}} \subset\mathcal{C}$. Thus from Lemma \ref{Lem_sing_ratio}, the functions $f_{x,y}$ and $g_{x,y}$ in $\mathbb{C}(\mathcal{C})$ are regular in a neighbourhood of the compact set $\overline{\{(z,v_z): z\in\mathbb{D}(0,R)\}}$ if, and only if, they are bounded on $\overline{\{(z,v_z): z\in\mathbb{D}(0,R)\}}$. 
		But this last property is verified thanks to the equalities (\ref{Eq_Green_F}) and (\ref{Eq_Green_G}), extended by analytic continuation and thanks to Corollary \ref{Continuity_extension_of_Greens_functions}.
	\end{proof}

We now discuss the radius of convergence of Green's functions and restricted Green's functions.	For $(x,y)$ a pair of vertices of $X$, let $R_F(x,y)$\index{$R_F(x,y)$} be the radius of convergence of the first-passage generating functions $F_z(x,y)$. For $(a,b)_y\in\Xi$, let $R(a,b;\mathcal{B}(y)^\complement)$ denotes the radius of convergence of the restricted Green's function $G_z(a,b;\mathcal{B}(y)^\complement)$. Denote by $R_F:=\min_{x,y}(R_F(x,y))$\index{$R_F$} the minimal radius of convergence of the first-passage generating functions $F_z(x,y)$ for $(x,y)\in X_0\times X_0$, and recall that $R_G=\min_{(a,b)_y}(R(a,b;\mathcal{B}(y)^\complement))$ is the minimal radius of convergence of the restricted Green's functions $G_z(a,b;\mathcal{B}(y)^\complement)$, for $(a,b)_y\in\Xi$.

Since $p^{(n)}(x,y)\geq p^{(n)}(x,y;\{y\}^\complement)$ and $
	p^{(n)}(a,b;\{b\}^\complement)\geq p^{(n)}(a,b; B_k(y)^\complement)$, for any positive real number $r>0$, we have $G_r(x,y)\geq F_r(x,y)$ and $F_r(a,b) \geq G_r(a,b;B(y)^\complement)$, so we necessarily have $$R\leq R_F\leq R_G.$$ 
	
	We wonder if we have equality. Recall that $R$ is the radius of convergence of Green's functions \og $G_z(x,y)$ \fg{} and does not depend on $(x,y)$ (See Proposition \ref{Prop_Markov_Ired_et_rayon_de_cv}). By the Cauchy-Hadamard formula we have 
	$$R=\frac{1}{\limsup p^{(n)}(x,y)^{1/n}},$$
whenever $(x,y)\in X_0\times X_0$.
	
	By theorem \ref{Finitude_of_Green_functions} we also have that $G_z$ is finite for any complex number $z$ of modulus $|z|=R$, so are $F_z(x,y)$ and $G_z(a,b;\mathcal{B}(y)^\complement)$.
	
	 Insofar as we got for any complex number $z\in\mathbb{D}(0,R)$, the equality $$G_z(x,y)=g_{x,y}(z,v_z),$$ from Proposition \ref{Extension_of_algebraicity_say}, if $R_G>R$, then the right handside of the above equality would admit a holomorphic extension on a disk of strictly greater radius than $R$. Which contradicts the definition of $R$. Hence $R=R_G$, and
\begin{equation}\label{equality_radius}
	R=R_F=R_G.
 \end{equation}
 
\begin{Rem}
	Note that we used the smoothness of $\mathcal{C}$ to prove that the rational functions $f_{x,y}$ and $g_{x,y}$ are regular in a neighbourhood of the closure $\overline{\{(z,v_z): z\in\mathbb{D}(0,R)\}}$ and to prove equality (\ref{equality_radius}). The regularity of $f_{x,y}$ and $g_{x,y}$ near $\overline{\{(z,v_z): z\in\mathbb{D}(0,R)\}}$ will not intervene in the proof of the smoothness of $\mathcal{C}$. Nevertheless the real number $R$ will be mentioned, but the equality (\ref{equality_radius}) will not play a role in the proof of the smoothness of $\mathcal{C}$ (Proposition \ref{LCurve-Regularity_of_C} in \cite{LCurve}) either.
\end{Rem}

We give in \cite{LCurve} (Proposition \ref{LCurve-Corollary_PSeries_A} and Proposition \ref{LCurve-Charac_radius} respectively) a characterisation of the pairs $(x,y)\in X_0\times X_0$ and triplets $(a,b)_y\in\Xi$ such that the radius of convergence $R_F(x,y)$ and $R(a,b;\mathcal{B}(y)^\complement)$ are strictly greater than $R$, in term of $p$-admissible path of $\mathcal{P}h(x,y;\{y\}^\complement)$ and $\mathcal{P}h(a,b;\mathcal{B}(y)^\complement)$. This characterisation is foreshadowed in the next section.

\section{Asymptotic of probability}\label{Section_Asymptotics}

In this section, we will heavily rely on the Thm \ref{Thm_Cor_Black_box_general} from  \cite{RTaub}(Corollary \ref{RTaub-Cor_Black_box_general}) that we state below. This is a Tauberian theorem that focuses on power series expansion of holomorphic functions with square root singularity at their radius of convergence.

\begin{Thm}[Corollary \ref{RTaub-Cor_Black_box_general} in \cite{RTaub}]\label{Thm_Cor_Black_box_general}
Let $R>0$ be a real number, $d$ be a positive integer, $\mathcal{C}$ be a smooth complex curve (i.e. a Riemann surface or a complex manifold of dimension 1) and $A\in Aut(\mathcal{C})$ be an automorphism of $\mathcal{C}$ of order $d$ (i.e such that $A,...,A^{d-1}\neq Id_\mathcal{C}$, and $A^d=Id_\mathcal{C}$). Let $u:\mathbb{D}(0,R)\to\mathcal{C}$ and $\lambda:\mathcal{C}\to\mathbb{C}$ be two holomorphic maps, and let $g$ be a complex valued function defined and holomorphic in a neighbourhood of the closure $\overline{\{ u(z): z\in\mathbb{D}(0,R)\}}\subset\mathcal{C}$. Denote by $\sum_{n\in\mathbb{N}} a_n z^n$ the power series expansion of the complex valued function $g\circ u$ near $z=0$. Assume that
	\begin{enumerate}[label=\roman*)]
		\item $\lambda\circ u=id_{\mathbb{D}(0,R)}$;
		\item $u$ extends by continuity into a continuous map $\overline{\mathbb{D}(0,R)}\to\mathcal{C}$;
		\item The first derivative of $\lambda$ at $p=u(R)$ is zero: $D_p\lambda=0$;
		\item The second derivative of $\lambda$ at $p$ is non-zero 
			: $D^2_p \lambda\neq 0$;
		\item $u$ extends holomorphically to the neighbourhood of any point on the boundary of $\mathbb{D}(0,R)$ minus the $d$-th roots of $R^d$: $\partial\mathbb{D}(0,R)\setminus\left\lbrace R(e^{2i\pi/d})^k: k\in \{0,...,d-1\}\right\rbrace;$
		\item The derivative of $g$ at $p$ is non-zero: $D_p g\neq 0$.
		\item $Fix(A^1)=...=Fix(A^{d-1})=\{u(0)\}$ and for any $ z\in\mathbb{D}(0,R)$, $$u(e^{2i\pi/d}z)=Au(z);$$
		\item There is an integer $r\in\{0,...,d-1\}$ such that for any $q\in \overline{\{ u(z): z\in\mathbb{D}(0,R)\}}$, $$g(Aq)=e^{2i\pi r/d}g(q).$$
	\end{enumerate}
	 Then, for any non-negative integer $n$ such that $n\notequiv r \, [d]$, we have $a_n=0$ and there exists a non-zero positive constant $C$, and a unique sequence of constants $(c_l)_{l\geq 1}$, such that, as $n$ goes to $\infty$, we have the asymptotic expansion:
	\begin{equation}
	    a_{dn+r}\sim_\infty C R^{-dn} \frac{1}{n^{3/2}}\left( 1 + \sum_{l=1}^\infty \frac{c_l}{n^{l}}\right ).
	\end{equation}
\end{Thm}
	To obtain the asymptotic expansion of the sequences of probabilities $(p^{(n)}(x,y))_n$, $(p^{(n)}(x,y;\{y\}^\complement))_n$ and $(p^{(n)}(a,b;\mathcal{B}(y)^\complement))_n$, we will apply this theorem with $\mathcal{C}$ being the Lalley's curve, $u:\mathbb{D}(0,R)\to\mathcal{C}$ the map that associates to any complex number $z\in\mathbb{D}(0,R)$, the point
\begin{equation}\label{Def_u}	
	u(z)=(z,v_z);
\end{equation} 
$\lambda:\mathcal{C}\to\mathbb{C}$, the map that sends any point $(z,J)$ of $\mathcal{C}$ to its first coordinate, 
\begin{equation}\label{Def_lambda}
	\lambda(z,J)=z;
\end{equation}
and different functions $g$ for which we will have to prove that its derivative at $u(R)$ is non-zero. We start by reviewing the assumptions of the previous theorem:

Note that it is clear from definitions that 
\begin{equation}
	\lambda\circ u = Id_{\mathbb{D}(0,R)}.
\end{equation}	
We proved that $u$ can be continuously extended to the closed disk $\overline{\mathbb{D}}(0,R)$, in Corollary \ref{Continuity_extension_of_Greens_functions}, besides, since $\mathcal{C}$ is closed, the values taken by $u$ on $\partial\mathbb{D}(0,R)$ are still in $\mathcal{C}$. We have also seen in subsection \ref{Subsection_ZdZ_action}, that there exists a diagonal operator $A\in Aut(\mathcal{C})$, of order $d$, with $d$ the period of the Markov chain $(X_0,p)$, such that
	\begin{equation}
		\forall z\in\mathbb{D}(0,R),\, u(e^{2i\pi/d} z)=Au(z).
	\end{equation}
	
In \cite{LCurve}, we prove (Proposition \ref{LCurve-Regularity_of_C}), that $\mathcal{C}$ is smooth in a neighbourhood of the compact set $\overline{\{u(z): z\in\mathbb{D}(0,R)\}}\subset \mathcal{C}$, and (Proposition \ref{LCurve-Prop_Holomorphic_Extension_of_u}) that $u$ can be holomorphically extended in a neighbourhood of any point of the boundary $\partial\mathbb{D}(0,R)$, that is not a $d$-th root of $R^d$. 
We finally show (subsection \ref{LCurve-subsection_derivative_of_lambda} of \cite{LCurve}) that the first derivative of $\lambda$ at $u(R)$ is zero (Proposition \ref{LCurve-Prop_derivative_of_lambda}), and its second derivative is non-zero (Proposition \ref{LCurve-Prop_derive_seconde_de_lambda}).

To prove the different asymptotics of probability, we have to define a function $g$ and prove that its first derivative at $u(R)\in\mathcal{C}$ is non-zero.
		\subsection{Asymptotic for restricted random walks}
	We begin by examining the behaviour of the sequence of probabilities $(p^{(n)}(a,b;\mathcal{B}(y)^\complement))_{n\in\mathbb{N}}$ for $(a,b)_y$ in the set $\Xi$. In this case, for $(a,b)_y\in\Xi$, we will consider the function $g$ to be the evaluation map, 
	\begin{equation}\label{evaluation_map_def}
	g:	(z,J)\mapsto J((a,b)_y),
	\end{equation} 
	whose first derivative is zero or non-zero in function of $(a,b)_y$. 
	Thereby, we partition the set $\Xi$ in two subsets, $\Xi_\infty$ and $\Xi_{<\infty}$. \index{$\Xi_{<\infty}$}\index{$\Xi_\infty$}
\begin{enumerate}
	\item $(a,b)_y\in\Xi$ belongs to $\Xi_{<\infty}$ if any $p$-admissible path $\gamma$ in $\mathcal{P}h(a,b;\mathcal{B}(y)^\complement)$ has its vertices at a uniformly bounded distance from $y$. That is to say, there exists a non-negative integer $M$ such that:  \begin{equation*}
	\forall \gamma=(\omega_0,...,\omega_m)\in\mathcal{P}h(a,b;\mathcal{B}(y)^\complement), \; \max_i(d(\omega_i,y))<M.
	\end{equation*}

	\item $(a,b)_y\in\Xi$ belongs to $\Xi_\infty$ if it does not have such uniform bound, i.e. such that for any positive integer $M>0$,
\begin{equation*}
\exists\gamma=(\omega_0,...,\omega_m)\in\mathcal{P}h(a,b;\mathcal{B}(y)^\complement) \text{ s.t }  \max_i(d(\omega_i,y))\geq M.
\end{equation*}
\end{enumerate}

Remark that if $(a,b)_y$ is in $\Xi_{<\infty}$ then by definition, the $p$-admissible paths of $\mathcal{P}h(a,b;\mathcal{B}(y)^\complement)$ can visit only a finite number of vertices. We thus have
\begin{Lem}\label{Rationality_of_g_xi_<_infty}
Let $X=(X_0,X_1)$ be an infinite tree of bounded valence, $\Gamma<\operatorname{Aut}(X)$ be a subgroup of automorphisms of the tree such that $\Gamma\backslash X_0$ is finite, let $p:X_0\times X_0\to [0,1]$ be a finite range $\Gamma$-invariant transition kernel, making $(X_0,p)$ an irreducible Markov chain and let $k\in\mathbb{N}$ be an integer such that 
		$$\forall x,y\in X_0,\, d(x,y)>k\Rightarrow p(x,y)=0.$$
		
If $(a,b)_y$ is an element of $\Xi_{<\infty}$, then the associated restricted Green's function $z\mapsto G_z(a,b;\mathcal{B}(y)^\complement)$ is a rational function. In particular it admits a pole at its radius of convergence $z=R(a,b;\mathcal{B}(y)^\complement)>R$.
\end{Lem}
\begin{proof}
Let $\mathcal{E}_{(a,b)_y}$ be the set of all vertices of $p$-admissible paths in $\mathcal{P}h(a,b;\mathcal{B}(y)^\complement)$: 
$$\mathcal{E}_{(a,b)_y}:=\{ w: w\in\gamma,\,\gamma\in\mathcal{P}h(a,b;\mathcal{B}(y)^\complement \}.$$

By assumption, the set $\mathcal{E}_{(a,b)_y}$ is finite. Consider the kernel operator $Q$ defined over the set of complex valued functions $\mathcal{F}(\mathcal{E}_{(a,b)_y},\mathbb{C})$  whose kernel is given by $q:\mathcal{E}_{(a,b)_y}\times \mathcal{E}_{(a,b)_y}\to [0,1]$, defined for any pair $(c,c')$ in $\mathcal{E}_{(a,b)_y}\times \mathcal{E}_{(a,b)_y}$ by
\begin{equation}
q(c,c')=\left\lbrace
\begin{matrix}
p(c,c'), \text{ if } c\notin \mathcal{B}(y)^\complement\\
0, \text{ else}.
\end{matrix}\right.
\end{equation}
More precisely $Q$ is the operator that sends any function $f$ in $\mathcal{F}(\mathcal{E}_{(a,b)_y},\mathbb{C})$ to the function denoted $Qf$: 
$$Qf:c\mapsto \sum_{c'\in\mathcal{E}_{(a,b)_y}}q(c,c')f(c').$$
It is a Perron (linear) operator (See Definition \ref{Def_Perron_operator}) on the vector space $\mathcal{F}(\mathcal{E}_{(a,b)_y},\mathbb{C})$.

By convention set $q^{(0)}(c,c'):=\delta_c(c')$. And set, for any positive integer $n\geq 1$, the function $q^{(n)}:\mathcal{E}_{(a,b)_y}\times \mathcal{E}_{(a,b)_y}\to [0,1]$ to be the kernel of the operator $Q^{\circ n}$, then $q^{(n)}$ is:
$$q^{(n)}:(c,c')\mapsto\sum_{\substack{x_0,...,x_n\in\mathcal{E}_{(a,b)_y}\\ x_0=c, \, x_n=c' }}q(x_0,x_1)\cdots q(x_{n-1},x_n).$$

An immediate induction on $n\geq 0$ shows that, for any non-negative integer $n\geq 0$, and any $c,c'\in\mathcal{E}_{(a,b)_y}$, $c\neq b$ we have:
	$$q^{(n)}(c,c')=p^{(n)}(c,c';\mathcal{B}(y)^\complement).$$
	We detail this induction below.
In particular we get that 
$$G_z(a,b;\mathcal{B}(y)^\complement)=\left(\sum_{n\geq 0} (zQ)^n\delta_{b}\right)(a)=\left((I-zQ)^{-1}\delta_b\right)(a),$$
where $\delta_b$ is the Dirac mass at $b$.
And so $z\mapsto G_z(a,b;\mathcal{B}(y)^\complement)$ is a rational functions, that admits a pole at $\rho(Q)^{-1}\in ]0,\infty]$, the inverse of the spectral radius of $Q$. And it does not possesses any pole in $\mathbb{D}(0,\rho(Q)^{-1})$.
By finiteness of $G_R(a,b;\mathcal{B}(y)^\complement)$ (Theorem \ref{Finitude_of_Green_functions}), we necessarily have $R(a,b;\mathcal{B}(y)^\complement)>R$.

We now prove the induction mentioned above, that is, for any non-negative integer $n\geq 0$, we have $$\forall c,c'\in \mathcal{E}_{(a,b)_y},\, c\neq b,\; q^{(n)}(c,c')=p^{(n)}(c,c';\mathcal{B}(y)^\complement).$$

For $n=0$, the result is immediate. Suppose that it is verified for some $n\geq 0$. Let $c,c'$ be in $\mathcal{E}_{(a,b)_y}$, with $c\neq b$, by definition and using the induction hypothesis,
\begin{align*}
q^{(n+1)}(c,c')&=\sum_{c''\in\mathcal{E}_{(a,b)_y}}q^{(n)}(c,c'')q(c'',c')\\
			   &= \sum_{c''\in\mathcal{E}_{(a,b)_y}\cap \mathcal{B}(y)^\complement} 
			   p^{(n)}(c,c'';\mathcal{B}(y)^\complement) p(c'',c') \\
			&=p^{(n+1)}(c,c';\mathcal{B}(y)^\complement),
\end{align*}
where the last equality comes from the formula $$p^{(n+1)}(c,c')=\sum_{c''\notin \mathcal{B}(y)^\complement} p^{(n)}(c,c'';\mathcal{B}(y)^\complement) p(c'',c')$$ and the definition of $\mathcal{E}_{(a,b)_y}$. 
\end{proof}
\begin{Rem}\label{Rem_Green_Poly}
    For $(a,b)_y\in\Xi_{<\infty}$, the pole of the associated restricted Green's function may not be simple (See example \ref{Example Green poly} below). The pole may also be $z=\infty$, corresponding to the case where the function is polynomial in $z$. This is, for example, the case in the situations described in the Example \ref{Example_Aquarterquinquies} below.
\end{Rem}
\begin{Ex}\label{Example_Aquarterquinquies}
    Let $\mu$ and $\nu$ be two probability distributions over $F=\left(\mathbb{Z}/2\mathbb{Z}\right)^{\star 3}$ with support $\operatorname{Supp}(\mu)=\{a, ac, ba\}$ and $\operatorname{Supp}(\nu)=\{a,ab,ac\}$, where $a,b,c$ are the three generators of each of the three copies of $\mathbb{Z}/2\mathbb{Z}$ (See Example \ref{Typical_Example}). Denote by $e$ the neutral element of the free product $F$. In these situations the Green's functions can be computed, and we have in particular:
    $$G_z^{(\mu)}(ca,e;\mathcal{B}_2(a)^\complement)=\mu(ac)z\; \text{ and }\; G_z^{(\nu)}(ba,e;\mathcal{B}_2(a)^\complement)=\nu(ab)z. $$
    These are polynomial function of $z$.
\end{Ex}
\begin{Ex}\label{Example Green poly}
Let $q\in]0,1[$ be a real number. Consider a four colored regular tree $\mathcal{T}_3$ of valence $3$, with colors $\alpha,\beta,\gamma,\delta$ such that any vertex of a given color has its neighbours colored with the three other colors. That is $\mathcal{T}_3$ is given a map $col:\mathcal{T}_3\to\{\alpha,\beta,\gamma,\delta\}$ such that if $x$ is a vertex and $x_1,x_2,x_3$ are its three neighbours in $\mathcal{T}_3$, then the set $\{col(x), col(x_1), col(x_2), col(x_3)\}$ equals $\{\alpha,\beta,\gamma,\delta\}$. We define the step distribution of the random walk (the transition kernel) $p:X_0\times X_0\to[0,1]$ as follow:
\begin{itemize}
    \item If $x$ has color $\alpha$, then for any $y\in X_0$, we set $p(x,y)=0$ if $d(x,y)>3$, and $p(x,y)=\frac{1}{22}$ else. Corresponding to the uniform distribution over $\mathcal{B}_3(x)$;
    \item If $x$ has color $\beta$, then for any $y\in X_0$, we set $p(x,y)=1$ if, and only if $x$ and $y$ are neighbours, and $col(y)=\alpha$ and $p(x,y)=0$ otherwise;
    \item If $x$ has color $ \gamma $, then  we set $p(x,x)=q$ and for any $y\in X_0\setminus\{x\}$, we set $p(x,y)=1-q$ if, and only if $x$ and $y$ are neighbours and $col(y)=\beta$, and $p(x,y)=0$ otherwise;
    \item If $x$ has color $\delta$, then  we set $p(x,x)=q$ and for any $y\in X_0\setminus\{x\}$, we set $p(x,y)=1-q$ if, and only if $d(x,y)=2$, $col(y)=\gamma$, and the color of the intermediate vertex $w\in[x,y]$ is $col(w)=\beta$, and we set $p(x,y)=0$ otherwise.
\end{itemize}
This defines an aperiodic irreducible random walk on $\mathcal{T}_3$. Any automorphism preserving the colors also preserves the transition kernel and the group of automorphism preserving the coloring acts cofinitely on the set of vertices of $\mathcal{T}_3$.

In this case let $a$ be a vertex colored $\delta$ and let $b$ be its only neighbour colored $\beta$.
Denote by $c$ the neighbour of $b$ with color $\gamma$, then let $y$ be any vertex at distance $d(a,y)=4$ from $a$ and at distance $d(b,y)=3$ from $b$ (See Figure \ref{fig:my_label}). 
Then any $p$-admissible path starting from $a$ must pass by the vertex $c$. Considering all possible admissible path in $\mathcal{P}h(a,b;\mathcal{B}(y)^\complement)$ and using the formula (\ref{SumRestrictedWeight}) we have $G_z(a,b;\mathcal{B}(y)^\complement)=z^2(1-p)^2\left(\dfrac{1}{1-p z}\right)^2$. This restricted Green's function possesses a pole of order $2$ at $z=\frac{1}{p}$.
\end{Ex}
\begin{figure}
    \centering
    \includegraphics[scale=0.25]{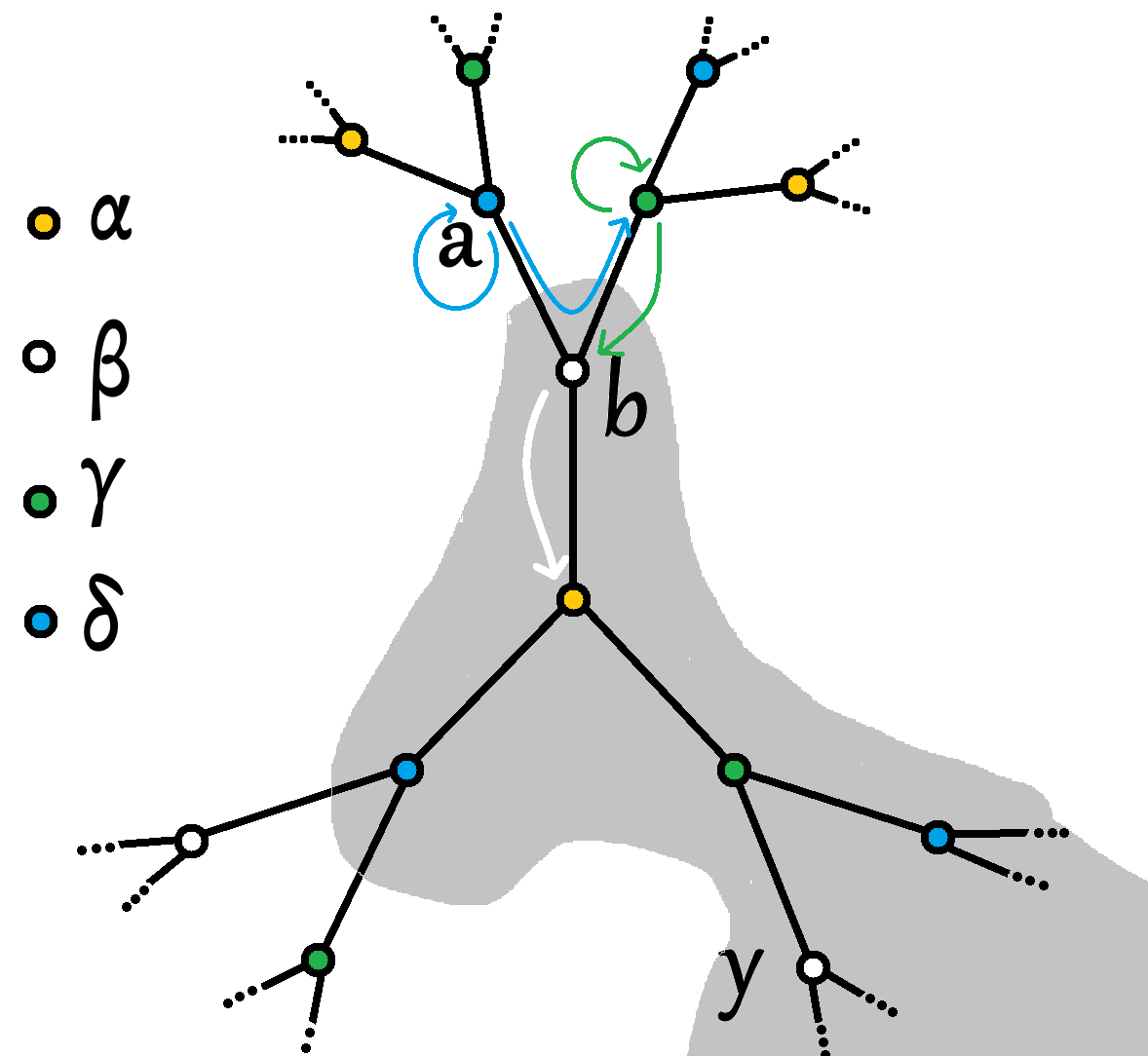}
    \caption{Illustration of Example \ref{Example Green poly}}
    \label{fig:my_label}
\end{figure}
\vspace{\baselineskip}
Depending on whether $(a,b)_y$ belongs to $\Xi_\infty$ or belongs to $\Xi_{<\infty}$, the sequence of probabilities $(p^{(n)}(a,b;\mathcal{B}(y)^\complement))_n$ will have different behaviour. We summarize this in the theorem \ref{Asymptotics_of_probability_Green_restricted_irreducible} below: 
\begin{Thm}\label{Asymptotics_of_probability_Green_restricted_irreducible}~
		Let $X=(X_0,X_1)$ be an infinite tree of bounded valence such that any vertex possesses at least three neighbours, $\Gamma<\operatorname{Aut}(X)$ be a subgroup of automorphisms of the tree such that $\Gamma\backslash X_0$ is finite, $p:X_0\times X_0\to [0,1]$ be a finite range $\Gamma$-invariant transition kernel, making $(X_0,p)$ an irreducible Markov chain and let $k\in\mathbb{N}$ be an integer such that 
		$$\forall x,y\in X_0,\, d(x,y)>k\Rightarrow p(x,y)=0.$$
		Denote by $d$ the period of $(X_0,p)$ and $r:X_0\times X_0\to \mathbb{Z}/d\mathbb{Z}$, the periodicity cocycle defined in Definition \ref{DEF:PERIODICITY_COCYCLE}.
		
		For all $(a,b)_y\in\Xi_\infty$, there exist constants $C>0,\, (c_l)_{l\geq 1}$ such that we have the asymptotic expansion
		\begin{equation*}
		p^{(dn+r(a,b))}(a,b;\mathcal{B}(y)^\complement) \sim_\infty  C R^{-dn} \frac{1}{n^{3/2}}\left( 1 + \sum_{l=1}^\infty \frac{c_l}{n^{l}}\right ).
		\end{equation*}
		and $p^{(dn+t)}(a,b;\mathcal{B}(y)^\complement)=0$ if $t\neq r(a,b) \, [d]$;
		
		For all $(a,b)_y\in\Xi_{<\infty}$, there exists a real number $R'>R$, a positive integer $M\geq 1$, a non-zero constant $D_{-M}$ and real constants $D_{-M+1},...,D_{-1}$ such that, we have as $n$ goes to $\infty$:
		\begin{equation*}
			p^{(dn+r(a,b))}(a,b;\mathcal{B}(y)^\complement)=\frac{D_{-M}}{n^{1-M}( R')^{dn}}+\frac{D_{-{M+1}}}{n^{2-M}( R')^{dn}}+...+\frac{D_{-M}}{( R')^{dn}} +O\left(\frac{1}{(R'+\varepsilon)^{n}}\right).
		\end{equation*}
		and $p^{(dn+t)}(a,b;\mathcal{B}(y)^\complement)=0$, if $t\neq r(a,b)\, [d]$.
	\end{Thm}

\begin{proof}
To prove the above theorem, consider $(a,b)_y\in\Xi$ and the evaluation map $g:\mathcal{C}\to\mathbb{C}$ defined earlier (\ref{evaluation_map_def}). We prove in \cite{LCurve} Proposition \ref{LCurve-Prop_Spectrum_belonging_C}, that the derivative $D_{u(R)}g$ of $g$ at $u(R)$ (which is a projection on the $(a,b)_y$ coordinate) is non-zero over the tangent space $T_{u(R)}\mathcal{C}$ when $(a,b)_y$ is in $\Xi_{\infty}$. 

In the first case, when $(a,b)_y$ is in $\Xi_\infty$, applying Theorem \ref{Thm_Cor_Black_box_general} with $\mathcal{C}$ being the Lalley's curve, $u$ and $\lambda$ being the function defined respectively in (\ref{Def_u}) and (\ref{Def_lambda}) and $g:\mathcal{C}\to\mathbb{C}$ the evaluation map (\ref{evaluation_map_def}), we get the theorem.

In the second case, when $(a,b)_y$ is in $\Xi_{<\infty}$, from previous Lemma \ref{Rationality_of_g_xi_<_infty} the
restricted Green's function $z\mapsto G_z(a,b;\mathcal{B}(y)^\complement)$ is a rational function that admits a pole at its radius of convergence $R(a,b;\mathcal{B}(y)^\complement)>R>0$. Besides from Corollary \ref{Cor_Admissible_Path_Lenght_without_admissible_paths} $G_{(\zeta_d z)}(a,b;\mathcal{B}(y)^\complement)=\zeta^{r(a,b)}G_z(a,b;\mathcal{B}(y)^\complement)$. Then applying theorem \ref{RTaub-Asymptotic_Expansion_for_merop_func} we obtain the desired asymptotic.
\end{proof}

	\subsection{Asymptotic for random walks}
	We continue with the sequence of first-passage probabilities $(p^{(n)}(x,y;\{y\}^\complement))_n$, for $x,y$ two vertices of the tree. This time the function $g$ considered is the function $f_{x,y}$ from Proposition \ref{algebraicity_of_F} whose first derivative at $u(R)$ is zero, or non-zero depending on the pair $(x,y)$ of vertices of $X$. Once again we separate pairs $(x,y)\in X_0\times X_0$ into two categories: on the one hand, the pairs $(x,y)$ such that any $p$-admissible path of $\mathcal{P}h(x,y;\{y\}^\complement)$ has its vertices at a uniformly bounded distance from $y$; on the other hand, those that do not have such uniform bound. Depending on the category to which the pair $(x,y)$ belongs, we will observe different behaviours, which are summarized in the theorem \ref{Asymptotic_for_first_passage_probabilities} below.

	\begin{Lem}\label{Lem:five_to_three}
	Let $X=(X_0,X_1)$ be an infinite tree of bounded valence, $\Gamma<\operatorname{Aut}(X)$ be a subgroup of automorphisms of the tree such that $\Gamma\backslash X_0$ is finite, $p:X_0\times X_0\to [0,1]$ be a finite range $\Gamma$-invariant transition kernel, making $(X_0,p)$ an irreducible Markov chain and let $k\in\mathbb{N}$ be an integer such that 
		$$\forall x,y\in X_0,\, d(x,y)>k\Rightarrow p(x,y)=0.$$
		
		Let $x,y$ be two vertices of $X$. Suppose that there exists an integer $M$ such that for any $p$-admissible path $\gamma=(\omega_0,...,\omega_n)\in\mathcal{P}h(x,y;\{y\}^\complement)$, $\max_i(d(\omega_i,y))\leq M$.
		
		Then the first-passage generating function $F_z(x,y)$ is a rational function. In particular it admits a pole at its radius of convergence $z=R_F(x,y)>R$
	\end{Lem}
\begin{proof}
The proof follows the same method as the one used in Lemma \ref{Rationality_of_g_xi_<_infty}, with the finite set $\mathcal{E}_{x,y}$ of all vertices of $p$-admissible paths in $\mathcal{P}h(x,y;\{y\}^\complement)$: 
$$\mathcal{E}_{x,y}:=\{ w: w\in\gamma,\,\gamma\in\mathcal{P}h(x,y;\{y\}^\complement \}.$$
And with the kernel operator $Q$ over the vector space of complex-valued functions $\mathcal{F}(\mathcal{E}_{x,y},\mathbb{C})$  whose kernel is given by $q:\mathcal{E}_{x,y}\times \mathcal{E}_{x,y}\to [0,1]$ defined for any pair $(a,b)$ in $\mathcal{E}_{x,y}\times \mathcal{E}_{x,y}$ by
\begin{equation}
q(a,b)=\left\lbrace
\begin{matrix}
p(a,b), \text{ if } a\neq y;\\
0, \text{ else}.
\end{matrix}\right.
\end{equation}


Then by finiteness of $F_R(x,y)$ (Theorem \ref{Finitude_of_Green_functions}), we necessarily have the inequality $R_F(x,y)>R$.
\end{proof}
\begin{Rem}
The pole $z=R_F(x,y)>R$ need not be simple and the first-passage generating function $z\mapsto F_z(x,y)$ may be a polynomial function. See Remark \ref{Rem_Green_Poly}
\end{Rem}

\begin{Thm}\label{Asymptotic_for_first_passage_probabilities}
	Let $X=(X_0,X_1)$ be an infinite tree of bounded valence such that any vertex possesses at least three neighbours, $\Gamma<\operatorname{Aut}(X)$ be a subgroup of automorphisms of the tree such that $\Gamma\backslash X_0$ is finite, $p:X_0\times X_0\to [0,1]$ be a finite range $\Gamma$-invariant transition kernel, making $(X_0,p)$ an irreducible Markov chain and let $k\in\mathbb{N}$ be an integer such that 
		$$\forall x,y\in X_0,\, d(x,y)>k\Rightarrow p(x,y)=0.$$
		Denote by $d\in\mathbb{N}$ the period of $(X_0,p)$ and $r:X_0\times X_0\to \mathbb{Z}/d\mathbb{Z}$, the periodicity cocycle defined in Definition \ref{DEF:PERIODICITY_COCYCLE}.
		
		For all $x,y$, with $x\neq y $:

	-If $\forall M>0,\exists\gamma\in\mathcal{P}h(x,y;\{y\}^\complement):\max_{\omega\in\gamma}( d(\omega,y)) >M$ then there exist constants $C>0,\, (c_l)_{l\geq 1}$ such that we have the asymptotic expansion:
	$$p^{(dn+r(x,y))}(x,y;\{y\}^\complement) \sim_\infty C R^{-dn}\frac{1}{n^{3/2}}\left( 1 +\sum_{l=1}^\infty \frac{c_l}{n^l}\right).$$
	and $p^{(dn+t)}(x,y;\{y\}^\complement)$, if $t\neq r(x,y)\,[d]$.

	-Otherwise, there exists a real number $R'>R$, a positive integer $l\geq 1$, a non-zero real constant $D_{-l}$ and real $D_{-l},...,D_{-1}$ such that, we have as $n$ goes to $\infty$,
		$$p^{(dn+r(x,y))}(x,y;\{y\}^\complement)=\frac{D_{-l}}{n^{1-l} (R')^{dn}}+\frac{D_{-{l+1}}}{n^{2-l} (R')^{dn}}+...+\frac{D_{-1}}{( R')^{dn}} +O\left(\frac{1}{(R'+\varepsilon)^{n}}\right),$$
		where $\varepsilon >0$ is small enough,	and $p^{(dn+t)}(x,y;\{y\}^\complement)=0$, if $t\neq r(x,y)\,[d]$.
	\end{Thm}

\begin{proof}
For pairs of vertices $(x,y)$ that admits a non-negative integer $M$ such that any $p$-admissible path $\gamma$ in $\mathcal{P}h(x,y;\{y\}^\complement)$ has all its vertices at distance at most $M$ from $y$ (i.e $\max_{\omega\in\gamma}( d(\omega,y) )\leq M$), we have proven in Lemma \ref{Lem:five_to_three} that the first-passage generating function $z\mapsto F_z(x,y)$ is a rational function that possesses a pole at its radius of convergence $R_F(x,y)>R$. From Corollary \ref{Cor_Admissible_Path_Lenght_without_admissible_paths} it satisfies for any complex number $z$ of modulus small enough $F_{(\zeta_d z)}(x,y)=\zeta_d^{r(x,y)}F_z(x,y)$.  Lastly applying Theorem \ref{RTaub-Asymptotic_Expansion_for_merop_func} to the first-passage generating function we obtain the desired asymptotic.
 
For pairs $(x,y)$ that do not admit such uniform bound, we use Theorem \ref{Thm_Cor_Black_box_general} with $\mathcal{C}$ the Lalley's curve, $u$ and $\lambda$ being the function defined in (\ref{Def_u}) and (\ref{Def_lambda}), and with $g$ the functions $f_{x,y}$ from Proposition \ref{algebraicity_of_F}. We prove in Proposition \ref{LCurve-Corollary_PSeries_A} in \cite{LCurve} that the derivative of $f_{x,y}$ at $u(R)$ is non-zero, in this case. We thus get the theorem.
\end{proof}

Finally we get to study sequence of probabilities $(p^{(n)}(x,y))_n$. This time the behaviour of the sequence is the same for any pair $(x,y)$ of vertices of $X$.

	\begin{Thm}\label{Asymptotics_of_probability_Green_General}
		~
		Let $X=(X_0,X_1)$ be an infinite tree of bounded valence such that any vertex possesses at least three neighbours, $\Gamma<\operatorname{Aut}(X)$ be a subgroup of automorphisms of the tree such that $\Gamma\backslash X_0$ is finite, $p:X_0\times X_0\to [0,1]$ be a finite range $\Gamma$-invariant transition kernel, making $(X_0,p)$ an irreducible Markov chain and let $k\in\mathbb{N}$ be an integer such that 
		$$\forall x,y\in X_0,\, d(x,y)>k\Rightarrow p(x,y)=0.$$
		Denote by $d$ the period of $(X_0,p)$ and $r:X_0\times X_0\to \mathbb{Z}/d\mathbb{Z}$, the periodicity cocycle (Definition \ref{DEF:PERIODICITY_COCYCLE}).
		
		For all $x,y\in X_0$, there exist constants $C>0,\, (c_l)_{l\geq 1}$ such that we have the asymptotic expansion as $n$ goes to $\infty$:
		$$ p^{(dn+r(x,y))}(x,y) \sim_\infty C R^{-dn}\frac{1}{n^{3/2}}\left( 1 + \sum_{l=1}^\infty\frac{c_l}{n^{l}}\right),$$
		and $p^{(dn+t)}(x,y;\{y\}^\complement)=0$, if $t\neq r(x,y)\,[d]$.
	\end{Thm}
	\begin{proof}
To prove this we apply Theorem \ref{Thm_Cor_Black_box_general} with $\mathcal{C}$ the Lalley's curve, $u$ and $\lambda$ being the function defined in (\ref{Def_u}) and (\ref{Def_lambda}), and with $g=g_{x,y}$ from Corollary \ref{Cor_algbraicity_of_Greens_function}. We prove in Proposition \ref{LCurve-Prop_derivative_of_g} from \cite{LCurve}, that the derivative of $g_{x,y}$ at $u(R)$ is non-zero.
\end{proof}

\textit{The study of the Lalley's curve $\mathcal{C}$ and of the functions \og $g$ \fg{} used in the last section is done in \cite{LCurve}. More precisely, we prove that the Lalley's curve is regular on a neighbourhood of $\{u(z):z\in\overline{\mathbb{D}(0,R)}\}$  and  we study the first derivatives at $u(R)$ of the different functions \og $g:\mathcal{C}\to\mathbb{C}$ \fg{} used in the previous section \ref{Section_Asymptotics}. To do so we introduce the \textit{Dependency Directed Graph} that encodes all the relations needed in our study.}

\appendix

\begin{center}
\huge{\textbf{Appendix}} 
\end{center}

	\section{General results on Green's Functions and Markov Chains}\label{Appendix:Greens_functions}
	We recall the definition of a Markov-Chain and a transition Kernel
	\begin{Def}\textit{Transition Kernel and Discrete Markov Chain on tree} (Recall from Definition \ref{DEF:TRANSITION_KERNEL_DSCT_MKV_CHN})

Let $X_0$ be a set.	A function $p:X_0\times X_0\to [0,1]$ is called \textit{transition kernel}\index{Transition kernel} over $X_0$ if it satisfies:
	\begin{equation}
		\forall x \in X_0,\,\sum_{y\in X} p(x,y)=1.
	\end{equation}
	
	Given a transition kernel $p$ over $X_0$, for any non-negative integer $n\geq 0$, we can define another transition kernel $p^{(n)}:X_0\times X_0\to[0,1]$, that associates to any pair of vertices $(x,y)\in X_0\times X_0$, the value:
	\begin{equation}\label{DEF:TRANS_KERNEL_n}
	    p^{(n)}(x,y):=\sum_{\substack{\omega_0,...,\omega_n\in X_0 \\ \omega_0=x, \, \omega_n=y }} p(\omega_0,\omega_1)\cdots p(\omega_{n-1},\omega_n).
	\end{equation}

The pair $(X_0,p)$ is called a \textit{discrete Markov chain}\index{Markov Chain (discrete)} with \textit{state space} $X_0$. It will be said to be \textit{irreducible}\index{Irreducible (Markov chain)} if for any pair $(x,y)\in X_0\times X_0$, there exists a non-negative integer $n$ such that $p^{(n)}(x,y)>0$.

	For an irreducible discrete Markov chain $(X_0,p)$ we define its \textit{period} $d$ as the greatest common divisor 
	\begin{equation}\label{DEF_of_the_period}\index{Period}
	d:=\gcd\left\lbrace n : p^{(n)}(x,x)>0 \right\rbrace.
	\end{equation}
	If $d=1$, we say that the Markov chain $(X_0,p)$ is \textit{aperiodic}. This definition does not depend on any choice of $x$ as proven in Proposition \ref{period_is_well_defined}.
\end{Def}
	
	From definition, an immediate property satisfied by the transition kernels $p^{(n)}$ is the following:

For any vertices $x,y\in X_0$ and any non-negative integers $n,m\in\mathbb{N}$, we have the formula
	\begin{equation}\label{Kolmo-Chap_eq}
	\index{Chapman-Kolmogorov equation}
		p^{(n+m)}(x,y)=\sum_{w\in X_0} p^{(n)}(x,w)p^{(m)}(w,y).
	\end{equation}
This equation is known as Chapman-Kolomogorov equation.

	\subsection{(Standard) Green's functions and Restricted Green's functions}\label{subsection_Rest_Green}
	
We recall the definition of Green's functions and restricted Green's functions, and give general results about Green's functions. 

\begin{Def}\textit{Green's functions} (Recall from Definition \ref{DEF:GREEN_FUNCTION})

	Let $(X_0,p)$ be a discrete Markov chain (See Definition \ref{DEF:TRANSITION_KERNEL_DSCT_MKV_CHN}).	For $(x,y)\in X_0 \times X_0$, the Green's function at $(x,y)$, associated with the Markov chain $(X_0,p)$ is the germ of holomorphic function in the neighbourhood of $z=0$ in $\mathbb{C}$, denoted by $z\mapsto G_z(x,y)$, given by the power series:
	\begin{equation}\label{Equation_Def_Green}
		G_z(x,y)=\sum_{n\in\mathbb{N}} p^{(n)}(x,y)z^n
	\end{equation}
	In other words, it is the generating function associated with the sequence $(p^{(n)}(x,y))_n$.
\end{Def}

\begin{Rem}
	Since $\sum_n p^{(n)}(x,y) z^n$ is a power series with non-negative real coefficients, for any non-negative real number $r$, the series $\sum_n p^{(n)}(x,y)r^n$ is an element of $[0,\infty]$. Actually, Green's functions are well-defined on any open disk $\mathbb{D}(0,r)$ in $\mathbb{C}$ of radius $r>0$ centred at $z=0$, such that $\sum_n p^{(n)}(x,y)r^n < \infty$.
\end{Rem}

Let us look at the radius of convergence of Green's functions. Consider a discrete Markov chain $(X_0,p)$. Let $R=R(x,y)$ be the radius of convergence of the power series expansion in the neighbourhood of $z=0$ of the  Green's function $z\mapsto G_z(x,y)$. According to the the Cauchy-Hadamard formula\footnote{See\cite{Queffelec_2017} formula (19) p.54.}, we have:
\begin{equation} \label{FormCauHad}
	\frac{1}{R}=\limsup\,(p^{(n)}(x,y))^\frac{1}{n}\leq 1.
\end{equation}
In particular, $R\geq 1$. If $(X_0,p)$ is irreducible, the value of $R$ does not depend on $(x,y)$ (see Proposition \ref{Prop_Markov_Ired_et_rayon_de_cv} below). Since we only consider irreducible Markov chains, we take the liberty of omitting the $(x,y)$ dependency of $R$.

\begin{Prop}\label{Prop_Markov_Ired_et_rayon_de_cv}
	Let $(X_0,p)$ be a discrete \underline{irreducible} Markov chain. For any $x,y$ in $X$, the series $\sum_{n\in\mathbb{N}}p^{(n)}(x,y)r^n$ are of same nature. I.e for any $r\geq 0$,
	$$\forall (x,y)\in X_0\times X_0,\,\sum_n p^{(n)}(x,y)r^n<\infty$$
	if, and only if, $$\exists (x,y)\in X_0\times X_0,\,\sum_n p^{(n)}(x,y)r^n <\infty$$
\end{Prop}
 \begin{proof}
	All we need to do is to show the reciprocal direction.

	Let $r\geq 0$, $(x,y)\in X_0\times X_0$ and $(x',y')\in X_0\times X_0$ be arbitrary. 
	By definition $G_r(x,y)=\sum_n p^{(n)}(x,y)r^n$, and by irreducibility there exist two non-negative integers $n_1,n_2\geq 0$ such that $p^{(n_1)}(x,x')>0$ and $p^{(n_2)}(y',y)>0$. From the Chapman-Kolmogorov equation (\ref{Kolmo-Chap_eq}), for any non-negative integer $n\geq 0$ we have:
	$$p^{(n_1+n+ n_2)}(x,y)\geq p^{(n_1)}(x,x')p^{(n)}(x',y')p^{(n_2)}(y',y).$$
Multiplying the above inequality by $r^n$, for $r\geq 0$ and summing it over $n\geq 0$ we get:
$$\sum_{n\geq n_1+n_2} p^{(n)}(x,y)r^n \geq  p^{(n_1)}(x,x') G_r(x',y') p^{(n_2)}(y',y).$$
	
	It follows by strict positivity of $ p^{(n_1)}(x,x')$ and $p^{(n_2)}(y',y)$ on $p$-admissible paths for $r>0$ (the case $r=0$ being trivial) that $G_r(x',y')$ is finite whenever $G_r(x,y)$ is.
\end{proof}

\begin{Def}\textit{Restricted transition kernel}

	Let $(X_0,p)$ be an irreducible discrete Markov chain and let $\Omega$ be a subset of $X_0$. For $(x,y)\in X_0\times X_0$ and $n\geq 1$ a non-negative integer, denote by $p^{(n)}(x,y;\Omega)\in [0,1]$ the value:
	$$ p^{(n)}(x,y;\Omega):=\sum_{\omega_0,...,\omega_n} p(\omega_0,\omega_1)\cdots p(\omega_{n-1},\omega_n),$$
	where the sum is taken over all $n$-tuple $(\omega_0,...,\omega_n)$ of elements in $X_0$ such that $\omega_0=x$, $\omega_n=y$ and for any $i\in\{1,...,n-1\}$, we have $\omega_i$ belongs to $\Omega$.
	By convention we set $p^{(0)}(x,y;\Omega)=1$ if $x=y$, and is zero otherwise.
\end{Def}

\begin{Def}\label{Def_Green_Restricted}\textit{Restricted Green's functions}

	Let $(X_0,p)$ be an irreducible discrete Markov chain and let $\Omega$ be a subset of $X_0$. For $(x,y)\in X_0\times X_0$, the \textit{restricted Green's function, to $\Omega$}, at $(x,y)$, associated with the Markov chain $(X,p)$ is the germ of holomorphic function in the neighbourhood of $z=0$ in $\mathbb{C}$, denoted $z\mapsto G_z(x,y;\Omega)$, given by the power series:
	\begin{equation}\label{Equation_Def_Green_restreintes}
		G_z(x,y;\Omega):=\sum_{n\in\mathbb{N}} p^{(n)}(x,y;\Omega)z^n
	\end{equation}
	
	In other words, it is the generating function associated with the sequence $(p^{(n)}(x,y;\Omega))_{n\in\mathbb{N}}$ of probabilities that a random walk based at $x$, is at $y$ after exactly $n\in\mathbb{N}$ steps with intermediate steps in $\Omega$.
\end{Def}

\begin{Rem}
	For an irreducible discrete Markov chain, we have seen that the radius of convergence $R$ of the power series' expansion of Green's function in the neighbourhood of $z=0$, does not depend on the pair $(x,y)\in  X_0\times X_0$. This is not the case for restricted Green's functions. For $(x,y)\in X_0\times X_0$, $\Omega\subset X_0$, if $R(x,y;\Omega)$ denotes the radius of convergence of the power series $\sum_n p^{(n)}(x,y;\Omega) z^n$ then we always have $R\leq R(x,y;\Omega)$, but the equality is not always satisfied (See example \ref{Example_Markov_Chain_1} below).
\end{Rem}
\begin{Ex}\label{Example_Markov_Chain_1}
	For the irreducible finite Markov chain $(X_0,p)$ with three state space $X_0:=\{x,y,w\}$, for which $p(y,y)=p(y,w)=p(w,y)=p(w,x)=1/2$ and $p(x,w)=1$ (in particular $p(x,y)=0$, see figure \ref{Fig_example_1}). If $\Omega$ is the singleton $\{y\}$ then in a neighbourhood of $z=0$, we have
	$$G_z(x,y;\Omega)=0 \text{ and } G_z(y,y;\Omega)=\frac{1}{1-z/2}.$$
	The respective radius of convergence of the above associated power series are $R(x,y;\Omega)=\infty$ and $R(y,y;\Omega)=2$. In particular, unlike Green's functions \og $G_r(x,y)$\fg{}, restricted Green's functions do not necessarily have the same radius of convergence. 
	\begin{figure}[ht]
	    \centering
		\includegraphics[scale=0.5]{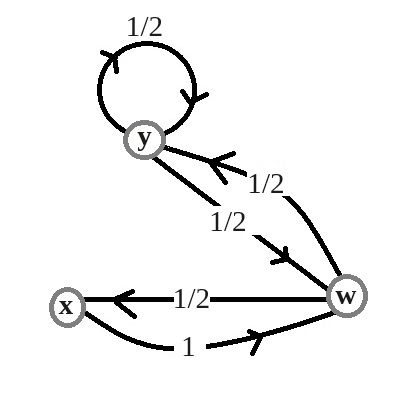}
		\caption{ Representation of the Markov chain $(X_0,p)$ from example \ref{Example_Markov_Chain_1}. }\label{Fig_example_1}
	\end{figure}
\end{Ex}
\begin{Def}:\textit{First-passage Generating function}\label{Def_Green_function_first_pass}

	Let $(X_0,p)$ be an irreducible discrete Markov chain. For $x,y\in X_0$. We call \textit{first-passage generating function} at $(x,y)$ the function
	$$z\mapsto F_z(x,y)$$
	which is the restricted Green's function $z\mapsto G_z(x,y;X_0\setminus\{y\})$, if $x\neq y$, and is constant equal to $1$ otherwise.
\end{Def}
	We separate the case $x=y$ from $x\neq y$ for we want to consider the first passage to $y$. In the case where $x=y$, then the first passage at the state $y$ is done at time $0$, thus the only path we want to consider is the path of length zero $(y)\in\mathcal{P}h(y,y)$. Hence the associated first-passage generating function must be constant equal to $1$.
\subsection{Properties of Green's functions}\label{subsection_properties_Green}
We state some classical properties verified by Green's functions that will be used in the sequel. Other similar results can be found in \cite{Woess_2000} chap.I section 1.
\begin{Prop}\label{SurMult} 
	Let $(X_0,p)$ be an irreducible Markov chain. Let $z$ be a complex number of modulus small enough or a non-negative real number:
	\begin{itemize}
		\item[i)] $\forall x,y\in X_0,\, G_z(x,y)=F_z(x,y)G_z(y,y)$
		\item[ii)] $\forall x\in X_0,\, G_z(x,y)=\delta_{x}(y)+ z \sum_{w\in X_0} p(x,w)G_z(w,y)$
	\end{itemize}
and in particular, if $| \sum_{w\in X_0} z p(x,w)F_z(w,x)|<1$ then
	\begin{itemize}
		\item[iii)] $\forall x\in X_0, G_z(x,x)=\left(1-\sum_{w\in X_0}z p(x,w)F_z(w,x) \right)^{-1}$.
	\end{itemize}
\end{Prop}
\begin{proof}
For the first point suppose $x\neq y$, the case $x=y$ being immediate. By definition for any non-negative integer $n\geq 0$, we have 
$$p^{(n)}(x,y)=\sum_{\substack{\omega_0,...,\omega_n\\ \omega_0=x,\, \omega_n=y}} p(\omega_0,\omega_1)\cdots p(\omega_{n-1},\omega_n).$$

We partition the set of $(n+1)$-tuple $(\omega_0,...,\omega_n)$ such that $\omega_0=x$ and $\omega_n=y$, with respect to the pre-image of the function $\tau_{x,y}^{(n)}:(\omega_0,...,\omega_n)\mapsto \min\{k: \omega_k=y\}$ and decompose the above sum with respect to this partition. 
 
$$\{(\omega_0,...,\omega_n)\in (X_0)^{n+1} : \omega_0=x ,\, \omega_n=y \} = \bigsqcup_{k=0}^n (\tau_{x,y}^{(n)})^{-1}(\{k\}),$$
with, for any $k\in\{0,...,n\}$, 
$$(\tau_{x,y}^{(n)})^{-1}(\{k\})=\{(\omega_0,...,\omega_n): \omega_0= x , \, \omega_n=\omega_k=y \text{ and } \omega_0,...,\omega_{k-1}\neq y \}.$$

Thus
$$p^{(n)}(x,y)=\sum_{k=0}^n p^{(k)}(x,y;\{y\}^\complement) p^{(n-k)}(y,y).$$
Summing over all non-negative integers $n\geq 0$, we get the first point $i)$.

For $ii)$, using the Chapman-Kolmogorov equation (\ref{Kolmo-Chap_eq}) we obtain for any positive integer $n\geq 0$,
$$p^{(n+1)}(x,y)=\sum_{w\in X_0} p(x,w)\times \left(\sum_{\substack{\omega_2,...,\omega_n\\ \omega_{n+1}=y}}p(w,\omega_2)p(\omega_2,\omega_3)\cdots p(\omega_{n},\omega_{n+1})\right).$$
Multiplying the above equality by $r^{n+1}$ and summing over all non-negative integer, we get for any non-negative real number $r\geq 0$:
$$G_z(x,y)=p(x,y) + \sum_{w\in X_0} zp(x,w) G_z(w,y).$$

Lastly for $x=y$ in $ii)$, replacing $G_z(w,x)$ by formula $i)$, in the previous sum, we get $iii)$.
\end{proof}

\begin{Rem}
	Consequently, the Green's function $G_z(x,y)$ are algebraic functions of $z$ and $(F_z(a,b))_{a,b}$. 
\end{Rem}
\begin{Def}\textit{Markov Operator}

	Let $(X_0,p)$ be an irreducible Markov chain. The \textit{Markov operator} associated with the Markov chain $(X_0,p)$ is the kernel operator, with kernel the transition kernel $p$. That is the operator denoted $P$ that associates to any function $f:X_0\to \mathbb{C}$ the function $Pf:X_0\to\mathbb{C}$ defined by:
	$$Pf:x\mapsto \sum_{y\in X_0} p(x,y) f(y).$$
\end{Def}

\begin{Cor}\label{Why_the_name_Greens_function}

Let $(X_0,p)$ be an irreducible Markov Chain and let $P$ be its associated Markov operator.

For any $y\in X_0$ and any complex number $z\in\mathbb{D}(0,R)$, the function $G_z^y:x\mapsto G_z(x,y)$ satisfies,
$$(I-zP)G_z^y = \delta_y.$$
i.e The Green's function is the impulse response\footnote{Actually, the general definition of a Green's function is to be the impulse response of some linear operator, it just happens here that the generating function is such impulse response.} of the linear operator $(I-zP)$.
\end{Cor}
\begin{proof}
	It is an immediate consequence of the point $ii)$ of the above Proposition \ref{SurMult}.
\end{proof}
	
\begin{Prop}\label{PROP_EQT}
Let $(X_0,p)$ be an irreducible Markov chain. Let $g:X_0\to X_0$ be a bijective map preserving the transition kernel $p:X_0\times X_0\to[0,1]$, i.e for any pair $(x,y)$ in $X_0\times X_0$, we have 
\begin{equation}\label{EQTT}
p(gx,gy)=p(x,y).
\end{equation}

Then for any non-negative integer $n\geq 0$, and any pair $(x,y)\in X_0\times X_0$, we have
\begin{equation}\label{EQTTT}
 p^{(n)}(gx,gy)=p^{(n)}(x,y).
\end{equation}

And in particular we have for any pair $(x,y)\in X_0\times X_0$, and any complex number $z$ in $\mathbb{D}(0,R)$, with $R$ the radius of convergence of the Green's function, 
\begin{equation}\label{EQTTTT}
G_z(gx,gy)=G_z(x,y).
\end{equation}
\end{Prop}
\begin{proof}
	Since $g$ is bijective, by definition of $p^{(n)}$ (Definition \ref{DEF:TRANS_KERNEL_n} for any non-negative integer $n\geq 0$, and any pair $(x,y)\in X_0\times X_0$, we have
	$$p^{(n)}(gx,gy)=\sum_{\substack{\omega_0,...,\omega_n\in g^{-1} X_0\\ \omega_0=x,\, \omega_n=y }}p(g\omega_0,g\omega_1)\cdots p(g\omega_{n-1},g\omega_n).$$
	Then using (\ref{EQTT}), we get (\ref{EQTTT}). 
	
	The formula (\ref{EQTTTT}) immediately follows.
\end{proof}
\subsection{Periodicity Cocycle}

We now define the periodicity cocycle associated with some irreducible Markov chain $(X_0,p)$.
	
	\begin{Prop}\label{period_is_well_defined}~
	Let $(X_0,p)$  be an irreducible discrete Markov chain. For any $x,y\in X$, $$\operatorname{gcd}\{ n : p^{(n)}(x,x)>0\}=\operatorname{gcd}\{ n : p^{(n)}(y,y)>0\}.$$
\end{Prop}
\begin{proof}
	Firstly note that for all $x\in X$, the set $\{ n : p^{(n)}(x,x)>0\}$ is non-trivial: Let $y$ be in $X_0$, $y\neq x$. By irreducibility of $(X_0,p)$, there exists $n_1,n_2$ two non-negative integers such that $p^{(n_1)}(x,y)>0$ and $p^{(n_2)}(y,x)>0$. Thus for any non-negative integer $n\geq 0$, using Kolmogorov-Chapman equation \ref{Kolmo-Chap_eq} we have that 
\begin{equation}\label{application_path_temp}
 p^{(n)}(y,y)>0 \Rightarrow p^{(n_1+n+n_2)}(x,x)>0 .
 \end{equation}
 
	Next, note that the set $\{ n: p^{(n)}(y,y)>0\}$ is additive. Indeed let $n,n'$ be integers such that $p^{(n)}(y,y)>0$ and $p^{(n')}(y,y)>0$, then $p^{(n+n')}(y,y)\geq p^{(n)}(y,y) p^{(n')}(y,y)>0$.

	Therefore if $d_y$ denotes the greatest common divisor of the set $\{ n : p^{(n)}(y,y)>0\}$, then there exists a non-negative integer $n_y\in\mathbb{N}$ such that for all $n'\geq n_y$,
	\begin{equation}\label{appartenance_temp}
		d_y n'\in\{ n : p^{(n)}(y,y)>0\}.
	\end{equation}

	Finally, let $x,y$ be arbitrary two vertices in $X_0$, and as before, let $n_1,n_2$ be two non-negative integers such that $p^{(n_1)}(x,y)>0$ and $p^{(n_2)}(y,x)>0$. We write $d_x:=\operatorname{gcd}\{ n : p^{(n)}(x,x)>0\}$ and $ d_y:=\operatorname{gcd}\{ n : p^{(n)}(y,y)>0\}$. By symmetry of roles, to obtain the proposition, we only need to show that $$d_x\leq d_y.$$
	Now, via (\ref{application_path_temp}), for any non-negative integer $n$ such that $p^{(n)}(y,y)>0$, we have $p^{(n_1+n+n_2)}(x,x)>0$. 
	By (\ref{appartenance_temp}), it follows that for all $n'\geq n_y$, $$n_1+n_2+d_y n'\in\{ n : p^{(n)}(x,x)>0\}\subset d_x\mathbb{N}.$$
	We deduce that $d_x$ divides $d_y$, and therefore $d_x\leq d_y$. This is what we wanted.
\end{proof}		

\begin{Prop}\label{Existence_of_a_period_cocycle}
	Let $(X_0,p)$ be an irreducible Markov chain and let $d$ be its period. There exists a map $r_0:X_0\to \mathbb{Z}/d\mathbb{Z}$, such that for any $x,y$ in $X_0$ 
	\begin{equation}\label{eq-t20}
		p(x,y)>0 \implies r_0(y)=r_0(x)+1.
	\end{equation}
\end{Prop}
\begin{proof}
	We fix $x_0\in X_0$ arbitrary and set $r_0(x_0)=0 \, [d]$. Then for $x\in X_0$, we set $r_0(x)=\bar{m}\in\mathbb{Z}/d\mathbb{Z}$ where $m$ belongs to the set $\{n: p^{(n)}(x_0,x)>0\}$ is arbitrary. We now prove that $r$ is well-defined: Let $m,m'$ be such that $p^{(m)}(x_0,x)>0$ and $p^{(m')}(x_0,x)>0$, and let $m''$ be such that  $p^{(m'')}(x,x_0)>0$, then we have that $m+m''$ and $m'+m''$ are in the set $\{n ; p^{(n)}(x_0,x_0)>0 \}$ which is a subset of $d\mathbb{N}$.
	
Thus:
	\begin{equation*}
		m+m''=m'+m'' \,[d].
	\end{equation*}
	which implies that $$m =  m' \,[d].$$
	Hence $r_0$ is well-defined. It also satisfies (\ref{eq-t20}) since, if $x,y\in X$ verify $p(x,y)>0$ then if $p^{(n)}(x_0,x)>0$, from the Chapman-Kolmogorov equation (\ref{Kolmo-Chap_eq}) with $m=1$, we obtain $$p^{(n+1)}(x_0,y)\geq p^{(n)}(x_0,x)p(x,y)>0.$$ Thus $r_0(y)=n+1\, [d]$ i.e $r_0(y)=r_0(x)+1\,[d]$ as desired. 
\end{proof}

\begin{Def}\textit{The periodicity cocycle}\label{DEF:PERIODICITY_COCYCLE}\index{Periodicity cocycle}

	With the notation of the above Proposition \ref{Existence_of_a_period_cocycle}, the $(\mathbb{Z}/d\mathbb{Z})$-valued map,  $$r:(a,b)\mapsto r_0(b)-r_0(a),$$ is an additive cocycle, that does not depend on the choice of $x_0$ taken in the proof of Proposition \ref{Existence_of_a_period_cocycle}. This cocycle is called the \textit{Periodicity Cocycle} associated with the irreducible Markov chain $(X_0,p)$.
\end{Def}
		\begin{Cor}\label{Cor_Admissible_Path_Lenght_without_admissible_paths}
		Let $(X_0,p)$ be an irreducible discrete Markov chain with period $d\geq 1$. Denote by $r$ the periodicity cocycle of $(X_0,p)$. For any pair $(a,b)\in X_0\times X_0$, and any subset $\Omega$ of $X_0$ (In particular for $\Omega=X_0$) we have:
		$$ \forall n \in\mathbb{N}, \, \left( p^{(n)}(a,b;\Omega)>0 \Rightarrow n= r(a,b)\, [d]\right).$$
		
		In particular for any complex number $z$ of modulus small enough, we have 
		$$G_{(e^{2i\pi/d} z)}(a,b;\Omega)=(e^{2i\pi/d})^{r(a,b)}G_z(a,b;\Omega).$$
	\end{Cor}
\section{Irreducible Perron Operators}\label{Irreducible_Perron_Operators}
		\begin{Def}\label{Def_Perron_operator}\textit{Perron Operator / Irreducible Perron Operator}

			Let $E$ be a finite-dimensional real vector space and let $T : E\to E$ be a linear operator on $E$.
			We say that $T$ is a \textit{Perron operator} if there exists a base $\{e_1,...,e_n\}$ of $E$ such that the matrix representation of $T$, with respect to $\{e_1,...,e_n\}$, is a real matrix with non-negative coefficients.

			We will say that $T$ is an \textit{irreducible Perron operator} if there exists a base $\{e_1,...,e_n\}$ of $E$ and an integer $n\in\mathbb{N}$ such that the matrix representation, with respect to the base $\{e_1,...,e_n\}$, of the linear operator: 
			$$\sum_{k=0}^n T^{\circ k} $$
			has strictly positive coefficients.
		\end{Def}
		\begin{Rem}
			If we consider a $\mathbb{C}$-vector space $E$ and a linear operator $T : E\to E$. we will say that the operator $T$ is a Perron operator (or irreducible Perron operator, if it is appropriate), if there exists a basis $\mathcal{B}$ of $E$, whose generated $\mathbb{R}$-vector space is stable by $T$ :
			$$T\left( \operatorname{Vect}_\mathbb{R}(\mathcal{B})\right)\subset \operatorname{Vect}_\mathbb{R}(\mathcal{B}),$$
			and if $T$ in restriction to this vector space is Perron (respectively Perron irreducible).
		\end{Rem}

		\begin{Lem}\textit{de Perron (1907)}\label{Lem_de_Perron} \cite{Perron_1907}
		
			Let $E$ be a finite-dimensional vector space and $T : E\to E$ be a Perron operator. Let $\{e_1,...,e_n\}$ be a basis of $E$ in which the matrix representation of $T$ has non-negative coefficients. Then the spectral radius $\rho(T)$ of $T$ is an eigenvalue of $T$ and there exists an associated eigenvector in $E$ whose representation in the base $\{e_1,...,e_n\}$ of $E$ has non-negative coefficients.
		\end{Lem}
        
		\begin{Thm}\label{Thm_Perron_irred}\textit{Perron irreducible theorem}(1912) [See \cite{Seneta_2006} theorem 1.1]
		
			Let $E$ be a finite-dimensional vector space and $T : E\to E$ an irreducible Perron operator. Let $\{e_1,...,e_n\}$ be a basis of $E$ appearing in the definition \ref{Def_Perron_operator} of an irreducible Perron operator, associated with $T$.
			The spectral radius $\rho(T)$ of $T$ is non-zero and is a simple eigenvalue of $T$ for a certain vector $\nu$ in $E$, whose representation in the base $\{e_1,...,e_n\}$ has strictly positive coefficients.
			
			Moreover, its associated generalized eigenspace is exactly $\mathbb{R}\cdot \nu$ and any other eigenvalue of $T$ is of strictly lower modulus than $\rho(T)$. 
			And finally if $C$ denotes the convex cone $\operatorname{Vect}_{\mathbb{R}+}(e_1,...,e_n)$ then
			\begin{equation}\label{eq_cone_perron}
			C\cap(T-\rho(T)Id)(E)=\{0\}.
			\end{equation}
		\end{Thm}
		\begin{Rem}
		    The above equality (\ref{eq_cone_perron}) is equivalent as stating that the range of the operator $(I-\frac{1}{\rho(T)}T)$ does not contain non-zero vector in $C$.
		\end{Rem}
		\begin{Thm}[ \cite{Seneta_2006} Theorem 1.2]\label{Seneta_thm1.2}
Let $T\in\mathcal{M}_n(\mathbb{C})$ be a \textit{primitive} matrix (that is a matrix such that for some integer $j$, the matrix $T^j$ has strictly positive coefficients). Denote by $\pi:\mathbb{C}^n\to E_T(\rho(T))$, a projection onto the characteristic subspace $E_T(\rho(T))$ associated with the simple eigenvalue $\rho(T)$ of $T$. And denote by $\lambda_2$ the eigenvalue of $T$ such that its modulus satisfies $|\lambda_2|=\max\{|l|: l\in Spec(T), l\neq \rho(T) \}$. Then :
\begin{enumerate}[label=\roman*)]
 	\item If $\lambda_2\neq 0$, then as $k\to\infty$, we have
	$$T^k=c \rho(T)^k\pi + O\left(k^{m_2-1}|\lambda_2|^k\right),$$
	where $m_2$ is the multiplicity of the eigenvalue $\lambda_2$, and $c$ is some positive constant.
	\item If $\lambda_2=0$, then for any $k\geq n-1$, we have
	$$T^k=c\rho(T)^k\pi,$$
    with $c$ a positive constant.
\end{enumerate}
\end{Thm}

\bibliographystyle{alpha}
\bibliography{sample}

Institut Mathématiques de Bordeaux/ Institut Montpellierain Alexander Grothendieck\\
\indent E-MAIL: gchevalier@protonmail.com
    
\end{document}